\documentclass{article}
\usepackage{amsthm,amsmath,amssymb}
\RequirePackage[numbers]{natbib}

\RequirePackage[colorlinks,citecolor=blue,urlcolor=blue]{hyperref}

\usepackage{graphicx}
\usepackage{subcaption}
\RequirePackage[OT1]{fontenc}
\RequirePackage[numbers]{natbib}
\RequirePackage[colorlinks,citecolor=blue,urlcolor=blue]{hyperref}

\usepackage{amsfonts, amsmath, amssymb, amsthm}
\usepackage{dsfont}
\usepackage{mathrsfs}
\usepackage{verbatim}
\usepackage{graphicx}
\usepackage{authblk}
\usepackage{url}
\usepackage{blindtext}
\usepackage{color}
\usepackage{enumitem}

\usepackage[top=1in,bottom=1in,left=1in,right=1in]{geometry}
\usepackage{xcolor}
\definecolor{darkred}{RGB}{100,0,0}
\definecolor{darkgreen}{RGB}{0,100,0}
\definecolor{darkblue}{RGB}{0,0,150}

\usepackage{hyperref}
\hypersetup{colorlinks=true, linkcolor=darkred, citecolor=darkgreen, urlcolor=darkblue}
\usepackage{url}

\newtheorem{thm}{Theorem}
\newtheorem{prp}{Proposition}
\newtheorem{lem}{Lemma}
\newtheorem{cor}{Corollary}

\newtheorem{remark}{Remark}

\def\beq{\begin{equation}}
\def\eeq{\end{equation}}
\def\beqn{\begin{eqnarray*}}
\def\eeqn{\end{eqnarray*}}
\def\bitem{\begin{itemize}}
\def\eitem{\end{itemize}}
\def\benum{\begin{enumerate}}
\def\eenum{\end{enumerate}}
\def\bmult{\begin{multline*}}
\def\emult{\end{multline*}}
\def\bcenter{\begin{center}}
\def\ecenter{\end{center}}

\DeclareMathOperator*{\argmax}{arg\, max}
\DeclareMathOperator*{\argmin}{arg\, min}

\def\cC{\mathcal{C}}

\def\cE{\mathcal{E}}
\def\cF{\mathcal{F}}

\def\cK{\mathcal{K}}
\def\cL{\mathcal{L}}

\def\cP{\mathcal{P}}
\def\cQ{\mathcal{Q}}

\def\cT{\mathcal{T}}

\def\cZ{\mathcal{Z}}

\def\bA{\boldsymbol{A}}
\def\bB{\boldsymbol{B}}
\def\bC{\boldsymbol{C}}

\def\bM{\boldsymbol{M}}

\def\bO{\boldsymbol{O}}

\def\bQ{\boldsymbol{Q}}
\def\bR{\boldsymbol{R}}
\def\bS{\boldsymbol{S}}
\def\bT{\boldsymbol{T}}
\def\bU{\boldsymbol{U}}
\def\bV{\boldsymbol{V}}

\def\bX{\boldsymbol{X}}
\def\bY{\boldsymbol{Y}}
\def\bZ{\boldsymbol{Z}}

\def\bb{\mathbf{b}}

\newcommand{\bepsilon}{{\boldsymbol\epsilon}}

\def\bTheta{\boldsymbol{\Theta}}
\newcommand\bSigma{{\boldsymbol\Sigma}}

\newcommand{\bPi}{{\boldsymbol\Pi}}

\def\bb\bU{\mathbb{\bU}}

\def\\bUnif{\text{\bUnif}}

\DeclareMathOperator{\sign}{sign}

\usepackage{ulem}
\begin{document}

\providecommand{\keywords}[1]
{
  \small	
  \textbf{\textit{Keywords---}} #1
}

\title{Maximum Likelihood Estimation of Sparse Networks with Missing Observations}
\author[1]{Solenne Gaucher \thanks{solenne.gaucher@ensae.fr}}
\author[2,1]{Olga Klopp \thanks{kloppolga@math.cnrs.fr}}
\affil[1]{CREST, ENSAE}
\affil[2]{ESSEC Business School}
\maketitle

\begin{abstract}
Estimating the matrix of connections probabilities is one of the key questions when studying sparse networks. In this work, we consider networks generated under the sparse graphon model and the inhomogeneous random graph model with  missing observations. Using the Stochastic Block Model as a parametric proxy, we bound the risk of the maximum likelihood estimator of network connections probabilities, and show that it is minimax optimal. Moreover, we show that our estimator can be efficiently approximated using tractable variational methods, and thus used in practice.
\end{abstract}


\section{Introduction}\label{section:intro}

In the past two decades, networks have attracted considerable attention, as many scientific fields are concerned by the advances made in the understanding of these complex systems. In social sciences \cite{WassermanSocio} as in physics \cite{RevModPhys.74.47} and biology \cite{yamanishi:hal-00433586}, networks are used to represent a great variety of systems of interactions between social agents, particles, proteins or neurons. These networks are often modeled as an observation drawn from a random graph. 

Missing observations is a common problem when studying real life networks. In social sciences, data coming from sample surveys are likely to be incomplete, especially, when dealing with large or hard-to-find populations. While biologists often use graphs to model interactions between proteins, experimental discovery of these interactions can require substantial time and investment from the scientific community \cite{BleakleyMissingProteins}. In many cases, collecting complete information on relations between actors can be difficult, expensive and time-consuming \cite{Kshirsagar2012TechniquesTC,Yan2012FindingME,handcock2010, GuimerMissing}. On the other hand, the emergence of detailed data sets coming, for example, from social networks or genome sequencing has fostered new challenges, as their large size makes using the full data computationally unattractive. This has lead scientists to consider only sub-samples of the available data \cite{Benyahia2017CommunityDI}. However, incomplete observation of the network structure may considerably affect the accuracy of inference methods \cite{KOSSINETS2006247}.

Our work focuses on the study of the inhomogeneous random graph model with \textit{missing observations}. In this setting, the problem of estimating the matrix of connections probabilities is of primary interest. Minimax optimal convergence rates for this problem have been shown to be attained by the least square estimator under full observation of the network for dense graphs in \cite{gao2015optimal} and for sparse graphs in \cite{KloppGraphon}. In \cite{2015gaoBiclustering}, the authors extended these results to the setting in which observations about the presence or absence of an edge are missing independently at random with the same probability $p$.  However their estimator requires the knowledge of $p$, and cannot be extended to non-uniform sampling schemes. Unfortunately, least square estimation is too costly to be used in practice. Many other approaches have been proposed, for example, spectral clustering \cite{SpecClusMcSherry, HagenSpecClus, rohe2011}, modularity maximization \cite{Newman8577, Bickel21068}, belief propagation \cite{Lenka}, neighborhood smoothing \cite{Zhang2017EstimatingNE}, convex relaxation of k-means clustering \cite{Giraud2018PartialRB} and of likelihood maximization \cite{amini2018}, and universal singular value thresholding \cite{chatterjee2015,KloppCut,USVTXu}. It was conjectured the existence of possible computational gap when no polynomial time algorithm can achieve minimax optimal rate of convergence. The present work shows that this is not the case.

In this work, we consider the maximum likelihood estimator. This estimator is also NP-hard but its computationally efficient approximations (under some additional conditions) have been proposed in the literature (see, e.g., \cite{refId0} for a detailed review of these methods). For example, the authors of \cite{Bickl} suggest to use pseudo-likelihood methods, as it leads to computationally tractable estimators. Alternatively, in \cite{VarEst} a tractable variational approximation of the maximum likelihood estimator is proposed. This methods has been applied successfully to study biological networks, political blogsphere networks and seeds exchange networks \cite{picard:hal-00391483, TabPra,LatoucheBlogsphere}. The authors of \cite{LikelihoodBickel} show asymptotic normality of the maximum likelihood estimate and of its variational approximation for sparse graphs generated by stochastic block models when the connections probabilities of the different communities are well separated. In \cite{TabPra}, these results are extended to the case of missing observations. These methods suffer from a lack of theoretical guarantees when the model is misspecified or non-identifiable. On the other hand, to the best of our knowledge, no non-asymptotic bound has been established for the risk of the maximum likelihood estimator. In this work, we close this gap and show that the maximum likelihood estimator is minimax optimal in a number of scenarii, and adaptative to non-uniform missing data schemes. Moreover we show that it can be efficiently approximated using tractable variational methods.

Our results find a natural application in predicting the existence of non-observed edges, a commonly encountered problem called \textit{link prediction} \cite{Lu2010LinkPI, LevinaLinkPred}. Interaction networks are often incomplete, as detecting interactions can require significant experimental effort. Instead of exhaustively testing for every connection, one might be interested in deducing the pairs of agents which are most likely to interact based on the relations already recorded and on available covariates. If these estimations are precise enough, testing for these interactions would enable scientists to establish the network topology while substantially reducing the costs \cite{Clauset2008HierarchicalSA}. In this context, estimating the probabilities of connections through likelihood maximization enables to accordingly rank unobserved pairs of nodes. Link prediction also finds applications in recommender systems for social networks. The missing observation scheme studied in this work is motivated by the above examples, and generalizes the model described in \cite{2015gaoBiclustering}.

\subsection{Inhomogeneous random graph model}\label{subsection:Models}
 We consider an undirected, unweighted graph with $n$ nodes indexed from $1$ to $n$. Its connectivity can be encoded by its \textit{adjacency matrix} $\bA$, defined as follows:  set $\bA$ a $n \times n$ symmetric matrix such that for any $i<j$, $\bA_{ij} = 1$ if there exists an edge between node $i$ and node $j$, $\bA_{ij} = 0 $ otherwise. In our model, we consider that there is no edge linking a node to itself, so $\bA_{ii}=0$ for any $i$. We assume that the variables $\left( \bA_{ij}\right)_{1 \leq i<j \leq n}$ are independent Bernoulli random variables of parameter $\bTheta^*_{ij} $, where $\bTheta^*$ is a $n \times n$ symmetric matrix with zero diagonal entries. This matrix $\bTheta^*$ corresponds to the matrix of probabilities of observing an edge between nodes $i$ and $j$. This model is known as the \textit{inhomogeneous random graph} model:
 \begin{equation} \label{graphSeq}
\forall 1 \leq i < j \leq n, \ \bA_{ij}\vert \bTheta^*_{ij} \overset{ind.}{\sim} \operatorname{Bernoulli} \left(\bTheta^*_{ij} \right).
 \end{equation}
In the present paper we consider the following problem: from a single partial observation of the graph, that is, given a sample of entries of the adjacency matrix $\bA$, we want to estimate the matrix of connections probabilities $\bTheta^*$.

The problem of estimating $\bTheta^*$ when some entries of the adjacency matrix are not observed is closely related to the 1-bit matrix completion problem. The matrix completion problem \cite{Candes2009,koltchinskii2011,  negahban2011} aims at recovering a matrix which is only partially observed. More precisely, we observe a random sample of its entries, which may be corrupted by some noise, and we wish to infer the rest of the matrix. In 1-bit matrix completion, first introduced in \cite{Davenport1Bit}, the entries $(i,j)$ of the observed matrix can only take two values $\{0,1\}$ with probabilities given respectively by $f(\bM_{ij})$ and $1-f(\bM_{ij})$. Here, the matrix $\bM$ corresponds to the real quantity of interest that one would like to infer, and the function $f$ can be seen as the cumulative distribution function of the noise. A typical assumption in this setting is that the matrix $\bM$ is low-rank. In \cite{KloppMultinomial}, the authors show that for 1-bit matrix completion the restricted penalized maximum likelihood estimator is minimax optimal up to a $\log$ factor. The methods used in our proofs are, to some extend, inspired by the methods developed for the framework of matrix completion. However, the problem we have in hand is in many aspects different from the 1-bit matrix completion problem. The structure of the connections probabilities matrix $\bTheta^*$ and the sparsity of the network allow for faster rates of convergence, and the technics of proof required to match the minimax optimal convergence rate are more involved.

Our approach for estimating the matrix of connections probabilities is based on the celebrated Regularity Lemma by Szemer\'edi \cite{LovaszBook}, which implies that any graph can be well approximated by a stochastic block model (SBM). We refer to \cite{LovaszBook} for a more detailed presentation of this result. In the SBM, each node $i$ is associated with a community $z^*(i)$, where $z^*: [n] \rightarrow [k]$ is called the index function. This index function can either be treated as a parameter to estimate (this model is sometimes called the conditional stochastic block model), or as a latent variable. In this case, the indexes follow a multinomial distribution: $\forall i$, $z^*(i) \overset{i.i.d}{\sim} \mathcal{M}(1; \alpha^*)$ where $\forall l \in [k]$, $\alpha_l$ is the probability that node $i$ belongs to the community $l$. Given this index function, the probability that there exists an edge between nodes $i$ and $j$ depends only on the communities of $i$ and $j$. For example, when considering citations networks, where two articles are linked if one is cited by the other, it amounts to saying that the probability that two articles are linked only depends on their topic. Similarly, if one considers students of a school in a social network, it is a reasonable assumption to say that the probability that two students are linked only depends on their cohorts. This implies that the matrix of connections probabilities $\bTheta^*$ can be factorized as follows: $\bTheta^*_{ij} = \bQ^*_{z^*(i) z^*(j)}$, with $\bQ^*$ a $k \times k$ symmetric matrix such that $\bQ^*_{ab}$ is the probability that there exists an edge between a given member of the community $a$ and a given member of the community $b$, so we have that the conditional SBM can be written as:
 \begin{equation} \label{blockmodel}
 \begin{split}
\exists \bQ^*\in[0,1]^{k \times k}_{\rm sym},&\ \exists z^*: [n] \rightarrow [k]\\
\ \forall 1 \leq i < j \leq n, \ \bA_{ij}\vert \bQ^*, z^* &\overset{ind.}{\sim} \operatorname{Bernoulli} \left(\bQ^*_{z^*(i)z^*(j)} \right),\ \bA_{ii} = 0.
\end{split}
 \end{equation}
While considering the SBM, the problem of estimating the matrix of connections probabilities reduces to estimating the label function $z^*$ and the matrix of probabilities of connections between communities $\bQ^*$.

In the past decade, the stochastic block model has known a growing interest from the statistical community and an important part of the work has focused on the problem of community recovery (i.e., the recovery of the vector of communities populations $\alpha^*$, or of the label function $z^*$ in the conditional model). Theoretical guarantees for this problem were established  under quite strong assumptions on the matrix of probabilities of connections between communities, $\bQ^*$, see, for example, \cite{mass2013, bordenave2018, Abbe2015CommunityDI, mossel2014consistency}.

Note that our results hold without assuming the existence of the true community structure, that is, without assuming that the matrix $\bTheta^*$ is block constant. With this in mind, we will focus on estimating the distribution giving rise to the adjacency matrix, i.e., on estimating $\bTheta^*$, rather than on estimating the label function or the populations of the communities. One important question in this setting is how to choose the number of communities for our estimator, as more communities implies a smaller bias and a greater variance. Optimizing this trade-off requires, first, establishing a non-asymptotic bound on the risk of our estimator for a number of communities that may depend on the number of nodes, and, in a second time, bounding the bias of an oracle block constant estimator.

Our work focuses on relevant in applications setting of partial observations of the network. We consider the following missing value setting. Let $\bX \in \{0,1\}^{n \times n}_{sym}$ denote the sampling matrix given by $\bX_{ij} = 1$ if we observe $\bA_{ij}$ and $\bX_{ij} = 0$ otherwise. We assume that the sampling matrix $\bX$ is random and, conditionally on $\bTheta^*$, independent from the adjacency matrix $\bA$. For any $1\leq i<j\leq n$, its entries $\bX_{ij}$ are mutually independent. Finally, we denote by $\bPi \in [0,1]^{n \times n}_{sym}$ the unknown matrix of sampling probabilities such that $\bX_{ij} \overset{ind.}{\sim} \text{Bernoulli}(\bPi_{ij})$. This sampling scheme includes for instance node-based sampling schemes such as the exo-centered design described in \cite{handcock2010}, where we observe $\bA_{ij}$ if $i$ or $j$ belongs to the set of sampled nodes. It also covers random dyad sampling schemes (described, e.g., in \cite{TabPra}). In this case, the probability of observing the entry $\bA_{ij}$ is allowed to depend on the communities of $i$ and $j$. 

\subsection{Graphon model}
While studying exchangeable random graphs, important questions such as how to compare two graphs with different numbers of nodes or how to study graphs with an increasing number of nodes call for a more general, non-parametric model. One of such models that has attracted a lot of attention recently is the \textit{graphon} model \cite{Olhede14722, gao2015optimal, KloppGraphon, USVTXu}. In this model, the connections probabilities $\bTheta^*_{ij}$ are the following random variables  
\begin{equation}
\label{dense_graphon}
\bTheta_{ij}^* = W^*(\zeta_i, \zeta_j)
\end{equation}
where $\zeta_1, ..., \zeta_n$ are unobserved (latent) independent random variables sampled uniformly in $[0,1]$. The graph is then sampled according to the inhomogeneous random graph model \eqref{graphSeq}. The function $W^*: [0,1]^2 \rightarrow [0,1]$ is measurable, symmetric and is called a graphon. Graphs encountered in practice are usually $\textit{sparse}$: the expected number of edges grows as $\rho_n n^2$ where $\rho_n$ is a decreasing sequence of sparsity inducing parameters. The dense graphon model can be modified in order to account for this sparsity:
\begin{equation}
\label{sparse_graphon}
\bTheta_{ij}^* = \rho_nW^*(\zeta_i, \zeta_j).
\end{equation}
Since the law of the graph is invariant under any change of labelling of its nodes, different graphons can give rise to the same distribution on the space of graphs of size $n$. More precisely, let $W$ be a graphon and $\tau: [0,1] \rightarrow [0,1]$ be a measure-preserving function. We write $W_{\tau}\left(x,y \right) = W(\tau(x),\tau(y))$ and say that two graphons $U$ and $V$ are \textit{weakly isomorphic} if there exists measure-preserving maps $\tau$ , $\phi $ such that $U_{\tau} = W_{\phi}$ almost everywhere. It is established in Section $10$, \cite{LovaszBook} that two graphons define the same probability measure on graphs if and only if they are weakly isomorphic. 

In the present paper we also consider the setting when the matrix of connections probabilities is generated following the sparse graphon model \eqref{sparse_graphon}. We deal with two classes of graphon functions previously studied in the literature, step-function graphons and smooth graphons, under the scenario  of partial observations of the network.

\subsection{Outline of the paper}
  
The present paper is devoted to the theoretical study of the maximum likelihood estimator in sparse network models with missing observations. First, we provide an oracle bound for the risk of the maximum likelihood estimator of the matrix of connections probabilities from a partial observation of the adjacency matrix $\bA$. Our results hold under fairly general assumptions on the missing observations scheme and we show that the maximum likelihood estimator matches the minimax optimal rates of convergence in a variety of scenarii, while being fully adaptative to the missing data scheme. Second, we provide a parameter-free version of our estimator which, in particular, does not require the knowledge of the sparsity parameter $\rho_n$. We also bound the Kullback-Leibler divergence between the true matrix of connections probabilities and its block constant approximation, and derive an optimal choice for the number of communities defining the maximum likelihood estimator. We show that the maximum likelihood estimator can be approximated using tractable variational methods, and we provide a bound on the risk of the variational estimator.

This manuscript is organized as follows. In Section \ref{subsection:MLEDescription}, we introduce the maximum likelihood estimator for the matrix of connections probabilities $\bTheta^*$ from partial observation of the adjacency matrix $\bA$. Then,   Theorem \ref{thm_matrix_oracle} in Section \ref{subsection:Results} provides a non-asymptotic oracle bound on the risk of this estimator. As a consequence, we  show that our estimator is minimax optimal in a number of scenarii and derive the corresponding bound for estimating $\bTheta^*$ in the case of full observation of the adjacency matrix $\bA$. Our estimation method requires bounds on the entries of $\bTheta^*$. In Section \ref{subsection:unknowngamma}, we first  propose a  method to choose these bounds under fairly general assumptions and, in Section \ref{subsection:graphon}, we specify it to the case of sparse graphon model \eqref{sparse_graphon}. We show that the resulting adaptative estimator is minimax optimal up to a log factor. In Section \ref{subsection:MLEGraphon}, Theorem \ref{thm:MySupGraphon}, we provide the choice for the number of communities that achieves the best trade off between the variability of our estimate and the fit of the oracle model. Finally, in Section \ref{Section:Variational}, we show that the maximum likelihood estimator can be consistently approximated using a tractable variational estimator.

 \subsection{Notations}
 
 We provide here a summary of the notations used throughout this paper.work.

\begin{itemize}
\item For any positive integer $d$, we denote by $[d]$ the set $\{1,...,d\}$.
\item For any set $\mathcal{S}$, we denote by $\vert \mathcal{S} \vert$ its cardinality.
\item For any matrix $\bA$, we denote by $\bA_{ij}$ its entry on row $i$ and column $j$. If $\bA \in [0,1]^{n \times n}$ and $\bA$ is symmetric, we write $\bA \in [0,1]^{n \times n}_{\rm sym}$.
\item Let $\cK(q,q')=q\log\left (\frac{q}{q'}\right )+(1-q)\log\left (\frac{1-q}{1-q'}\right)$ denote the Kullback-Leibler divergence of a Bernoulli distribution with parameter $q$ from a Bernoulli distribution with parameter $q'$. For any three symmetric matrices with zero diagonal entries $\bA$,  $\bB$, $\bX \in [0,1]^{n \times n}_{sym}$ we set
\begin{equation*}
\cK_{\bX}(\bA,\bB) =\underset{i<j}{\sum}\bX_{ij}\cK(\bA_{ij},\bB_{ij})\  \text{ and }\ \cK(\bA,\bB)=\underset{i<j}{\sum}\cK(\bA_{ij},\bB_{ij}).
\end{equation*}
    
\item For any three symmetric matrices with zero diagonal entries $\bA$,  $\bB$, $\bC \in \mathbb{R}^{n \times n}_{sym}$, let $\left\langle \bA \vert \bB \right\rangle = \underset{i<j}{\sum} \bA_{ij}\bB_{ij}$, $\left\langle \bA \vert \bB \right\rangle_{\bC} = \underset{i<j}{\sum} \bC_{ij}\bA_{ij}\bB_{ij}$, $\left\Vert \bA\right\Vert_2 = \sqrt{\left\langle \bA \vert \bA \right\rangle}$,  $\left\Vert \bA\right\Vert_{2,\bC} = \sqrt{\left\langle \bA \vert \bA \right\rangle_{\bC}}$, and $\left\Vert \bA\right\Vert_{\infty} = \underset{i,j}{\max} \vert \bA_{ij}\vert$. With these notations, $\left\Vert \bA\right\Vert_{2,\bPi} = \sqrt{\mathbb{E}\left[\left\Vert \bA\right\Vert_{2,\bX}^2\right]}$ corresponds to the $L_2$-norm of the matrix $\bA$ with respect to the sampling probabilities $\bPi$.

\item We denote by $\mathcal{Z}_{n,k}$  the label functions $z: [n] \rightarrow [k]$. For any $z \in \mathcal{Z}_{n,k}$, we denote by $\mathcal{T}_z$ the set of block constant matrices corresponding to the label $z$: 

$\mathcal{T}_z \triangleq \left \{\bA: \forall i \in [n], \bA_{ii} = 0 \ \& \ \exists \bQ \in [0,1]^{k \times k}, \forall 1 \leq i<j\leq n, \bA_{ij} = \bA_{ji} = \bQ_{z(i) z(j)} \right \}.$ 

\item To ease notations, for $\bA \in \mathcal{T}_z$ and $(a, b)\in [k]^2$, we sometimes denote by  $\bA_{z^{-1}(a) z^{-1}(b)}$ any entry $\bA_{ij}$ such that $(i,j) \in (z^{-1}(a), z^{-1}(b))$ and $i \neq j$. We write $\cT_k = \underset{z \in \cZ_{n,k}}{\cup} \cT_z$.

\item We denote by $C$ and $C'$ positive constants that can vary from line to line. These are absolute constants unless otherwise mentioned. For any two positive sequences $\left(a_n\right)_{n \in \mathbb{N}}$, $\left(b_n\right)_{n \in \mathbb{N}}$, we write $a_n = \omega(b_n)$ if $a_n/b_n \rightarrow \infty$.
\item We denote respectively by $\mathbb{E}^{\bX}$ and $\mathbb{P}^{\bX}$ the expectation and the probability conditionally on the random variable $\bX$, and respectively by $\mathbb{E}$ and $\mathbb{P}$ the expectation and the probability over all random variables.
\end{itemize}

\section{Convergence rate for the maximum likelihood estimator}
\label{section:Results}

\subsection{Maximum likelihood estimator under missing observations}\label{subsection:MLEDescription}

 We start by introducing the conditional log-likelihood for the model \eqref{graphSeq}. Conditionally on the probability matrix $\bTheta^*$, the entries $(\bA_{ij})_{1 \leq i<j\leq n}$ of the adjacency matrix $\bA$ are independent Bernoulli variables with parameters $(\bTheta^*_{ij})_{1 \leq i<j\leq n}$. Therefore, for any $\bTheta \in [0,1]^{n\times n}$, the conditional log-likelihood of the parameter matrix $\bTheta$ with respect to the observed entries of the adjacency matrix $\bA$ is given by
  \begin{align*}
\cL_{\bX}(\bA;\bTheta)&=\sum_{i<j}\bX_{ij}\left (\bA_{ij}\log(\bTheta_{ij})+(1-\bA_{ij})\log(1-\bTheta_{ij})\right).
   \end{align*}
For any $z\in \cZ_{n,k}$ and $\bQ\in [0,1]^{k\times k}_{\mbox{sym}}$, the matrix of connections probabilities corresponding to the block model $(z,\bQ)$ is given by $\bTheta_{ij} = \bQ_{z(i)z(j)}$ for $1\leq i < j \leq n$ and $\bTheta_{ii} =0$ for $i \in [n]$. With these notations, the conditional log-likelihood of a block model $(z,\bQ)$  with respect to the observed entries of the adjacency matrix $\bA$ is  
 \begin{align*}
\cL_{\bX}(\bA;z,\bQ)&=\sum_{1 \leq i<j\leq n}\bX_{ij}\left(\bA_{ij}\log(\bQ_{z(i)z(j)})+(1-\bA_{ij})\log(1-\bQ_{z(i)z(j)})\right)\\
 &=\sum_{1 \leq a\leq b \leq k}\sum _{\underset{i \neq j}{i\in z^{-1}(a),\ j\in z^{-1}(b)}}\bX_{ij}\left(\bA_{ij}\log(\bQ_{ab})+(1-\bA_{ij})\log(1-\bQ_{ab})\right)\\
 &=\sum_{a\leq b}\log(\bQ_{ab})\sum _{\underset{i \neq j}{i\in z^{-1}(a),\ j\in z^{-1}(b)}}\bX_{ij}\bA_{ij}+\sum_{a\leq b}\log(1-\bQ_{ab})\sum _{\underset{i \neq j}{i\in z^{-1}(a),\ j\in z^{-1}(b)}}\bX_{ij}(1-\bA_{ij}).&
   \end{align*}
The maximum likelihood estimator for the stochastic block model is 
  \begin{align*}
  (\widehat \bQ,\widehat z)\in \underset{\bQ\in [0,1]^{k\times k}_{\rm sym}, z\in \cZ_{n,k}}{\argmax} \cL_{\bX} (\bA;z,\bQ).
    \end{align*}
    The block constant maximum likelihood estimator  of $\bTheta^{*}$ is defined as $\widehat \bTheta_{ij}=\widehat \bQ_{\widehat z(i)\widehat z(j)}$ for any $i<j$. Note that maximizing the log-likelihood is equivalent to minimizing  a sum of Bernoulli Kullback-Leibler divergences. Indeed, an easy calculation leads
\begin{align}\label{def_estimator}
  (\widehat \bQ,\widehat z)\in \underset{\bQ\in [0,1]^{k\times k}_{\rm sym}, z\in \cZ_{n,k}}{\argmax} \cL_{\bX} (\bA;z,\bQ)=\underset{\bQ\in [0,1]^{k\times k}_{sym}, z\in \cZ_{n,k}}{\argmin} \sum_{i<j}\bX_{ij}\cK(\bA_{ij},\bQ_{z(i)z(j)}) .
    \end{align}
 Moreover, for any fixed assignment $z\in \cZ_{n,k}$ and any sampling matrix $\bX$, the log-likelihood with regards to the observed entries of $\bA$ will be maximized by taking 
 \begin{equation*}
 \bQ_{ab}= \overline{\bX\bA}^z_{ab}\triangleq \frac{\sum _{\underset{i \not= j}{i\in z^{-1}(a),j\in z^{-1}(b)}}\bX_{ij}\bA_{ij}}{\sum _{\underset{i \not= j}{i\in z^{-1}(a),j\in z^{-1}(b) }}\bX_{ij}}.
 \end{equation*}
Note that for any label function $z$, maximizing the likelihood or minimizing the least square criterion defined as $\cC_{\bX}\left(\bA; z, \bQ\right) = \underset{i<j}{\sum}\bX_{ij}\left(\bA_{ij} - \bQ_{z(i),z(j)} \right)^2$ with regards to $\bQ$ yields the same estimator. However, the label functions selected by the two criterions can be different, as is shown in figure \ref{fig:diff_ML_LS}.
\begin{figure}[h] 
\centering
  \begin{subfigure}{7cm}
    \centering
    \includegraphics[width=5.5cm]{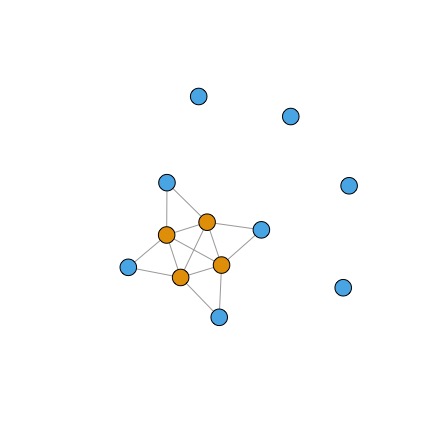}
    \caption{Communities obtained using the least squares criterion.}
  \end{subfigure}
  \hspace{1cm}
  \begin{subfigure}{7cm}
    \centering\includegraphics[width=5.5cm]{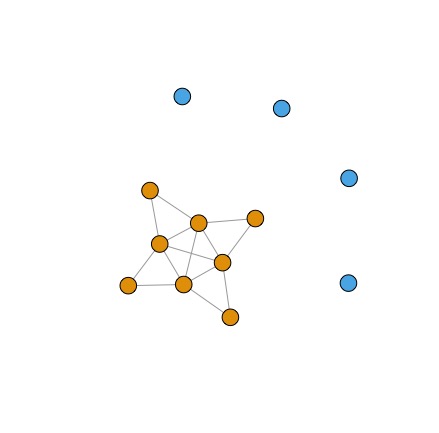}
    \caption{Communities obtained using the maximum likelihood criterion.}
  \end{subfigure}
  \caption{\label{fig:diff_ML_LS}\small We fit a SBM with two communities to the graph above using the maximum likelihood criterion (left) and the least squares criterion (right). Four nodes are classified differently by the two methods.}
\end{figure}

In the rest of this work, we will denote by $\widetilde\bTheta$ the oracle probability matrix, i.e., the best approximation to $\bTheta^*$ in the sense of the weighted Kullback Leibler divergence: 
\begin{equation}\label{eq:barThetaKL}
\begin{split}
&\widetilde\bTheta_{i<j} = \bQ^*_{z^*(i) z^*(j)}, \widetilde\bTheta_{ii}=0\ \\
 &(\bQ^*,z^*)\in \underset{\bQ\in [0,1]^{k\times k}_{\rm sym}, z\in \cZ_{n,k}}{\argmin} \sum_{i<j}\cK_{\bPi}(\bTheta^*_{ij},\bQ_{z(i)z(j)}) .
    \end{split}
\end{equation}

\subsection{Upper bound on the risk of the restricted maximum likelihood estimator}
 \label{subsection:Results}
  In this section, we establish an upper bound on the risk of the maximum likelihood estimator and show that it matches the minimax convergence rate obtained in \cite{KloppGraphon,2015gaoBiclustering}.  We will measure the risk of our estimator in $\bPi$-weighted Frobenius norm. To bound the risk of the maximum likelihood estimator, we assume that there exists sequences $\rho_n$ and $\gamma_n$ such that $\forall i<j$, \begin{equation}\label{cond_bounds}
      0 < \gamma_n \leq \bTheta^*_{ij} \leq \rho_n < 1.
  \end{equation}
  Note that for sparse graphs, $\rho_n$ corresponds to the sparsity inducing sequence in equation \eqref{sparse_graphon}.
  
  The assumption \eqref{cond_bounds} is systematic in the literature studying the maximum likelihood estimator for the stochastic block model, as it guarantees that the loss associated to the maximum likelihood estimator is Lipschitz. See, for example, \cite{LikelihoodBickel} and  \cite{BickelMS}, where the authors assume that the adjacency matrix is generated by an homogeneous stochastic block model for which the matrix $\bQ^*/\rho_n$ has entries bounded away from $0$. In our model, this corresponds to imposing that $\rho_n = O(\gamma_n)$, i.e., that all entries of the matrix of connections probabilities $\bTheta^*$ are of the same order of magnitude. Our assumptions are more general than the one developped in these articles, as they also cover the case $\gamma_n = o(\rho_n)$. 
  
  In \cite{celisse2012}, the authors consider the dense SBM and assume that the entries of $\bQ^*$ belong to $\{0\} \cup [\zeta, 1-\zeta] \cup \{1\}$ for some $\zeta >0$. They prove the consistency of the maximum likelihood estimator constrained to a restricted subset of the parameters. However, the definition of this subset implies knowing the set $\Omega_0 = \{ (i,j) : \bTheta^*_{ij} \in \{0,1\}\}$ prior to estimating the matrix of connections probabilities. Note that, if we assume that $\Omega_0$ is known and that $\bQ^*$ belong to $\{0\} \cup [\zeta, 1-\zeta] \cup \{1\}$, we can set $\widehat\bTheta_{ij} = 0$ for any $(i,j) \in \Omega_0$ and estimate the remaining entries (which are bounded away from $0$ and $1$) with our procedure.
  
  On the other hand, cases where the entries of $\bTheta^*$ are of different order of magnitude are common in the literature in the case of planted partition models and assortative and disassortative SBM. In the planted partition model, the matrix of connections probabilities between communities is given by $\bQ^* = (p-q) \bold{I}_k + q \bold{1}_k \bold{1}_k^T$, where $p>q$,  $\bold{I}_k$ is the identity matrix and $\bold{1}_k \bold{1}_k^T$ the matrix whose entries are all equal to $1$. This amounts to saying that the probability that two nodes are connected only depends on whether they belong to the same community or not. This model can be relaxed to give rise to the assortative model, where the within group probabilities of connection $\bQ^*_{aa}$ are larger than the between group probabilities of connection $\bQ^*_{bc}$ : there exists $p,q \in [0,1]$ such that for any $a\neq b$, one has $\bQ^*_{ab} \leq q < p \leq \bQ^*_{aa}$. The disassortative model corresponds to the case where between communities connections are more likely than within community connections: one has for any $a \neq b$, $\bQ^*_{aa} \leq q < p \leq \bQ^*_{ab}$. The last two models are closely related. Indeed, if $\bA$ is drawn from an assortative SBM, $\bold{1}_n \bold{1}_n^T - \bold{I}_n - \bA$ corresponds to a realization of a disassortative SBM.
  
  In the planted partition model, maximizing the likelihood is equivalent to finding a partition maximizing the within group connectivity, i.e., maximizing $\sum_{i<j}\bA_{ij}\bZ_{ij}$ where $\bZ_{ij} = \mathds{1}\{z(i) = z(j)\}$. Convex relaxations of the constraints on $\bZ$ have been studied in the literature \cite{HajekSPD, Bandeira2018, SPDSBM}, and theoretical guarantees for these algorithms for the problem of communities recovery have been established under assumptions on the gap $p-q$. In these models, communities are characterized by higher (respectively lower) connectivity, and the assumption that $q \ll p$ actually makes the recovery problem easier. By contrast, the definition of a community in the SBM as a set of nodes with the same stochastic behaviour is far more general. It covers settings not suitably described by assortative or disassortative models, as, e.g., graphs with leaders and followers such as  the well known example of Zachary's  Karate Club (see, e.g., \cite{Leger2014}). In these models, leaders  are seldomly linked one to another, but are highly connected to their own set of followers. On the other hand, followers rarely connect  one to another or to more than one leader. By comparison, our results hold without any assumption on the assortativity or the disassortativity of the model.

In a first time, we assume that we know $\gamma_n$ and $\rho_n$. We will discuss how to estimate these values in Section
   \ref{subsection:graphon}. Let  $\widehat{\bTheta}$ be the block constant estimator based on the maximization of the likelihood among block constant matrices with entries in $[\gamma_n, \rho_n]$:
\begin{equation}\label{eq:MLE}
\begin{split}
&\widehat{\bTheta}_{i<j} = \widehat{\bQ}_{\widehat{z}(i) \widehat{z}(j)},\  \widehat{\bTheta}_{ii}=0\\
&(\widehat{\bQ},\widehat{z}) \in \underset{\bQ \in [\gamma_n,\rho_n]^{k\times k}_{\rm sym}, z\in \cZ_{n,k}}{\argmin}\sum_{i<j} \bX_{ij}\cK (\bA_{ij},\bQ_{z(i)z(j)}).
\end{split}
\end{equation}
Here we assume that $k$ is fixed and that it can depend on the number of
nodes $n$. $k$ can be chosen using a network cross-validation
method \cite{CrossValK} or, when the graphon is a step function, it can be chosen using a sequential goodness-of-fit testing procedure \cite{lei2016} or a
likelihood-based model selection method \cite{BickelMS}. When graphon is H\"older-continuous, we provide a choice of $k$ to optimize the usual trade-off between bias and variance of our estimator in Section \ref{subsection:graphon}.
Note that the maximum likelihood estimator does not require the knowledge of the matrix of sampling probabilities $\bPi$. On the other hand, this matrix, which characterizes the difficulty induced by missing observations, appears when bounding the error of this estimator. In Theorem \ref{thm_matrix_oracle}, we provide a general bound on the $\Vert \cdot \Vert_{2,\bPi}$-norm of the error, while in Corollaries \ref{corPi1} and \ref{cor2} we derive specific bounds on its $\Vert \cdot \Vert_{2}$-norm under further assumptions on the matrix $\bPi$.
\begin{thm}\label{thm_matrix_oracle}
Assume that $\bA$ is drawn according to \eqref{graphSeq}, and that $\rho_n = \omega(n^{-1})$. Then, there exists absolute constants $C,C'>0$ such that with probability at least $1 - 9\exp\left(-C\rho_nn\log(k)\right)$
\begin{align}\label{eq:th1}
\Vert \bTheta^{*}-\widehat \bTheta\Vert^{2}_{2,\bPi}\leq C' \rho_n \left (\mathcal{K}_{\bPi}(\bTheta^{*}, \widetilde \bTheta) + \frac{\rho_n^2}{\left(1 - \rho_n \right)^2 \land \gamma_n^2}\left(k^2 + n\log(k) \right)\right ).
\end{align}
\end{thm}
\begin{remark}
 Note that we are not interested in regimes for which $\rho_n = O(n^{-1})$, as Theorem \ref{KloppBorneInf} implies that in this setting the constant estimator with all entries equal to the average node degree attains the minimax rate.
\end{remark}
\begin{remark}
 This bound is stated as a function of the of weighted Kullback-Leibler divergence $\mathcal{K}_{\bPi}$ and of the oracle matrix $\widetilde \bTheta$ defined in \eqref{eq:barThetaKL}. Note that it implies the weaker bound $$\Vert \bTheta^{*}-\widehat \bTheta\Vert^{2}_{2,\bPi}\leq C' \rho_n \left (\mathcal{K}(\bTheta^{*}, \widetilde \bTheta^f) + \frac{\rho_n^2}{\left(1 - \rho_n \right)^2 \land \gamma_n^2}\left(k^2 + n\log(k) \right)\right ) $$ where $\widetilde \bTheta^f$ is the oracle matrix for the full Kullback-Leibler divergence $\mathcal{K}$. Indeed, one has $\mathcal{K}_{\bPi}(\bTheta^{*}, \widetilde \bTheta) \leq \mathcal{K}_{\bPi}(\bTheta^{*}, \widetilde \bTheta^f) \leq \mathcal{K}(\bTheta^{*}, \widetilde \bTheta^f)$.
\end{remark}
In the case where all entries are observed, that is $\bPi_{ij} = 1$ for any $i<j$, the rate attained by the maximum likelihood estimator is given by the following corollary.
\begin{cor}\label{corPi1}
Assume that $\bA$ is drawn according to \eqref{graphSeq}, that $\forall 1 \leq i < j \leq n, \ \bPi_{ij} =1$ and that $\rho_n = \omega(n^{-1})$. Then, there exists positive constants $C,C'>0$ such that with probability at least $1 - 9\exp\left(-C\rho_n n\log(k)\right)$
\begin{align*}
\Vert \bTheta^{*}-\widehat \bTheta\Vert^{2}_{2}\leq C' \rho_n \left (\mathcal{K}(\bTheta^{*}, \widetilde \bTheta) + \frac{\rho_n^2}{\left(1 - \rho_n \right)^2 \land \gamma_n^2}\left(k^2 + n\log(k) \right)\right ).
\end{align*}
\end{cor}
If we assume that the probability of observing any entry of the adjacency matrix is bounded away from $0$, Theorem \ref{thm_matrix_oracle} can be adapted to provide a bound on the risk of our estimator under the Frobenius norm. Indeed, if $\underset{1\leq i<j\leq n}{\min}\{\bPi_{ij}\} \geq p$, then $\Vert \bTheta^{*}-\widehat \bTheta\Vert^{2}_{2}\leq \frac{1}{p}\Vert \bTheta^{*}-\widehat \bTheta\Vert^{2}_{2,\bPi}$ and we get the following result.
\begin{cor}\label{cor2}
Assume that $\bA$ is drawn according to \eqref{graphSeq}, that $\underset{1\leq i<j\leq n}{\min}\{\bPi_{ij}\} \geq p$ and that $\rho_n = \omega(n^{-1})$. Then, there exists absolute constants $C,C'>0$ such that with probability at least $1 - 9\exp\left(-C\rho_n \left(k^2 + n\log(k) \right)\right)$
\begin{align*}
\Vert \bTheta^{*}-\widehat \bTheta\Vert^{2}_{2}\leq C' \frac{\rho_n}{p}\left (\mathcal{K}_{\bPi}(\bTheta^{*}, \widetilde \bTheta) + \frac{\rho_n^2}{(\left(1 - \rho_n \right)^2 \land \gamma_n^2)}\left(k^2 + n\log(k) \right)\right ).
\end{align*}
\end{cor}

 Previously, the problem of estimation of connections probabilities matrix $\bTheta^*$ from partial observations of the network was studied, in particular, by Gao et al.  \cite{2015gaoBiclustering}. In this paper, the authors assume  that any entry of the adjacency matrix $\bA$ is observed independently from the others with the same probability $p$,  which is assumed to be known. They establish the following lower bound on the risk of any estimator for the stochastic block model.
  \begin{thm}[Gao et al., 2017]\label{KloppBorneInf}
Assume that $\bA$ is drawn according to \eqref{blockmodel}, and that each edge is observed independently from the others with probability $p$. There exists universal constants $C,C'>0$ such that 
$$\underset{\widehat \bTheta}{\inf} \ \underset{\bTheta^* \in \cT_k,\ \left \Vert\bTheta^* \right \Vert_{\infty} \leq \rho_n}{\sup} \mathbb{P} \left[\left \Vert \bTheta^* - \widehat \bTheta \right \Vert_2^2  \geq C \left(\frac{\rho_n(n\log(k) + k^2)}{p}\land \rho_n^2n^2 \right)\right] > C'.$$
\end{thm}
\noindent The authors of \cite{2015gaoBiclustering} study the following weighted least square estimator :
\begin{equation}
    \widehat{\bTheta} \in \argmin \left \Vert\bTheta \right \Vert_2^2\nonumber - \frac{2}{p}\underset{i<j}{\sum}\bX_{ij}\bTheta_{ij}\bA_{ij}
\end{equation}
where $p$ is the probability of observing any entry of the adjacency matrix. They prove that this estimator is minimax optimal in the uniform sampling setting. However, this estimator cannot handle more realistic, non-uniform sampling schemes. Our estimator, on the other hand, works in the non-uniform sampling setting, is adaptative to the sampling design, and does not require information on the probability of observing the entries of the adjacency matrix $\bA$.  In the particular case when $\bX_{ij} \sim \text{Bernoulli}(p)$ and $\rho_n = O(\gamma_n)$, Corollary \ref{cor2} and the lower bound in Theorem \ref{KloppBorneInf} ensures that the maximum likelihood estimator is minimax optimal. We underline that although the lower bound  has been established in \cite{2015gaoBiclustering} for $\bTheta \in \cT_k$, $\left \Vert \bTheta \right \Vert_{\infty} \leq \rho_n$, its proof can be adapted to provide a lower bound on the convergence rate for a smaller set of parameters $\left\{\bTheta \in \cT_k, \ \left \Vert \bTheta \right \Vert_{\infty} \leq \rho_n, \ \underset{i<j}{\min}\{\bTheta_{ij}\} \geq \gamma_n \right\}$. Indeed, the ``non parametric" as well as the ``clustering" components of the rate are established using matrices with entries close to $\frac{\rho_n}{2}$.

\subsection{Choice of $\gamma_n$ under general assumptions}
\label{subsection:unknowngamma}
In this section, we deal with the setting when condition $\underset{i<j}{\min} \ \bTheta^*_{ij} > \gamma_n $ is violated. In what follows we consider the sparse case, that is $\rho_n \rightarrow 0$, so $\gamma_n \leq 1 - \rho_n$ for $n$ large enough. As discussed in \citep{KloppGraphon}, we can easily estimate $\rho_n$ (see also Section \ref{subsection:graphon}). On the other hand, when some entries of the matrix of connections probabilities $\bTheta^*$ can be $0$ or arbitrarily close to $0$, choosing the best sequence $\gamma_n$ comes down to a trade-off between errors caused by estimating entries smaller than $\gamma_n$ by $\gamma_n$, and the bound obtained in Theorem \ref{thm_matrix_oracle}. We first consider the case when there exists a sequence $\gamma_n$ such that number of small entries $n_s = \sum_{i<j} \mathds{1}\{\widetilde \bTheta_{ij} < \gamma_n\}$ is small enough. Then, we have the following result:
\begin{cor}
Assume that $\bA$ is drawn according to \eqref{graphSeq}, that $\rho_n = \omega(n^{-1})$ and that $n_s \leq \frac{k^2 \lor (n\log(k))}{\rho_n}$. Then, there exists absolute constants $C,C'>0$ such that with probability at least $1 - 9\exp\left(-C\rho_n \left( k^2 + n\log(k) \right)\right)$
 \begin{align*}
\Vert \bTheta^{*}-\widehat \bTheta\Vert^{2}_{2,\bPi}\leq C' \rho_n \left (\mathcal{K}_{\bPi}(\bTheta^{*}, \widetilde \bTheta) + \frac{\rho_n^2}{\gamma_n^2} \left(k^2 + n\log(k) \right)\right ).
\end{align*}
\end{cor}
To see it, we define
\begin{equation}\label{eq:barThetaS}
\begin{split}
&\widetilde\bTheta_{ij}^s = \bQ^s_{z^*(i) z^*(j)}, \widetilde\bTheta^s_{ii}=0\ \\
 &\bQ^s_{ab} = \bQ^*_{ab} \lor \gamma_n
    \end{split}
\end{equation}
 where $\bQ^*$ is given by \eqref{eq:barThetaKL}. Note that $\widetilde\bTheta^s$ and $\widehat{\bTheta}$ are defined on the same set, and thus $\mathcal{K}_{\bX}(\bA, \widehat\bTheta) \leq \mathcal{K}_{\bX}(\bA, \widetilde\bTheta^s)$. Adapting the proof of Theorem \ref{thm_matrix_oracle} gives 
\begin{eqnarray}\label{BorneAvecns}
\nonumber \Vert \bTheta^{*}-\widehat \bTheta\Vert^{2}_{2,\bPi}&\leq& C' \rho_n \left (\mathcal{K}_{\bPi}(\bTheta^{*}, \widetilde\bTheta^s) + \frac{\rho_n^2}{\gamma_n^2}\left(k^2 + n\log(k) \right)\right )\\
\nonumber &\leq & C' \rho_n \left (\mathcal{K}_{\bPi}(\bTheta^{*}, \widetilde\bTheta) + \mathcal{K}_{\bPi}(\bTheta^{*}, \widetilde\bTheta^s) - \mathcal{K}_{\bPi}(\bTheta^{*},\widetilde\bTheta) + \frac{\rho_n^2}{\gamma_n^2}\left(k^2 + n\log(k) \right)\right )\\
\label{tradeoff}&\leq& C' \rho_n \left (\mathcal{K}_{\bPi}(\bTheta^{*}, \widetilde\bTheta) + 2\gamma_n n_s + \frac{\rho_n^2}{\gamma_n^2}\left(k^2 + n\log(k) \right)\right )
\end{eqnarray}
where \eqref{tradeoff} follows from Lemma \ref{KLSeuil}. Note that, if there exists a sequence $\gamma_n$ such that $\rho_n = O(\gamma_n)$ and $n_s \leq k^2 \lor (n\log(k))/\rho_n$, the upper bound on the risk obtained in \eqref{BorneAvecns} matches the bound of Theorem \ref{KloppBorneInf} and is minimax optimal. 

Without any assumption on the number of small entries of the matrix of connections probabilities, we choose $\gamma_n = \gamma(\rho_n) \triangleq n^{\frac{-2}{3}}\rho_n^{\frac{2}{3}}\left(k^2 + n\log(k) \right)^{\frac{1}{3}}$ and obtain the following bound. 
\begin{cor}
Assume that $\bA$ is drawn according to \eqref{graphSeq}, and that $\rho_n = \omega(n^{-1})$. Let 
\begin{equation*}
\begin{split}
&\widehat{\bTheta}_{i<j} = \widehat{\bQ}_{\widehat{z}(i) \widehat{z}(j)},\  \widehat{\bTheta}_{ii}=0\\
&(\widehat{\bQ},\widehat{z}) \in \underset{\bQ \in [\gamma(\rho_n),\rho_n]^{k\times k}_{\rm sym}, z\in \cZ_{n,k}}{\argmin} \sum_{i<j}\bX_{ij}\cK(\bA_{ij},\bQ_{z(i)z(j)}). 
\end{split}
\end{equation*}
There exists absolute constants $C,C'>0$ such that with probability at least $1 - 9\exp\left(-C\rho_n \left( k^2 + n\log(k) \right)\right)$
 \begin{align*}
\Vert \bTheta^{*}-\widehat \bTheta\Vert^{2}_{2,\bPi}\leq C' \rho_n \left (\mathcal{K}_{\bPi}(\bTheta^{*}, \widetilde \bTheta) + \rho_n^{\frac{2}{3}} n^{\frac{4}{3}}\left(k^2 + n\log(k) \right)^{\frac{1}{3}}\right ).
\end{align*}
\end{cor}

\noindent If $k$ is not too large, the rate of convergence is essentially multiplied by $\left(n\rho_n\right)^{\frac{2}{3}}$.

\subsection{Choice of $\gamma_n$ for sparse positive graphons}
\label{subsection:graphon}
 In Theorem \ref{thm_matrix_oracle} we have established an oracle bound for the maximum likelihood estimator with entries belonging to $[\gamma_n, \rho_n]$. Defining our estimator requires us to estimate the values of these two sparsity parameters, which are usually unknown. When the matrix of connections probabilities $\bTheta^*$ is generated according to the sparse graphon model \eqref{sparse_graphon} where $W^*$ is bounded away from $0$, these bounds will be of the same order of magnitude and decrease as the expected node degree. Under this assumption, we can use $\widehat{d}$, the average number of edges, to estimate $\gamma_n$ and $\rho_n$. Indeed, it is easy to see that, with probability close to $1$, $\widehat{d}$ is close to $d = \rho_n\underset{0}{\overset{1}{\int}}\underset{0}{\overset{1}{\int}}W^*(x,y)dx dy$, the expected node degree.  Note that if the graphon $W^*$ is H\"older continuous or is a step function, assuming that $W^*>0$ is enough to ensure that there exists a constant $C_{inf}>0$ such that $W^* \geq C_{inf}$. 
 
 To simplify the exposition, we will assume that we observe all the entries of $\bA$. Our results can be extended to the missing observations scheme described in Section \ref{subsection:MLEDescription} under the assumption that the entries of the sampling probability matrix $\bPi$ are bounded away from $0$. Let $\Omega$ be a subset of $\{(i,j) \in [n]^2, i<j\}$ of size $n$ sampled independently of $\bA$, and let 
\begin{equation*}
\begin{split}
&\widehat{d} = \frac{1}{n}\sum_{(i,j) \in \Omega} \bA_{ij}\\
& \widehat{\rho_n} = (\log(n))^{\frac{1}{5}}\widehat{d} \ , \ \widehat{\gamma_n} = (\log(n))^{-\frac{1}{5}}\widehat{d}. 
\end{split}
\end{equation*}
We use $\widehat{\rho_n}$ and $\widehat{\gamma_n}$ to build the restricted maximum likelihood estimator of the matrix of connections probabilities based on the   the observations of $\bA_{ij}$ with $\{(i,j) \in [n]^2, i<j\} \backslash \Omega$: 
\begin{equation*}
\begin{split}
&\widehat{\bTheta}_{i<j} = \widehat{\bQ}_{\widehat{z}(i) \widehat{z}(j)},\  \widehat{\bTheta}_{ii}=0\\
&(\widehat{\bQ},\widehat{z}) \in \underset{\bQ \in [\widehat{\gamma_n}, \widehat{\rho_n}]^{k\times k}_{\rm sym}, z\in \cZ_{n,k}}{\argmin} \sum_{(i,j) \not \in \Omega}\cK(\bA_{ij},\bQ_{z(i)z(j)}).
\end{split}
\end{equation*}
We prove the following upper bound on the risk of this adaptive estimator:
\begin{thm}\label{thm:AdaptEst}
Assume that $\bA$ is drawn according to the sparse graphon model and that $C_{inf} \triangleq \underset{(x,y) \in [0,1]^2}{\inf} W^*(x,y)>0$, $ \rho_n = o(\log(n)^{\frac{-1}{5}})$ and $\rho_n = \omega(n^{-1})$. Then, there exists positive constants $N, C, C'$ depending only on $C_{inf}$, such that, for $n \geq N$, with probability at least $1 - 7\exp\left( - C n \rho_n \right)$, we have
\begin{align*}
\left\Vert \bTheta^{*}-\widehat \bTheta\right\Vert^{2}_{2}\leq C' \rho_n \log(n) \left (\mathcal{K}(\bTheta^{*}, \widetilde \bTheta) + \left(k^2 + n\log(k)\right)\right ).
\end{align*}
\end{thm}
 
In the sparse graphon model, if the graphon $W^*$ is bounded away from $0$ and $ n^{-1} \ll \rho_n \ll \log(n)^{\frac{-1}{5}}$, our adaptive estimator is optimal in the minimax sense up to a log factor. This setting includes the sparse stochastic block model considered for example in \cite{LikelihoodBickel}: in this model, the matrix of probabilities of connection between communities $\bQ^*$ is such that $\bQ^* = \rho_n \bQ^0$, for some fixed matrix $\bQ^0$ with entries in $(0,1]$ (in this case, the graphon function $W^*$ is a step function bounded away from $0$). More generally, Theorem \ref{thm:AdaptEst} covers the case of $\bTheta^* = \rho_n \bTheta^0$, and $\bTheta^0$ is a fixed matrix with entries in $(0, 1]$.

When we can not assume that the graphon $W^*$ is bounded away from $0$, we can  use the same trade-off as in \eqref{eq:barThetaS} and choose $\widehat{\gamma_n} = \gamma(\widehat\rho_n)$. Then, with high probability, we obtain  the following bound on the risk of the adaptative estimator: 
\begin{align*}
\Vert \bTheta^{*}-\widehat \bTheta\Vert^{2}_{2}\leq C' \rho_n \log(n)\left (\mathcal{K}(\bTheta^{*}, \widetilde \bTheta) + \left(\log(n)^{\frac{1}{5}} \rho_n\right)^{\frac{2}{3}} n^{\frac{4}{3}}\left(\frac{k^2}{n^2} + n\log(k) \right)^{\frac{1}{3}}\right).
\end{align*}

\subsection{Smooth graphons}
\label{subsection:MLEGraphon}

We have established a non-asymptotic bound on the risk of the maximum likelihood estimator depending on the Kullback-Leibler divergence between $\bTheta^*$ and its oracle approximation by a block constant matrix corresponding to a SBM with $k$ communities. 
While studying the graphon model \eqref{sparse_graphon}, two classes of graphons are of particular interest: step function graphons and H\"older continuous graphons \cite{Olhede14722, KloppGraphon, gao2015optimal,USVTXu}. A graphon $W$ is called a \textit{step function} if there exists a partition $S_1 \cup ... \cup S_k$ of $[0,1]$ into measurable sets such that the graphon $W$ is constant on any product set $S_a \times S_b$.  For step function graphons, the model corresponds to the stochastic block model described in \eqref{blockmodel}: in this case, the oracle matrix $\widetilde{\bTheta}$ is equal to the matrix of connections probabilities $\bTheta^*$. Next, we bound the Kullback-Leibler divergence between $\bTheta^*$ and its oracle approximation by a block constant matrix for H\"older continuous graphons. We also provide the optimal choice for the number of communities $k$ for our estimator. 

We  consider graphons that are weakly isomorphic to a smooth function. More precisely, for any $\alpha >0$ and $M>0$, let $\cF_{\alpha}(M)$ be the class of H\"older functions, defined as follows:
\begin{eqnarray*}
\cF_{\alpha}(M) = \Big\{W &: &[0,1]^2 \rightarrow [0,1],\ \forall (x,y),\ (x',y') \in [0,1]^2, \\
&&\left \vert W(x',y') - \cP_{\lfloor\alpha \rfloor}((x,y), (x'-x, y'-y)) \right \vert \leq M\left(\left \vert x-x' \right \vert^{\alpha - \lfloor\alpha \rfloor} + \left \vert y-y' \right \vert^{\alpha - \lfloor\alpha \rfloor}\right)  \Big\}
\end{eqnarray*}
where $\cP_{\lfloor\alpha \rfloor}((x,y), \cdot )$ is the Taylor polynomial of $W$ of degree $\lfloor\alpha \rfloor$ at point $(x,y)$. In particular, if $W \in \cF_{\alpha}(M)$, $\forall (x,y),\ (x',y') \in [0,1]^2$, 

\begin{equation}
\label{Hoelder}
\left \vert W(x',y') - W(x,y)\right \vert \leq M\left(\left \vert x-x' \right \vert^{\alpha \land 1} + \left \vert y-y' \right \vert^{\alpha\land 1}\right).
\end{equation}
When the graphon is H\"older continuous, the following proposition provides an upper bound on the Kullback-Leibler divergence between $\bTheta^*$ and $\widetilde{\bTheta}$.
\begin{prp}\label{prop:MLGraphon}
Consider the sparse graphon model \eqref{sparse_graphon} with $W^* \in \cF_{\alpha}(M)$ where $\alpha, M>0$ and we assume that $C_{inf} \triangleq \underset{(x,y) \in [0,1]^2}{\inf} W^*(x,y)>0$, $\rho_n\leq 1 - C_{inf}$ and $\rho_n = \omega(n^{-1})$. Then, almost surely, there exists a k-block constant matrix $\bTheta^{bc}$ such that
\begin{equation}\label{bound_kullback}
  \cK\left(\bTheta^*, \bTheta^{bc} \right) \leq  \frac{4n^2\rho_nM^2}{C_{inf}(1-\rho_n)} \left(\frac{1}{k}\right)^{2(\alpha \land 1)}.  
\end{equation}
\end{prp}
Proposition \ref{prop:MLGraphon} enables us to bound the bias of estimating $\bTheta^*$ by an oracle SBM with $k$ communities. On the other hand, the bound given in Theorem \ref{thm_matrix_oracle} can be considered as the variance term of a block constant estimator with $k$ blocks. To optimize the trade-off between these two terms, we choose $k$ as follows 
\begin{equation}\label{choicek}
   k = \left \lceil n^{\frac{1}{1+ (\alpha\land 1)}}\rho_n^{\frac{1}{2+ 2(\alpha\land 1)}}  \right \rceil
\end{equation}
and obtain the following result:
\begin{thm}\label{thm:MySupGraphon}
Consider the sparse graphon model \eqref{sparse_graphon} with $W^* \in \cF_{\alpha}(M)$ where $\alpha, M>0$ and we assume that $C_{inf} \triangleq \underset{(x,y) \in [0,1]^2}{\inf} W^*(x,y)>0$, $\rho_n\leq 1 - C_{inf}$ and that $\rho_n = \omega(n^{-1})$. Then,
there exists constants $C, C'>0$, depending only on $M$, $\alpha$ and $C_{inf}$, such that, the restricted maximum likelihood estimator defined by \eqref{eq:barThetaKL} constructed with $k$ defined by \eqref{choicek} satisfies
$$ \left\Vert\bTheta_{ij}^* - \widehat \bTheta_{ij}\right\Vert_2^2 \leq C\rho_n\left(n^{\frac{2}{1 + (\alpha \land 1)}}\rho_n^{\frac{1}{1 + (\alpha \land 1)}}  + n\log(\rho_n n) \right)$$
with probability larger than $1- 9\exp(-C'\rho_n n \log(\rho_nn)).$
\end{thm}
The bound  obtained in Theorem \ref{thm:MySupGraphon} matches the minimax optimal rate established in
\cite{KloppGraphon}  and proves that the maximum likelihood estimator is optimal for estimating the matrix of connections probabilities in graphon model for graphons $W^*$ in the H\"older class.

\section{Variational approximation to the maximum likelihood estimator}\label{Section:Variational}

Due to the existence of local minima in the likelihood function and the subsequent necessary optimisation over the set of $k^n$ labels, the maximum likelihood estimator defined in Equation \eqref{eq:MLE} cannot be computed in polynomial time. A tractable variational approximation to this estimator has been introduced in \cite{VarEst} to study dense and fully-observed stochastic block models. It has been extended to sparse stochastic block models in \cite{LikelihoodBickel}. More recently, the authors of \cite{TabPra} used variational methods to approximate the maximum likelihood estimator in networks with missing observations. This method estimates the matrix $\bQ^*$ of probabilities of connections between communities, and the proportions of the different communities, and can be used to estimate the community label function $z^*$ and the matrix $\bTheta^*$ of probabilities of connections between nodes.

Note that variational approximations to the maximum likelihood estimator have been introduced in a model closely related to the conditional stochastic block model described in Equation \eqref{blockmodel}. In this model, the labels are assumed to be latent random variables drawn according to the following distribution\begin{equation*}
    \forall \ 1\leq i \leq n, z(i) \sim \mathcal{M}(\alpha^*)
\end{equation*}where $\alpha^* = (\alpha^*_1, ..., \alpha^*_k)$, and $\alpha^*_an$ is the expected size of community $a$. Conditionally on $z = z^*$, the edges are then drawn according to the conditional stochastic block model described in Equation \eqref{blockmodel}. Note that the stochastic block model with random labels is a special case of the graphon model described in Equation \eqref{dense_graphon}, in which the graphon $W^*$ is a step function.

The likelihood of the observed adjacency matrix under a stochastic block model with random labels with parameters $(\alpha, \bQ)$ is given by \begin{eqnarray*}
\mathfrak{l}_{\bX}(\bA; \alpha, \bQ) &=& \underset{z \in \cZ_{n,k}}{\sum} \left(\underset{i \leq n}{\prod}\alpha_{z(i)} \right) \exp \left(\cL_{\bX}(\bA;z,\bQ) \right).
\end{eqnarray*}The Expectation-Maximization (EM) algorithm cannot be used to maximise $\mathfrak{l}_{\bX}$, as it would require to evaluate the expectation of the label function $z$ for given parameters $(\alpha, \bQ)$ by summing over $k^n$ possible labels. To circumvent this problem, variational methods approximate the posterior distribution of $z$ given the observed entries $\bX \odot \bA$ of the adjacency matrix $\bA$, denoted $\mathbb{P}\left(\cdot \vert \bX\odot\bA, \alpha, \bQ \right)$, by a simpler distribution chosen so that the expectation step of the EM procedure becomes tractable. More precisely, the posterior distribution $\mathbb{P}\left(\cdot \vert \bX\odot\bA, \alpha, \bQ \right)$ is approximated by a multinomial distribution denoted $\mathbb{P}_\tau$, such that $\mathbb{P}_\tau(z) = \prod_{1\leq i \leq n} m(z\vert \tau^i)$, where $m( \cdot \vert \tau^i)$ is the density of the multinomial distribution with parameter $\tau^i = \left(\tau^i_1, ..., \tau^i_K\right)$, and $\tau = \left(\tau^1, ..., \tau^n\right)$. Then, the variational estimator is defined as
\begin{eqnarray*}
\left(\widehat{\alpha}^{VAR}, \widehat{\bQ}^{VAR}, \widehat{\tau}^{VAR}\right) &=&  \underset{(\alpha, \bQ) \in \mathcal{Q}, \tau \in \mathcal{T}}{\argmax} \mathcal{J}_{\bX}(\bA; \tau, \alpha, \bQ)\\
\text{for }\ \ \  \mathcal{J}_{\bX}(\bA; \tau, \alpha, \bQ)  &=& \log\left(\mathfrak{l}_{\bX}(\bA; \alpha, \bQ)\right) - KL\left(\mathbb{P}_\tau(\cdot)\vert\vert\mathbb{P}\left(\cdot \vert \bX\odot\bA, \alpha, \bQ \right) \right)
\end{eqnarray*}
where $\mathcal{Q}$ and $\mathcal{T}$ are the respective parameter spaces for the parameters $(\alpha,\bQ)$ and $\tau$, and $KL$ denotes the Kullback-Leibler divergence between two distribution. Note that $\exp\left(\mathcal{J}_{\bX}(\bA; \tau, \alpha, \bQ)\right)$ provides a lower bound on $\mathfrak{l}_{\bX}(\bA; \alpha, \bQ)$, as $KL\left(\mathbb{P}_\tau(\cdot)\vert\vert\mathbb{P}\left(\cdot \vert \bX\odot\bA, \alpha, \bQ \right) \right) \geq 0$ for any parameter $(\alpha, \bQ)$. The variational estimator is obtained using the EM algorithm derived in \cite{TabPra}, which alternates between the following two steps :
\begin{itemize}
    \item Estimation Step : for given parameters $(\alpha, \bQ)$, the variational parameter $\tau$ maximizing $\mathcal{J}_{\bX}(\bA; \tau, \alpha, \bQ)$ is given by the following fixed point equation :  
    $$\tau^i_a =  c_i \alpha_a \prod_{j\neq i : \bX_{ij} = 1} \prod_{b \leq k} \left(\bQ_{ab}^{\bA_{ij}}\left(1-\bQ_{ab}\right)^{1-\bA_{ij}}\right)^{\tau_b^j} \ \ \text{ where $c_i$ is a normalising constant};$$
    \item Maximisation Step : for a given parameter $\tau$, the parameters $(\alpha, \bQ)$ maximizing $\mathcal{J}_{\bX}(\bA; \tau, \alpha, \bQ)$ are given by  
    $$\alpha_a = \frac{\sum_i \tau^i_a}{n} \text{   ,    } \bQ_{ab} = \frac{\sum_{i \neq j}\bX_{ij} \tau^i_a\tau^j_b \bA_{ij}}{\sum_{i \neq j}\bX_{ij} \tau^i_a\tau^j_b}.$$
\end{itemize}
This algorithm is implemented in the package \href{https://cran.r-project.org/web/packages/missSBM/index.html}{\texttt{missSBM}}. Since the EM algorithm is not guaranteed to converge to a global maximum, a clever initialization of the parameter $\tau$ using a clustering step is performed.

Theoretical properties of the variational estimator have been first studied in \cite{VarEst} for the dense stochastic block model, in \cite{LikelihoodBickel} for the sparse stochastic block model, and more recently in \cite{TabThe} for a dense stochastic block model with observations missing uniformly at random. These results establish that maximizing $\max_{\tau \in \mathcal{T}} \mathcal{J}_{\bX}(\bA; \tau, \alpha, \bQ)$ is equivalent to maximizing $\mathfrak{l}_{\bX}(\bA; \alpha, \bQ)$. They also establish that the estimator obtained by maximizing $\mathfrak{l}_{\bX}(\bA; \alpha, \bQ)$ converges to the true parameters $(\alpha^*, \bQ^*)$, and conclude that $(\widehat{\alpha}^{VAR}, \widehat{\bQ}^{VAR})$ also converges to $(\alpha^*, \bQ^*)$. 

We build on this approximation and use the following estimator for the label function $z^*$:
\begin{eqnarray*}.
   &&\forall \ i \leq n, \ \widehat{z}^{VAR}(i) \triangleq \argmax_{a\leq k} \left(\widehat{\tau}^{VAR}\right)^i_a 
\end{eqnarray*}
Next, we compare the asymptotic behavior of $\widehat{z}^{VAR}$ with that of the profile maximum likelihood estimator of the label function $\widehat{z}$, defined as
\begin{equation*}
(\widehat{\bQ},\widehat{z}) \in \underset{\bQ \in \mathcal{Q}, z\in \cZ_{n,k}}{\argmin}\sum_{i<j} \bX_{ij}\cK (\bA_{ij},\bQ_{z(i)z(j)})
\end{equation*}
where we use the convention that $\bQ \in \mathcal{Q}$ if there exists $\alpha$ such that $(\alpha, \bQ) \in \mathcal{Q}$. 

In Proposition \ref{prp:variational}, we establish the asymptotic equivalence of $\widehat{z}^{VAR}$ and of $\widehat{z}$ in the case of sparse stochastic block model with full observations. To prove it, we will assume that the proportions of different communities are held constant, while the probabilities of connections between communities may decreases at rate $\rho_n$. That is, the parameters $(\alpha^*,\bQ^*)$ verify

\begin{enumerate}[label=A\arabic*]
\item $\alpha^* = \alpha^0$ for some fixed $\alpha^0$ such that $\alpha^0_a >0$ for any $a\in \{1,...,k\}$ \label{A1}
\item $\bQ^* = \rho_n \bQ^0$ for some fixed $\bQ^0 \in (0,1)^{k \times k}$ such that $\overset{k}{\underset{a,b = 1}{\sum}} \alpha^0_{a}\alpha^0_{b}\bQ^0_{ab} = 1$\label{A2}
\end{enumerate}
The normalisation constraint $\overset{k}{\underset{a,b = 1}{\sum}} \alpha^0_{a}\alpha^0_{b}\bQ^0_{ab} = 1$ ensure the identifiability of the parameters $(\bQ^0, \rho_n)$ (see \cite{LikelihoodBickel}). In the following, we denote by $\mathcal{Q}$ the set of parameters $(\alpha, \bQ)$ verifying Assumptions \ref{A1} and \ref{A2}. Note that the stochastic block model is identifiable up to a simultaneous permutation of communities and of rows and columns of the parameters $\bQ^*$. For any two label functions $z, z'$, we write $z \sim z'$ if there exists a permutation $\sigma$ of $\{1, ..., k\}$ such that $\left(z(\sigma(a))\right)_{a\leq k} = \left(z(a)\right)_{a\leq k}$. 

\begin{prp}\label{prp:variational}
Assume that $(z^*,\bA)$ is generated from a stochastic block model with parameters $(\alpha^*, \bQ^*) \in \mathcal{Q}$ such that $\bQ^0$ has no identical columns and the sparsity inducing sequence $\rho_n$ satisfies $\rho_n\gg \log(n)/n$.
Then, under Assumptions \ref{A1} and \ref{A2}, $\mathbb{P}\left(\widehat{z}^{VAR} \sim \widehat{z}\right) \rightarrow 1$ when $n \rightarrow \infty$.
\end{prp}

The proof of Proposition \ref{prp:variational} relies on results established in \cite{LikelihoodBickel} and \cite{Bickel21068}, and is given in Section \ref{proof:variational}. These results have been extended in \cite{TabThe} to the case of dense stochastic block model with incomplete observation of the network, where entries of the adjacency matrix are observed uniformly at random. Building on their results, Proposition \ref{prp:variational} can be extended to the case of missing observations.

Once we have estimated the community labels, we can derive an estimator of the matrix of connection probabilities. For all $(a,b) \in \{1,..,k\}$, define 
$$n_{ab}(\widehat{z}^{VAR})  = \left\{
    \begin{array}{ll}
        \vert (\widehat{z}^{VAR}) ^{-1}(a)\vert \times \vert( \widehat{z}^{VAR})^{-1}(b)\vert & \mbox{ if } a \neq b \\
        \vert (\widehat{z}^{VAR}) ^{-1}(a)\vert \times \left(\vert (\widehat{z}^{VAR})^{-1}(a)\vert - 1 \right)  & \mbox{otherwise}
    \end{array}
\right.
$$
and 
$$ \widehat{\bQ}^{ML-VAR}_{ab} \triangleq \frac{\underset{i\in (\widehat{z}^{VAR})^{-1}(a),j\in (\widehat{z}^{VAR}) ^{-1}(b)}{\sum}\bA_{ij}}{n_{ab}(\widehat{z}^{VAR})}.$$
Proposition \ref{prp:variational} implies that with large probability, there exists a permutation $\sigma$ of $\{1, ..., k\}$ such that $\left(\widehat{z}^{VAR}(\sigma(a))\right)_{a \leq k} = \left(\widehat{z}(a)\right)_{a \leq k}$. When this hold, $\left(\widehat{\bQ}^{ML-VAR}_{\sigma(a), \sigma(b)}\right)_{a,b\leq k} = \left(\widehat{\bQ}_{a,b}\right)_{a,b\leq k}$, and the risk bounds obtained in Theorem \ref{thm_matrix_oracle} for the estimator $\widehat{\bTheta}$ hold for the tractable estimator  $\widehat{\bTheta}^{VAR}$, defined as
$$\widehat{\bTheta}^{VAR}_{i\neq j}  =   \widehat{\bQ}^{ML-VAR}_{\widehat{z}^{VAR}(i), \widehat{z}^{VAR}(j)}, \ \widehat{\bTheta}^{VAR}_{ii} = 0.$$

\begin{prp}\label{prp:boundvariational}
Assume that $(z^*,\bA)$ is generated from a stochastic block model with parameters $(\alpha^*, \bQ^*) \in \mathcal{Q}$ such that $\bQ^0$ has no identical columns and the sparsity inducing sequence $\rho_n$ satisfies $\rho_n\gg \log(n)/n$. Then, under Assumptions \ref{A1} and \ref{A2}, there exists a constant $C_{\bQ^0}>0$ depending on $\bQ^0$ such that 
$$\mathbb{P}\left(\left\Vert  \bTheta^* - \widehat{\bTheta}^{VAR} \right \Vert_2^2 \leq  C_{\bQ^0}\rho_n\left(k^2 + n\log(k)\right) \right) \rightarrow 1$$when $n \rightarrow \infty$.
\end{prp}


\section{Conclusion}\label{section:Ccl}

We have studied the problem of estimating the matrix of connections probabilities for the inhomogeneous random graph model and the graphon model with missing link. We have established a non-asymptotic bound on the risk of the maximum likelihood estimator. In particular, we have shown that, if the entries of the probability matrix decrease at the same rate, our estimator achieves the minimax convergence rate. This result is adaptative to the unknown sampling scheme. Moreover, we show that our estimator can be efficiently approximated using variational methods, and thus used in practice.

\section*{Acknowledgments}

The work of O. Klopp was conducted as part of the project Labex MME-DII (ANR11-LBX-0023-01).
The authors want to thank Catherine Matias and Nicolas Verzelen for 
extremely valuable suggestions and discussions.

\medskip

\noindent The authors declare that there is no conflict of interest.

 \section{Proofs} \label{section:proofs}
 
  The proof of Theorem \ref{thm_matrix_oracle} requires bounding the domain of definition of our estimator away from $0$ and $1$ in order to ensure that the loss function associated with the maximum likelihood estimator is Lipschitz. The Lipschitz constant here is equal to  $\frac{1}{\left(1 - \rho_n \right) \land \gamma_n}$. We balance this term by $\rho_n$ by taking advantage of the sparsity of the graph, which implies, in particular, the low variance of $\bA$. For ease of notations, we will assume that $1 - \rho_n\geq \gamma_n$. This is the case when the graph is sparse, and our results still hold in the dense case if we replace $\gamma_n$ by $\gamma_n \land (1 - \rho_n)$ in our bounds.
 
 \subsection{Proof of Theorem \ref{thm_matrix_oracle}}\label{subsection:proofMainTh}
 In previous works on sparse network estimation, the authors used the definition of the least square estimator to bound the norm of the error of this estimator by the scalar product between this error and a noise term. The main difficulty in their proofs consists in bounding this scalar product. Unfortunately, the definition of the maximum likelihood estimator does not yield such a bound, and these technics cannot be used to control its error. Instead, we make use of more refined peeling arguments, in order to take advantage of the Lipschitz continuity of the loss.
 
 Let $\epsilon_n = C \frac{\rho_n^2}{\gamma_n^2}\left( n\log(k) + k^2 \right)$ for some absolute constant $C$ defined as the maximum of the absolute constants appearing in Lemma \ref{Lemma:esperancePi}, Lemma \ref{Lemma:ProbBeKLPi} and Lemma \ref{Lemma:ProbBlKLX}, and let $\epsilon^0 \triangleq \rho_n\epsilon_n$. We start by considering the following two cases:

\textbf{Case 1}: $\Vert \widetilde\bTheta-\widehat \bTheta\Vert^{2}_{2,\bPi} \leq 2 \epsilon^0$. Then the statement of Theorem \ref{thm_matrix_oracle} follows from Lemma \ref{kl_frobenius}:
\[\Vert \bTheta^*-\widehat \bTheta\Vert^{2}_{2,\bPi} \leq 2\Vert \widetilde\bTheta-\widehat \bTheta\Vert^{2}_{2,\bPi} + 2\Vert \bTheta^*-\widetilde\bTheta\Vert^{2}_{2,\bPi} \leq 4\rho_n \epsilon_n + 16 \rho_n \mathcal{K}_{\bPi}(\bTheta^{*}, \widetilde\bTheta).\]

\textbf{Case 2}: $\Vert \widetilde\bTheta-\widehat \bTheta\Vert^{2}_{2,\bPi} > 2\epsilon^0$. Then $\widehat \bTheta$ belongs to the set
\[\mathcal{S}_{\bPi} = \left \{\bTheta \in \underset{z \in \mathcal{Z}_{n,k}}{\cup}\mathcal{T}_z: \Vert \widetilde\bTheta-\bTheta\Vert^{2}_{2,\bPi} \geq 2 \epsilon^0,\ \Vert \bTheta\Vert_{\infty}\leq \rho_n,\ \underset{i<j}{\min} \{\bTheta_{ij}\} \geq \gamma_n \right. \}\]
 The following lemma helps us bound the error in $\left \Vert \cdot \right \Vert_{\bPi}$-norm by the error in $\left \Vert \cdot \right \Vert_{\bX}$-norm. Its proof departs from previous technics used to study network with missing observations. It allows for non-uniform sampling design, and is one of the main technical contribution of this work.
\begin{lem}\label{EqNPiNX}
There exists and absolute constant $C>0$ such that for all $\bTheta \in \mathcal{S}_{\bPi}$ simultaneously we have
\[
\left \vert \left \Vert \bTheta - \widetilde \bTheta \right \Vert_{2, \bPi}^2 - \left \Vert  \bTheta - \widetilde \bTheta \right \Vert_{2, \bX}^2 \right \vert \leq \frac{1}{2} \left \Vert  \bTheta - \widetilde \bTheta \right \Vert_{2, \bPi}^2 
\]
with probability greater than $1 - 2\exp(-Cn \log(k))$.
\end{lem}
Lemma \ref{EqNPiNX} implies that with large probability, $\widehat{\bTheta}$ belongs to the set $\mathcal{S}_{\bX}$ defined as 
\[\mathcal{S}_{\bX} = \left \{\bTheta \in \underset{z \in \mathcal{Z}_{n,k}}{\cup}\mathcal{T}_z: \Vert \widetilde\bTheta-\bTheta\Vert^{2}_{2,\bX} \geq \epsilon^0,\ \Vert \bTheta\Vert_{\infty}\leq \rho_n,\ \underset{i<j}{\min} \{\bTheta_{ij}\} \geq \gamma_n \right \}.\]
To bound $\Vert \widetilde\bTheta-\widehat \bTheta\Vert^{2}_{2,\bPi}$ when $\widehat{\bTheta} \in \mathcal{S}_{\bX} \cap \mathcal{S}_{\bPi}$, we introduce the following notation. For $\bTheta,\ \bTheta' \in (0,1)^{n \times n}_{sym}$ and $\bB, \bC \in [0,1]^{n \times n}_{sym}$ we set $\Delta \cK^{\bC}_{\bB}(\bTheta, \bTheta') = \mathcal{K}_{\bB}(\bC, \bTheta)-\mathcal{K}_{\bB}(\bC, \bTheta').$ Using Lemma \ref{kl_frobenius} we get
\begin{eqnarray*}
\Vert \bTheta^{*}-\widehat \bTheta\Vert^{2}_{2, \bPi}&\leq& 8\rho_n\mathcal{K}_{\bPi}(\bTheta^{*}, \widehat \bTheta)\\
& \leq & 8\rho_n \mathcal{K}_{\bPi}(\bTheta^{*}, \widetilde \bTheta) + 8\rho_n \Delta \cK^{\bTheta^*}_{\bPi}(\widehat\bTheta, \widetilde\bTheta).
\end{eqnarray*}
On the other hand, the definition of $\widehat\bTheta$ implies that $\Delta \cK^{\bA}_{\bX}(\widehat \bTheta, \widetilde \bTheta) \leq 0$ so 
\begin{eqnarray}\label{case2}
\Vert \bTheta^{*}-\widehat \bTheta\Vert^{2}_{2, \bPi}& \leq& 8\rho_n \mathcal{K}_{\bPi}(\bTheta^{*}, \widetilde \bTheta) + 8\rho_n \Delta \cK^{\bTheta^*}_{\bPi}(\widehat\bTheta, \widetilde\bTheta) - 8\rho_n \Delta \cK^{\bA}_{\bX}(\widehat \bTheta, \widetilde \bTheta) \\
 &\leq& 8\rho_n \left( \mathcal{K}_{\bPi}(\bTheta^{*}, \widetilde \bTheta) + \left( \Delta \cK^{\bTheta^*}_{\bPi}(\widehat\bTheta, \widetilde\bTheta) - \Delta \cK^{\bTheta^*}_{\bX}(\widehat\bTheta, \widetilde\bTheta) \right)+ \left( \Delta \cK^{\bTheta^*}_{\bX}(\widehat\bTheta, \widetilde\bTheta) - \Delta \cK^{\bA}_{\bX}(\widehat \bTheta, \widetilde \bTheta)\right) \right) \nonumber .
\end{eqnarray}
To bound the terms involved in equation \eqref{case2}, we control $\underset{\bTheta \in \mathcal{S}_{\bPi}}{\sup}\left \vert \Delta \cK^{\bTheta^*}_{\bPi}(\bTheta, \widetilde\bTheta) - \Delta \cK^{\bTheta^*}_{\bX}(\bTheta, \widetilde\bTheta) \right \vert $ using the concentration of $\bX$ around its expectation $\bPi$, 
and we control 
$\underset{\bTheta \in \mathcal{S}_{\bX}}{\sup}\left \vert \Delta \cK^{\bTheta^*}_{\bX}(\bTheta, \widetilde\bTheta) - \Delta \cK^{\bA}_{\bX}(\bTheta, \widetilde \bTheta) \right \vert$ conditionally on $\bX$ using the concentration of $\bA$ around its expectation $\bTheta^*$.

\begin{lem}\label{EqKl2KlPi}
There exists absolute constants $C, C'>0$ such that for all $\bTheta \in \mathcal{S}_{\bPi}$ simultaneously we have
\[
\left \vert \Delta \cK^{\bTheta^*}_{\bPi}(\bTheta, \widetilde\bTheta) - \Delta \cK^{\bTheta^*}_{\bX}(\bTheta, \widetilde\bTheta) \right \vert  \leq \frac{1}{2\times32\rho_n}\left \Vert \bTheta - \widetilde \bTheta \right \Vert_{2, \bPi}^2 + C\frac{\rho_n^2}{\gamma_n^2}\left(n\log(k) + k^2 \right)
\]
with probability greater than $1 - 2\exp(-C'\rho_n n\log(k))$.
\end{lem}
\begin{lem}\label{ControlKlX}
There exists absolute constants $C,C'>0$ such that conditionally on $\bX$, for any $\bTheta \in \mathcal{S}_{\bX}$ simultaneously we have
\[
\left \vert \Delta \cK^{\bTheta^*}_{\bX}(\bTheta, \widetilde\bTheta) - \Delta \cK^{\bA}_{\bX}(\bTheta, \widetilde \bTheta) \right \vert \leq \frac{1}{4\times 32 \rho_n}\left \Vert \bTheta - \widetilde \bTheta \right \Vert_{2, \bX}^2 + C\frac{\rho_n^2}{\gamma_n^2}\left(n\log(k) + k^2 \right)
\]
with probability greater than $1 - 5\exp(-C' \rho_n n\log(k))$.
\end{lem}
Combining Lemma \ref{EqNPiNX}, Lemma \ref{EqKl2KlPi}, Lemma \ref{ControlKlX} and \eqref{case2} yields that there exists two absolute constants $C,C'>0$ such that with probability greater than $1 - 9\exp\left(-C'\rho_n n \log(k)\right)$
\begin{eqnarray}\label{eqfraise}
\Vert \bTheta^{*}-\widehat \bTheta\Vert^{2}_{2,\bPi} &\leq& 8\rho_n\mathcal{K}_{\bPi}(\bTheta^{*}, \widetilde\bTheta)+ 8 \rho_n\times \frac{1}{2\times 32 \rho_n} \left \Vert \widehat \bTheta - \widetilde \bTheta \right \Vert_{2, \bPi}^2 + 8\rho_n\times \frac{1}{4\times 32 \rho_n} \left \Vert \widehat \bTheta - \widetilde \bTheta \right \Vert_{2, \bX}^2\nonumber\\ &&+ C\rho_n \frac{\rho_n^2}{\gamma_n ^2} \left(n\log(k) + k^2\right)\nonumber\\
&\leq & 8\rho_n\mathcal{K}_{\bPi}(\bTheta^{*}, \widetilde\bTheta)+ \frac{1}{8} \left \Vert \widehat \bTheta - \widetilde \bTheta \right \Vert_{2, \bPi}^2 + \frac{1}{16} \times \frac{3}{2}\left \Vert \widehat \bTheta - \widetilde \bTheta \right \Vert_{2, \bPi}^2 + C\rho_n \frac{\rho_n^2}{\gamma_n^2} \left(n\log(k) + k^2\right)\nonumber\\
&\leq& 8\rho_n\mathcal{K}_{\bPi}(\bTheta^{*}, \widetilde\bTheta)+ \frac{1}{2} \left \Vert \bTheta^* - \widetilde \bTheta \right \Vert_{2, \bPi}^2 + \frac{1}{2} \left \Vert \bTheta^* - \widehat\bTheta \right \Vert_{2, \bPi}^2 + C\rho_n \frac{\rho_n^2}{\gamma_n^2} \left(n\log(k) + k^2\right).
\end{eqnarray}
Lemma \ref{kl_frobenius} and \eqref{eqfraise} imply that there exists two absolute constants $C, C'>0$ such that with probability larger than $1 - 9\exp\left(-C'\rho_n n \log(k)\right)$,
\begin{eqnarray*}
\frac{1}{2}\Vert \bTheta^{*}-\widehat \bTheta\Vert^{2}_{2,\bPi} &\leq& 8\rho_n\mathcal{K}_{\bPi}(\bTheta^{*}, \widetilde\bTheta)+ \frac{1}{2} \times 8\rho_n\mathcal{K}_{\bPi}(\bTheta^{*}, \widetilde\bTheta) + C\rho_n \frac{\rho_n^2}{\gamma_n^2} \left(n\log(k) + k^2\right).
\end{eqnarray*}
 This completes the proof of Theorem \ref{thm_matrix_oracle}.

\subsection{Proof of Lemma \ref{EqNPiNX}}

 To prove Lemma \ref{EqNPiNX}, we show that the probability of the following ``bad" event is small:
\[\mathcal{E} \triangleq \left\{ \exists \bTheta \in \mathcal{S}_{\bPi}: \left \vert \left \Vert \bTheta - \widetilde \bTheta \right \Vert_{2, \bPi}^2 - \left \Vert  \bTheta - \widetilde \bTheta \right \Vert_{2, \bX}^2 \right \vert > \frac{1}{2} \left \Vert\bTheta - \widetilde \bTheta \right \Vert_{2, \bPi}^2\right\}.\]
We use a standard peeling argument (see, e.g., \cite{KloppLowRank}): we slice $\mathcal{S}_{\bPi}$ in different sets, on which we control $\left \Vert \bTheta - \widetilde \bTheta \right \Vert_{2,\bPi}^2$. Recall that $\epsilon_n \triangleq C\frac{\rho_n^2}{\gamma_n^2}\left(n\log(k) + k^2 \right)$ where the absolute constant $C$ is larger than the constant appearing in Lemma \ref{Lemma:esperancePi}, and that $\epsilon^0 \triangleq \rho_n\epsilon_n$. For $l \in \mathbb{N}^*$, we set 
\[\mathcal{S}_{l,\bPi} \triangleq \left \{\bTheta \in \mathcal{S}_{\bPi}:  2^{l-1}(2\epsilon^0) \leq \left \Vert \bTheta - \widetilde\bTheta \right \Vert_{2,\bPi}^2 \leq 2^l(2\epsilon^0)\right \}.\]
If the event $\mathcal{E}$ holds, there exists $l \in \mathbb{N}^*$ such that $\bTheta \in \mathcal{S}_{l,\bPi}$ and 
\begin{equation*}
\begin{split}
\left \vert \left \Vert \bTheta - \widetilde \bTheta \right \Vert_{2, \bPi}^2 - \left \Vert  \bTheta - \widetilde \bTheta \right \Vert_{2, \bX}^2 \right \vert > \frac{1}{2} \left \Vert\bTheta - \widetilde \bTheta \right \Vert_{2, \bPi}^2.
\end{split}
\end{equation*}
Note that $\mathbb{E} \left[\left \Vert  \bTheta - \widetilde \bTheta \right \Vert_{2, \bX}^2\right] =\left \Vert \bTheta - \widetilde \bTheta \right \Vert_{2, \bPi}^2$. The events that we need to control are the following:
\[\mathcal{E}_{l} \triangleq \left\{ \exists \bTheta \in \mathcal{S}_{l,\bPi}: \left \vert \left \Vert  \bTheta - \widetilde \bTheta \right \Vert_{2, \bX}^2 - \mathbb{E} \left[\left \Vert  \bTheta - \widetilde \bTheta \right \Vert_{2, \bX}^2\right]\right \vert > \frac{2^{l-1}(2\epsilon^0)}{2} \right\}.\]
If $\mathcal{E}$ holds for some $\bTheta \in \mathcal{S}_{\bPi}$, there exists $l \in \mathbb{N}^*$ such that $\bTheta \in \mathcal{S}_{l,\bPi}$, thus there exists $l \in \mathbb{N}^*$ such that $\mathcal{E}_{l}$ holds, i.e., $\mathcal{E} \subset \underset{l \in \mathbb{N}^*}{\cup} \mathcal{E}_{l,\bPi}$. For $T>0$, let $\mathcal{S}_{\bPi}(T)$ be defined as follows:
\[\mathcal{S}_{\bPi}(T) = \left \{ \bTheta \in \underset{z \in \mathcal{Z}_{n,k}}{\cup}\mathcal{T}_z : \ \Vert\bTheta\Vert_{\infty}\leq \rho_n,\ \underset{i<j}{\min} \{\bTheta_{ij}\} \geq \gamma_n, \left \Vert \bTheta - \widetilde\bTheta \right \Vert_{2,\bPi}^2 \leq T \right\}.\]
We see that $\mathcal{S}_{l,\bPi} \subset \mathcal{S}_{\bPi}(2^l \epsilon^0)$, so we only need to control the probability of the events
\[\mathcal{E}(T) = \left\{ \exists \bTheta \in \mathcal{S}_{\bPi}(T): \left \vert \left \Vert  \bTheta - \widetilde \bTheta \right \Vert_{2, \bX}^2 - \mathbb{E} \left[\left \Vert  \bTheta - \widetilde \bTheta \right \Vert_{2, \bX}^2\right]\right \vert > \frac{T}{4}\right\}.\] 
The following lemma helps us bound the probability of the events $\mathcal{E}(T)$.
\begin{lem}\label{Lemma:ProbBeNPi}
 For $T>\epsilon^0$, let $Z_{T}=\underset{\bTheta\in \mathcal{S}_{\bPi}(T)}{\sup} \left \vert \left \Vert \bTheta - \widetilde \bTheta \right \Vert_{2, \bPi}^2 - \left \Vert  \bTheta - \widetilde \bTheta \right \Vert_{2, \bX}^2 \right \vert$. There exists an absolute constant $C>0$ such that
 \begin{equation*}
 \mathbb{P}\left( Z_{T} \geq \frac{T}{4}\right)\leq \exp \left(-\frac{CT}{\rho_n}\right).
 \end{equation*}
 \end{lem}
 \begin{proof}
 To prove Lemma \ref{Lemma:ProbBeNPi}, we first show that $Z_T$ concentrates around its expectation and then bound this term. 
\begin{lem}\label{ConcentrationNPi}
Let $Z_T$ be defined as in \ref{Lemma:ProbBeNPi}. Then
 $$\mathbb{P} \left(Z_T> 2\mathbb{E}[Z_T] + \frac{T}{16} \right) \leq \exp\left(-\frac{T}{64\rho_n}\right).$$
\end{lem}
 \begin{lem}\label{Lemma:esperancePi}
Let $Z_{T}$ be as in Lemma \ref{Lemma:ProbBeNPi}, then there exists an absolute constant $C>0$ such that 
\begin{align*}
\mathbb{E}\left[Z_{T}\right] \leq \frac{T}{16} + C\rho_n \frac{\rho_n^2}{\gamma_n^2}\left(n\log(k) + k^2 \right).
\end{align*}
 \end{lem}
 Putting together Lemma \ref{ConcentrationNPi} and Lemma \ref{Lemma:esperancePi}, we get that there exists an absolute constant $C>0$ such that 
\begin{align*}
\mathbb{P} \left( Z_{T} \geq \frac{3T}{16}+ \frac{C}{8}\rho_n \frac{\rho_n^2}{\gamma_n^2}\left(n\log(k) + k^2 \right)\right) \leq \exp\left(-\frac{T}{64\rho_n}\right).
\end{align*}
Our choice of $\epsilon_0$ allows us to conclude that for $T \geq 2\epsilon_0$, $\frac{C}{8}\rho_n \frac{\rho_n^2}{\gamma_n^2}\left(n\log(k) + k^2 \right) \leq \frac{T}{16}$ and
\begin{align*}
\mathbb{P} \left( Z_{T} \geq \frac{T}{4}\right) \leq \exp(-\frac{T}{64 \rho_n}).
\end{align*}\end{proof}
For this choice of $\epsilon_0$,
 \begin{equation*}
 \begin{split}
\mathbb P\left (\mathcal{E}\right )&\leq \underset{l=1}{\overset{\infty}{\Sigma}}\mathbb P\left (\mathcal{E}(2^l(2 \epsilon^0))\right )\\
 &\leq \underset{l=1}{\overset{\infty}{\Sigma}} \exp\left(-2C\epsilon^0 2^l/\rho_n \right)\\
 &\leq \underset{l=1}{\overset{\infty}{\Sigma}} \exp\left(-2Cl\log(2)\epsilon^0/\rho_n \right)\\
  &\leq \frac{\exp\left(-2C\log(2)\epsilon^0/\rho_n \right)}{1-\exp\left(-2C\log(2)\epsilon^0/\rho_n \right)} = \frac{1}{\exp\left(-2C\log(2)\epsilon^0/\rho_n \right) - 1} \leq 2 \exp\left(-C n\log(k) \right)
 \end{split}
 \end{equation*} 
 for $n$ large enough. This completes the proof of Lemma \ref{EqNPiNX}.
 
\subsubsection{Proof of Lemma \ref{ConcentrationNPi}}

To control the deviation of $Z_T$ from its expectation, we apply the following theorem from Bousquet, as stated in \cite{gine_nickl_2015}, Theorem 3.3.16.
\begin{thm}[Bousquet] \label{Bousquet}
Let $X_i$, $i \in \mathbb{N}$ be independent $\mathcal{S}$-valued random variables, and let $\cF$ be a countable class of functions $f = (f_1,..., f_n): \mathcal{S} \rightarrow [-1,1]^n$ such that $\mathbb{E}[f_i(X_i)] = 0$ for any $f \in \cF$ and $i \in [n]$. Set
$Z = \underset{f \in \cF}{\sup} \left \vert \underset{1 \leq i \leq n}{\sum} f_i(X_i) \right \vert$ and $v = \underset{f \in \cF}{\sup} \underset{1\leq i \leq n}{\sum} \mathbb{E} \left[f_i(X_i)^2\right]$. Then, for any $x>0$,
$$\mathbb{P} \left(Z > \mathbb{E}[Z] + \frac{x}{3} + \sqrt{2x(2\mathbb{E}[Z] + v)} \right) \leq \exp(-x).$$
\end{thm}
We apply Theorem \ref{Bousquet} to the random variable 
\begin{eqnarray*}
Z_{T} &=&\underset{\bTheta \in \mathcal{S}_{\bPi}(T)}{\sup} \left \vert \left \Vert \bTheta - \widetilde \bTheta \right \Vert_{2, \bPi}^2 - \left \Vert  \bTheta - \widetilde \bTheta \right \Vert_{2, \bX}^2 \right \vert\\
&=&\underset{\bTheta \in \mathcal{S}_{\bPi}(T)}{\sup} \left \vert \sum_{1 \leq i < j \leq n} \left(\bPi_{ij} - \bX_{ij}\right)\left( \bTheta_{ij} - \widetilde \bTheta_{ij} \right)^2 \right \vert\\
&=& \rho_n \underset{\bTheta\in \mathcal{S}_{\bPi}(T)}{\sup} \left \vert \sum_{1 \leq i < j \leq n} f^{\bTheta}_{ij}(\bX_{ij}) \right \vert
\end{eqnarray*}
where we set $f^{\bTheta}_{ij}(\bX_{ij}) \triangleq \frac{\left(\bX_{ij} - \bPi_{ij}\right)\left(\bTheta_{ij} - \widetilde \bTheta_{ij} \right)^2}{\rho_n}$. The set of functions $\left\{f^{\bTheta}_{ij}, \bTheta \in \mathcal{S}_{\bPi}(T) \right\}$ is separable and we can apply Theorem \ref{Bousquet} (see, e.g., \cite{gine_nickl_2015}, Section 2.1). Note that for any $1\leq i < j \leq n$, $\mathbb{E}\left[f^{\bTheta}_{ij}(\bX_{ij})\right] = 0$, $\left\vert f^{\bTheta}_{ij}(\bX_{ij})\right\vert \leq 1$, $\mathbb{E} \left[\left(\bX_{ij} - \bPi_{ij}\right)^2 \right] \leq \bPi_{ij}$ and $\left \vert\bTheta_{ij} - \widetilde \bTheta_{ij} \right \vert \leq \rho_n$ so 
\begin{eqnarray*}
\underset{\bTheta \in \mathcal{S}_{\bPi}(T)}{\sup} \underset{1\leq i<j \leq n}{\sum} \mathbb{E} \left[f^{\bTheta}_{ij}(X_{ij})^2\right] &\leq& \frac{1}{\rho_n^2}\underset{\bTheta \in \mathcal{S}_{\bPi}(T)}{\sup} \underset{1\leq i<j \leq n}{\sum}\bPi_{ij}\left(\bTheta_{ij} - \widetilde \bTheta_{ij} \right)^4\\
&\leq& \underset{\bTheta \in \mathcal{S}_{\bPi}(T)}{\sup} \underset{1\leq i<j \leq n}{\sum}\bPi_{ij}\left(\bTheta_{ij} - \widetilde \bTheta_{ij} \right)^2\\ 
&\leq& T.
\end{eqnarray*}
Theorem \ref{Bousquet} implies that 
\begin{eqnarray*}
\mathbb{P} \left(\frac{Z_T}{\rho_n} > \frac{1}{\rho_n}\mathbb{E}[Z_T] + \frac{x}{3} + \sqrt{2x\left(\frac{2}{\rho_n}\mathbb{E}[Z_T] + T\right)} \right) \leq \exp(-x) \\
\mathbb{P} \left(Z_T> \mathbb{E}[Z_T] + \frac{\rho_n x}{3} + \sqrt{2x\rho_n (2\mathbb{E}[Z_T] + \rho_nT)} \right) \leq \exp(-x)\\
\mathbb{P} \left(Z_T> \mathbb{E}[Z_T] + \frac{\rho_n x}{3} + 2\rho_nx + \mathbb{E}[Z_T] + \rho_n\sqrt{2xT} \right) \leq \exp(-x)
\end{eqnarray*}
where we have used $\sqrt{a+b} \leq \sqrt{a} + \sqrt{b}$ and $2\sqrt{ab} \leq a + b$. Setting $x = \frac{T}{64\rho_n}$ and noticing that $\rho_n \leq 1$ leads to 
\begin{eqnarray*}
\mathbb{P} \left(Z_T> 2\mathbb{E}[Z_T] + \frac{T}{16} \right) \leq \exp(-\frac{T}{64\rho_n}).
\end{eqnarray*}

\subsubsection{Proof of Lemma \ref{Lemma:esperancePi}}

Once we have bounded $Z_T$ by its expectation, we bound $\mathbb{E} \left[ Z_T \right]$. To do so, we use a symetrization argument and Talagrand's contraction principle (see, e.g., \cite{gine_nickl_2015} for a proof).
\begin{lem}[Symmetrization]\label{Lemma:symmetrization}
Let $\{\bY_i\}_{1 \leq i \leq n}$ be independent random variables, $\{\bepsilon_i\}_{1 \leq i \leq n}$ be a Rademacher sequence, and $\mathcal{A}$ be a subset of $\mathbb{R}^n$, then
\begin{align*}
\mathbb{E} \left[\underset{\bA \in \mathcal{A}}{\sup}\left \vert \sum_{1\leq i\leq n} (\bY_i - \mathbb{E}[\bY_i])\bA_i\right \vert \right] & \leq 2 \mathbb{E} \left[ \underset{\bA \in \mathcal{A}}{\sup} \left \vert \sum_{1\leq i\leq n} \bepsilon_i \bY_i \bA_i\right \vert \right].
\end{align*}
 \end{lem}
\begin{lem}[Talagrand's contraction principle]\label{Lemma:contraction}
 Let $\{\phi_i\}_{1\leq i \leq n}: \mathbb{R} \rightarrow \mathbb{R}$ be $1$-Lipshitz functions vanishing at $0$, $\mathcal{A}$ be a compact subset of $\mathbb{R}^n$ and $\{\bepsilon_i\}_{1\leq i \leq n}$ be a Rademacher sequence, then
$$\mathbb{E} \left[\underset{\bTheta \in \mathcal{A}}{\sup} \left \vert \underset{1 \leq i \leq n}{\sum} \bepsilon_i\phi_i\left(\bTheta_i\right)\right\vert\right] \leq  2\mathbb{E}\left[\underset{\bTheta \in \mathcal{A}}{\sup} \left \vert \underset{1 \leq i \leq n}{\sum} \bepsilon_i\bTheta_i\right\vert\right] .$$
 \end{lem}
 Recall that
\begin{eqnarray*}
\mathbb{E}\left[Z_{T}\right] &=& \mathbb{E} \left[\underset{\bTheta \in \mathcal{S}_{\bPi}(T)}{\sup} \left \vert \sum_{1 \leq i < j \leq n} \left(\bPi_{ij} - \bX_{ij}\right)\left( \bTheta_{ij} - \widetilde \bTheta_{ij} \right)^2 \right \vert\right].
\end{eqnarray*}
Let $(\bepsilon_{ij})_{1 \leq i<j\leq n}$ be a Rademacher sequence. Lemma  \ref{Lemma:symmetrization} implies
\begin{eqnarray*}
\mathbb{E}\left[Z_{T}\right] &\leq & 2 \mathbb{E} \left[\underset{\bTheta \in \mathcal{S}_{\bPi}(T)}{\sup} \left \vert \sum_{1 \leq i < j \leq n} \bepsilon_{ij}\bX_{ij}\left( \bTheta_{ij} - \widetilde \bTheta_{ij} \right)^2 \right \vert\right].
\end{eqnarray*}
For any $1 \leq i < j \leq n$, let $\phi_{ij}: x \rightarrow \frac{x^2}{2\rho_n}$. Note that on $[-\rho_n, \rho_n]$, $\phi_{ij}$ is a $1$-Lipschitz and vanishes at $0$. Applying Lemma \ref{Lemma:contraction}, we get that
\begin{eqnarray}\label{ici}
\mathbb{E}\left[Z_{T}\right] &\leq & 4\rho_n \mathbb{E} \left[\underset{\bTheta \in \mathcal{S}_{\bPi}(T)}{\sup} \left \vert \sum_{1 \leq i < j \leq n} \bepsilon_{ij}\phi_{ij}\left(\bX_{ij}\left( \bTheta_{ij} - \widetilde \bTheta_{ij}\right) \right) \right \vert\right]\nonumber\\
&\leq & 8\rho_n \mathbb{E} \left[\underset{\bTheta \in \mathcal{S}_{\bPi}(T)}{\sup} \left \vert \sum_{1 \leq i < j \leq n} \bepsilon_{ij}\bX_{ij}\left( \bTheta_{ij} - \widetilde \bTheta_{ij} \right) \right \vert\right]. \label{eqPdtSalairePi}
\end{eqnarray}
We bound the term in \label{eqPdtSalairePi} using the following lemma.
\begin{lem}\label{BornePdtScalairePi}
Let $\bB \in \{\bPi, \bX\}$ and let $\bSigma$ be a random matrix such that almost surely, $\left \Vert \bSigma \right \Vert_{\infty} \leq 1 $ and that conditionally on $\bB$, for any $1\leq i < j \leq n$ the coefficients $\bSigma_{ij}$ are independent and centered. Assume that there exists $\alpha>0$ such that for any $\bTheta \in \mathbb{R}^{n \times n}_{sym}$, $\underset{i<j}{\sum}\mathbb{E}^{\bB} \left[\bSigma_{ij}^2\bTheta_{ij}^2\right] \leq \alpha \left \Vert \bTheta \right \Vert_{2,\bB}^2.$ There exists an absolute constant $C$ such that
$$ \mathbb{E}^{\bB} \left[\underset{\bTheta \in \mathcal{S}_{\bB}(T)}{\sup} \left \vert \sum_{1 \leq i < j \leq n} \bSigma_{ij}\left( \bTheta_{ij} - \widetilde \bTheta_{ij} \right) \right \vert\right] \leq \frac{\gamma_n\alpha T}{32\times 64^2\rho_n^2} + C\frac{\rho_n^2}{\gamma_n}(n\log(k) + k^2).$$
\end{lem}
Note that for any $1 \leq i < j \leq n$, $\mathbb{E}[\bX_{ij}^2] \leq \bPi_{ij}$, so for any $\bTheta \in \mathbb{R}^{n \times n}_{sym}$, $\underset{i<j}{\sum}\mathbb{E} \left[\bepsilon_{ij}^2\bX_{ij}^2\bTheta_{ij}^2\right] \leq \left \Vert \bTheta \right \Vert_{2,\bPi}^2.$ We apply Lemma \ref{BornePdtScalairePi} with $\bB = \bPi$, $\alpha = 1$ and for any $1 \leq i <j \leq n$, $\bSigma_{ij} = \bepsilon_{ij}\bX_{ij}$ and combine it with \eqref{ici} to get that for some absolute constant $C$
\begin{eqnarray*}
\mathbb{E}\left[Z_{T}\right] &\leq & \frac{T\gamma_n\times 8 \rho_n}{32\times 64^2 \rho_n^2} + C\rho_n \frac{\rho_n^2}{\gamma_n^2}\left(n\log(k) + k^2 \right)\\
&\leq & \frac{T}{4\times64} + C\rho_n \frac{\rho_n^2}{\gamma_n }\left(n\log(k) + k^2 \right).
\end{eqnarray*}
This concludes the proof of Lemma \ref{Lemma:esperancePi}.

\subsubsection{Proof of Lemma \ref{BornePdtScalairePi}}

To get an upper bound on $\mathbb{E}^{\bB} \left[\underset{\bTheta \in \mathcal{S}_{\bB}(T)}{\sup} \left \vert \sum_{1 \leq i < j \leq n} \bSigma_{ij}\left( \bTheta_{ij} - \widetilde \bTheta_{ij} \right) \right \vert\right]$, we use Bernstein's inequality, which we state here for the reader's convenience: 
\begin{thm}[Bernstein's inequality]\label{thm:Bernstein}
Let $X_1, ..., X_n$ be independent centered random variables. Assume that for any $i \in [n]$,  $\vert X_i\vert \leq M$ almost surely, then
\begin{equation*}
\mathbb{P}\left(\left \vert \sum_{1 \leq i \leq n}X_i \right \vert \geq  \sqrt{2t\sum_{1 \leq i \leq n}\mathbb{E}[X_i^2]} + \frac{2M}{3}t\right) \leq 2e^{-t}.
\end{equation*}
\end{thm}
Recall that for $\bB \in \{\bPi, \bX \}$, $$\mathcal{S}_{\bB}(T) = \left \{ \bTheta \in \underset{z \in \mathcal{Z}_{n,k}}{\cup}\mathcal{T}_z : \ \Vert\bTheta\Vert_{\infty}\leq \rho_n,\ \underset{i<j}{\min} \{\bTheta_{ij}\} \geq \gamma_n, \left \Vert \bTheta - \widetilde\bTheta \right \Vert_{2,\bB}^2 \leq T \right\}$$ and let $\mathcal{S}_z(T) \triangleq \mathcal{T}_z \cap \mathcal{S}_{\bB}(T)$ be the set of matrices in $\mathcal{S}_{\bB}(T)$ that are block constant for the label $z$. Let $\widetilde\bTheta^z$ be the projection of $\widetilde \bTheta$ onto $\mathcal{T}_z$ for the $\bB$-weighted Frobenius norm:
$$\widetilde \bTheta^z \triangleq \underset{\bTheta \in \mathcal{T}_z }{\argmin} \left \Vert\bTheta - \widetilde \bTheta \right \Vert_{2,\bB}.$$ 
Note that if $\mathcal{S}_z(T) \not = \emptyset$, then $\widetilde\bTheta^z \in \mathcal{S}_z(T)$. If $\mathcal{S}_z(T) = \emptyset$, we set $\underset{\bTheta \in \mathcal{S}_z(T)}{\sup}\left \vert \left\langle\bSigma \vert \widetilde\bTheta^z- \bTheta\right\rangle \right \vert = 0$.
We decompose the error in two terms. 
\begin{equation}\label{eq:(I)(II)N2}
\begin{split}
\mathbb{E}^{\bB} \left[\underset{\bTheta \in \mathcal{S}_{\bB}(T)}{\sup} \left \vert \sum_{1 \leq i < j \leq n} \bSigma_{ij}\left( \bTheta_{ij} - \widetilde \bTheta_{ij} \right) \right \vert\right] & \leq \mathbb{E}^{\bB} \left[ \underset{z\in \mathcal{Z}_{n,k},\ \mathcal{S}_z(T) \not = \emptyset}{\sup}\left \vert \left\langle\bSigma \Big\vert \widetilde{\bTheta} - \widetilde\bTheta^z\right\rangle\right \vert\right]\\ 
& \ \text{ }\ + \  \mathbb{E}^{\bB} \left[ \underset{z\in \mathcal{Z}_{n,k}}{\sup}\underset{\bTheta \in \mathcal{S}_z(T)}{\sup}\left \vert \left\langle\bSigma \Big\vert \widetilde\bTheta^z- \bTheta\right\rangle \right \vert\right]\\
 & \leq  (I) + (II).
\end{split}
\end{equation}
The term $(I)$ denotes $\mathbb{E}^{\bB} \left[ \underset{z\in \mathcal{Z}_{n,k},\ \mathcal{S}_z(T) \not = \emptyset}{\sup}\left \vert \left\langle\bSigma \Big\vert \widetilde{\bTheta} - \widetilde\bTheta^z\right\rangle\right \vert\right]$ and corresponds to the error induced by an error on the label. The term $(II)$ denotes $\mathbb{E}^{\bB} \left[ \underset{z\in \mathcal{Z}_{n,k}}{\sup}\underset{\bTheta \in \mathcal{S}_z(T)}{\sup}\left \vert \left\langle\bSigma \Big\vert \widetilde\bTheta^z- \bTheta\right\rangle \right \vert\right]$ and corresponds to the error induced by a Bernoulli noise.

\underline{Control of (I):} To control the first term of \eqref{eq:(I)(II)N2}, recall that for any $z \in \mathcal{Z}_{n,k}$ such that $\mathcal{S}_z(T) \not = \emptyset$, $\widetilde{ \bTheta}^z \in \mathcal{S}_z(T)$ and by hypothesis,  $\underset{i<j}{\sum}\mathbb{E}^{\bB} \left[\bSigma_{ij}^2\left(\widetilde \bTheta_{ij} - \widetilde{ \bTheta}^z_{ij}\right)^2\right] \leq \alpha\left \Vert \widetilde\bTheta - \widetilde{ \bTheta}^z \right \Vert_{2,\bB}^2 \leq \alpha T $. Furthermore, $\left \Vert \bSigma \right  \Vert_{\infty} \leq 1$ so $\left \vert \bSigma_{ij}(\widetilde\bTheta_{ij} - \widetilde{ \bTheta}^z_{ij})\right \vert \leq \rho_n$. Since $\left \vert \mathcal{Z}_{n,k} \right \vert \leq n\log(k)$, the union bound and Bernstein's inequality imply
\begin{eqnarray*}
\mathbb{P}^{\bB}\left(\underset{z\in \mathcal{Z}_{n,k},\ \mathcal{S}_z(T) \not = \emptyset}{\sup}\left \vert \left\langle\bSigma \vert \widetilde \bTheta - \widetilde\bTheta^z\right\rangle \right\vert \geq  \sqrt{2\alpha T\left(t + n\log(k)\right)} + \frac{2\rho_n}{3}\left(t + n\log(k)\right)\right) &\leq& 2e^{-t}\\
\mathbb{P}^{\bB}\left(\underset{z\in \mathcal{Z}_{n,k},\ \mathcal{S}_z(T) \not = \emptyset}{\sup}\left \vert \left\langle\bSigma \vert \widetilde \bTheta - \widetilde\bTheta^z\right\rangle \right\vert \geq  \frac{\gamma_n\alpha T}{64^3\rho_n^2} + \left(\frac{2\rho_n}{3} + \frac{64^3\rho_n^2}{\gamma_n} \right)\left(t + n\log(k)\right)\right) &\leq& 2e^{-t}.
\end{eqnarray*}
Integrating the last inequality and using $\frac{\rho_n}{\gamma_n} \geq 1$, we get that for some absolute constant $C$
\begin{equation}\label{cardit}
(I) \leq  \frac{\alpha\gamma_n T}{64^3\rho_n^2} + C \frac{\rho_n^2}{\gamma_n} n\log(k).
\end{equation}
\underline{Control of (II):} The control of the second term of \eqref{eq:(I)(II)N2} is more involved. We adapt the argument developped in \cite{KloppGraphon} and consider only $z \in \cZ_{n,k}$ such that $\mathcal{S}_z(T) \not = \emptyset$. By property of the projection, we have for any $\bTheta \in \mathcal{S}_z(T)$, $ \left \Vert \bTheta - \widetilde \bTheta^z\right \Vert_{2,\bB}^2 \leq  \left \Vert \bTheta - \widetilde \bTheta \right \Vert_{2,\bB}^2 \leq T$. Thus $\left( \widetilde\bTheta^z- \bTheta \right) \in \mathcal{A}_z(T)$, where $$\mathcal{A}_z(T) \triangleq \left \{\bTheta \in \mathcal{T}_z: \left \Vert \bTheta \right \Vert_{\infty} \leq \rho_n,\ \left \Vert \bTheta \right \Vert_{2,\bB}^2 \leq T \right \},$$ so $\underset{\bTheta \in \mathcal{S}_z(T)}{\sup}\left \vert \left\langle\bSigma \vert \widetilde\bTheta^z- \bTheta\right\rangle  \right \vert \leq \underset{\bTheta \in \mathcal{A}_z(T)}{\sup}\left \vert \left\langle\bSigma \vert \bTheta\right\rangle  \right \vert.$ Let $\widehat{\bT}^z \in \mathcal{A}_z(T)$ be such that 
\begin{equation}\label{defTz}
\left \vert \left\langle\bSigma \vert \widehat{\bT}^z\right\rangle  \right \vert \triangleq \underset{\bTheta \in \mathcal{A}_z(T)}{\sup}\left \vert \left\langle\bSigma \vert \bTheta\right\rangle  \right \vert.
\end{equation}
Note that $\bTheta \rightarrow \left \vert \left\langle \bSigma \vert \bTheta \right\rangle \right \vert$ is continous and reaches its supremum on $\mathcal{A}_z(T)$. Indeed, either for any $1\leq i<j\leq n$, $\bB_{ij} >0$ so $\left \Vert \cdot \right \Vert_{\bB}$ is a norm and $\mathcal{A}_z(T)$ is compact, or we can find a subspace $\mathcal{V}$ of $\mathbb{R}^{\frac{(n-1)(n-2)}{2}}$ of dimension $\left \vert \left\{ 1 \leq i < j \leq n: \bB_{ij} >0 \right\}\right \vert$ such that for any $\bTheta \in \mathbb{R}^{\frac{(n-1)(n-2)}{2}}$, $\langle \bSigma \vert \bTheta \rangle = \langle \bSigma \vert \mathcal{P}_{\mathcal{V}}(\bTheta) \rangle$ where $\mathcal{P}_{\mathcal{V}}$ denotes the projection onto $\mathcal{V}$, and $\mathcal{A}_z(T) \cap \mathcal{V}$ is compact.

To control $\left \vert \left\langle\bSigma \vert \widehat{\bT}^z\right\rangle  \right \vert$, we build a finite set with small cardinality that approximates $\widehat{\bT}^z$ well both in the weighted Frobenius norm and in the supremum norm. More precisely, our goal is to construct a finite set $\widetilde{\mathcal{C}}_z(T)$ containing a matrix $\widehat \bV$ such that $2\left(\widehat{\bT}^z - \widehat \bV\right) \in \mathcal{A}_z(T)$. To apply Bernstein's inequality, we also need to be able to control the suppremum norm on this set. Our first step will be to construct such a set. 

We denote by $\mathcal{B}_r$ the ball centered at $\textbf{0}$ and of radius $r$ for the weighted Frobenius norm $\left \Vert \cdot \right \Vert_{2, \bB}$. Let $\mathcal{C}_z$ be a minimal $\sqrt{T}/2$-net for the weighted Frobenius norm on $ \mathcal{B}_{\sqrt{T}} \cap \mathcal{T}_z$.  Note that $\mathcal{A}_z \subset \mathcal{B}_{\sqrt{T}} \cap \mathcal{T}_z$, so there exists $\widehat{\bV} \in \mathcal{C}_z(T)$ such that $\left \Vert \widehat{\bV} - \widehat{\bT}^z \right \Vert_{2, \bB} \leq \frac{\sqrt{T}}{2}$. Since our choice of net does not allow us to directly bound $\left \Vert \widehat\bV - \widehat{\bT}^z \right \Vert_{\infty}$, we extend this net using the following argument. For any $\bV \in \mathcal{C}_z$ and any matrix $\bU \in \{-1,0,1\}^{k \times k}$, let $\bV^{\bU} \in \mathbb{R}^{n\times n}$ be such that $\bV^{\bU}_{ii} = 0$ and for any $i<j$, 
$$\bV^{\bU}_{ij} = \operatorname{\sign}(\bV_{ij})\left(\vert\bV_{ij} \vert\land \rho_n \right)\left( 1 - \left \vert \bU_{z(i) z(j)} \right \vert\right) + \bU_{z(i) z(j)}\frac{\rho_n}{2}.$$
Recall that $\left \Vert \widehat{\bT}^z \right \Vert_{\infty} \leq \rho_n$ so for any $\bV \in \mathcal{C}_z(T)$ we have $\left \vert \operatorname{\sign}(\bV_{ij})\left(\vert\bV_{ij} \vert\land \rho_n \right) - \widehat{\bT}^z_{ij} \right \vert \leq \left \vert \bV_{ij} - \widehat{\bT}^z_{ij} \right \vert$. This implies that $\left \Vert \bV^{\textbf{0}}- \widehat{\bT}^z \right \Vert_{2,\bB} \leq \left \Vert \bV- \widehat{\bT}^z \right \Vert_{2, \bB}.$

Now, let $ \widetilde{\mathcal{C}}_z(T)= \left\{ \bV^{\bU}: \bV \in \mathcal{C}_z(T), \bU \in \{-1,0,1\}^{k \times k}_{sym} \right\}$ and $\widehat{\bU} = \underset{\bU \in \{-1,0,1\}^{k \times k}}{\operatorname{\argmin}} \left \Vert \widehat{\bV}^{\bU} - \widehat{\bT}^z \right \Vert_{\infty}$. By definition, for any $(a,b) \in k \times k$, $\widehat{\bU}$ minimises $\left \vert \widehat{\bV}^{\bU}_{z^{-1}(a) z^{-1}(b)} - \widehat{\bT}^z_{z^{-1}(a) z^{-1}(b)} \right \vert$, so it is also a minimizer of $\left \Vert \widehat{\bV}^{\bU} - \widehat{\bT}^z \right \Vert_{2,\bB} = \underset{a,b \in [k]}{\sum} \left(\underset{(i,j) \in z^{-1}(a) \times z^{-1}(b), i \neq j}{\sum}\bB_{ij} \right) \left \vert \widehat{\bV}^{\bU}_{z^{-1}(a) z^{-1}(b)} - \widehat{\bT}^z_{z^{-1}(a) z^{-1}(b)} \right \vert^2$. Therefore $$\left \Vert \widehat{\bV}^{\widehat\bU}- \widehat{\bT}^z \right \Vert_{2,\bB} \leq \left \Vert \widehat{\bV}^{\textbf{0}}- \widehat{\bT}^z \right \Vert_{2,\bB} \leq \left \Vert \widehat{\bV}- \widehat{\bT}^z \right \Vert_{2,\bB} \leq\frac{\sqrt{T}}{2}.$$ 
Furthermore $\left \Vert \widehat{\bV}^{\widehat\bU}- \widehat{\bT}^z \right \Vert_\infty \leq \left \Vert \widehat{\bV}^{\bU^*}- \widehat{\bT}^z \right \Vert_\infty$, where $\bU^*_{ab} = \operatorname{\sign}(\widehat{\bT}^z_{z^{-1}(a)z^{-1}(b)})$. By construction, $$\left \Vert \widehat{\bV}^{\bU^*}- \widehat{\bT}^z \right \Vert_\infty = \sup_{i<j} \left \vert  \widehat{\bT}^z_{ij} - \mbox{sign}(\widehat{\bT}^z_{ij}) \frac{\rho_n}{2}\right \vert = \sup_{i<j} \left \vert \left \vert \widehat{\bT}^z_{ij} \right \vert - \frac{\rho_n}{2}\right \vert \leq \frac{\rho_n}{2}.$$ Hence, $2\left(\widehat{\bT}^z - \widehat{\bV}^{\hat\bU}\right) \in \mathcal{A}_z(T)$. Thus, we have shown that
\begin{eqnarray*}\label{arguSup}
2 \left \vert \left\langle\bSigma \vert \widehat{\bT}^z - \widehat{\bV}^{\hat\bU} \right\rangle  \right \vert &\leq &\underset{\bTheta \in \mathcal{A}_z(T)}{\sup}\left \vert \left\langle\bSigma \vert \bTheta\right\rangle  \right \vert \triangleq \left \vert \left\langle\bSigma \vert \widehat{\bT}^z\right\rangle  \right \vert\\
2\left \vert \left\langle\bSigma \vert \widehat{\bT}^z \right\rangle  \right \vert - 2\left \vert \left\langle\bSigma \vert \widehat{\bV}^{\hat\bU}\right\rangle  \right \vert  & \leq &\left \vert \left\langle\bSigma \vert \widehat{\bT}^z\right\rangle  \right \vert \\ 
\left \vert \left\langle\bSigma \vert \widehat{\bT}^z \right\rangle  \right \vert  &\leq& 2\left \vert \left\langle\bSigma \vert \widehat{\bV}^{\widehat\bU}\right\rangle \right \vert .
\end{eqnarray*}
This and \eqref{defTz} allows us to conclude that 
\begin{equation} \label{epNet1}
\underset{z\in \mathcal{Z}_{n,k}}{\sup}\underset{\bTheta \in \mathcal{S}_z(T)}{\sup}\left \vert \left\langle\bSigma \vert \widetilde\bTheta^z - \bTheta\right\rangle \right \vert \leq 2 \underset{z\in \mathcal{Z}_{n,k}}{\sup} \underset{\bV \in \widetilde{\mathcal{C}}_{z}(T)}{\sup} \left \vert \left\langle\bSigma \vert \bV \right\rangle  \right \vert.
\end{equation}
To bound the right hand side of \eqref{epNet1}, we recall that by hypothesis for any $\bV \in \widetilde{\mathcal{C}}_{z}(T)$,  $\underset{i<j}{\sum}\mathbb{E}^{\bB} \left[\bSigma_{ij}^2\bV_{ij}^2\right] \leq \alpha \left \Vert \bV\right \Vert_{2,\bB}^2$ and note that $\left \Vert \bV \right \Vert_{\infty}\leq \rho_n$ and $\left \Vert \bV \right \Vert_{2,\bB}\leq \sqrt{T}$. We use Bernstein's inequality and the union bound to obtain
\begin{eqnarray}\label{eq:cardi1}
\mathbb{P}\left(\underset{z\in \mathcal{Z}_{n,k}}{\sup} \underset{\bV \in \widetilde{\mathcal{C}}_{z}(T)}{\sup} \left \vert \left\langle\bSigma \vert \bV \right\rangle  \right \vert \geq  \sqrt{2\alpha Tt} + \frac{2}{3}t\right) &\leq& 2e^{-t + n\log(k) + \underset{\bV \in \widetilde{\mathcal{C}}_{z}(T)}{\sup} \log \left(\left \vert \widetilde{\mathcal{C}}_{z}(T) \right \vert\right)}.
\end{eqnarray}
By construction of $\widetilde{\mathcal{C}}_{z}(T)$, we have $\left \vert \widetilde{\mathcal{C}}_{z}(T)\right \vert = \left \vert \mathcal{C}_{z}(T)\right \vert \times 3^{k^2}$. The following classical result on the covering number of a ball will help us bound $\left \vert \mathcal{C}_{z}(T)\right \vert$ (see, e.g., Lemma 5.2 in \cite{vershynin}).
\begin{lem}\label{Lemma:coveringNumber}
Let $\mathcal{B}_r$ the ball of a subspace of $\mathbb{R}^{n}$ of dimension $d$ centered at $\textbf{0}$ and of radius $r$ for the euclidean norm, and $\mathcal{N}\left(\mathcal{B}_r,\epsilon\right)$ its $\epsilon$-covering number, that is the minimal cardinality of a set $\mathcal{C}$ such that for any $\bX \in \mathcal{B}_r$, there exists $\bY \in  \mathcal{C}$ such that $\left \Vert \bX - \bY \right \Vert \leq \epsilon$. Then 
$$\mathcal{N}\left(\mathcal{B}_r,\epsilon\right) \leq \left(\frac{3r}{\epsilon}\right)^{d}.$$ 
\end{lem}
Extending the proof Lemma \ref{Lemma:coveringNumber} to a weighed euclidean norm is straightforward. Putting Lemma \ref{Lemma:coveringNumber} into equation \eqref{eq:cardi1} and noting that $\mathcal{T}_z$ spans a subspace of $\mathbb{R}^{\frac{(n-1)(n-2)}{2}}$ of dimension $\frac{k(k-1)}{2}$, we get that for some absolute constant $C$
\begin{eqnarray*}
\mathbb{P}\left(\underset{z\in \mathcal{Z}_{n,k}}{\sup} \underset{\bV \in \widetilde{\mathcal{C}}_{z}(T)}{\sup} \left \vert \left\langle\bSigma \vert \bV \right\rangle  \right \vert \geq  \sqrt{2\alpha T\left( t + n\log(k) + k^2\log(C)\right)} + \frac{2\rho_n}{3}\left(t + n\log(k) + k^2\log(C)\right)\right) &\leq &2e^{-t} \nonumber \\
\mathbb{P}\left(\underset{z\in \mathcal{Z}_{n,k}}{\sup} \underset{\bV \in \widetilde{\mathcal{C}}_{z}(T)}{\sup} \left \vert \left\langle\bSigma \vert \bV \right\rangle  \right \vert \geq \frac{\alpha\gamma_n T}{2 \times 64^3\rho_n^2} + \left(\frac{2\rho_n}{3} + \frac{2\times 64^2\rho_n^2}{\gamma_n} \right)\left(t + n\log(k) + k^2\log(C)\right)\right) \leq 2e^{-t}.
\end{eqnarray*}
 We integrate and find for some absolute constant $C>0$
\begin{equation}\label{cardis}
\mathbb{E}\left[
\underset{z\in \mathcal{Z}_{n,k}}{\sup}\underset{\bTheta \in \mathcal{S}_z(T)}{\sup}\left \vert \left\langle\bSigma \vert \widetilde{\bTheta}^z- \bTheta\right\rangle  \right \vert \right]  \leq  \frac{\alpha \gamma_n T}{64^3\rho_n^2}  + C\frac{\rho_n^2}{\gamma_n}\left(n\log(k) + k^2\right).
\end{equation}
Combining the bounds \eqref{cardis} and \eqref{cardit} yields the desired result.

\subsection{Proof of Lemma \ref{EqKl2KlPi}}
 The proof of Lemma \ref{EqKl2KlPi} closely follows that of Lemma \ref{EqNPiNX} and we only sketch it. Recall that $\epsilon_n \triangleq C\frac{\rho_n^2}{\gamma_n^2}\left(n\log(k) + k^2 \right)$ where the absolute constant $C$ is larger than the constant appearing in Lemma \ref{Lemma:ProbBeKLPi}, and that $\epsilon^0 \triangleq \rho_n\epsilon_n$. We show that the probability of the following ``bad" event is small:
\[\mathcal{E} \triangleq \left\{ \exists \bTheta \in \mathcal{S}_{\bPi}:\left \vert \Delta \cK^{\bTheta^*}_{\bPi}(\bTheta, \widetilde\bTheta) - \Delta \cK^{\bTheta^*}_{\bX}(\bTheta, \widetilde\bTheta) \right \vert > \frac{1}{2\times32\rho_n} \left \Vert\bTheta - \widetilde \bTheta \right \Vert_{2, \bPi}^2 + \epsilon_n \right\}.\]
Again, we slice $\mathcal{S}_{\bPi}$ in different sets $\mathcal{S}_{l,\bPi}$ defined as
$\mathcal{S}_{l,\bPi} \triangleq \left \{\bTheta \in \mathcal{S}_{\bPi}:  32^{l-1}(2\epsilon^0) \leq \left \Vert \bTheta - \widetilde\bTheta \right \Vert_{2,\bPi}^2 \leq 32^l(2\epsilon^0)\right \}$
on which we control the events $\mathcal{E}_{l} \triangleq \left\{ \exists \bTheta \in \mathcal{S}_{l,\bPi}:\left \vert \Delta \cK^{\bTheta^*}_{\bPi}(\bTheta, \widetilde\bTheta) - \Delta \cK^{\bTheta^*}_{\bX}(\bTheta, \widetilde\bTheta) \right \vert > \frac{32^{l-1}\times 2\epsilon^0}{4\times 32\rho_n} + \epsilon_n \right\}.$
To do this, we set $\mathcal{S}_{\bPi}(T) \triangleq \left \{\bTheta \in \mathcal{S}_{\bPi}: \left \Vert \bTheta - \widetilde\bTheta \right \Vert_{2,\bPi}^2 \leq T\right \}$ and we control the probability of the events
\[\mathcal{E}(T) = \left\{ \exists \bTheta \in \mathcal{S}_{\bPi}(T): \left \vert \Delta \cK^{\bTheta^*}_{\bPi}(\bTheta, \widetilde\bTheta) - \Delta \cK^{\bTheta^*}_{\bX}(\bTheta, \widetilde\bTheta) \right \vert > \frac{T}{64^2\rho_n} +\epsilon_n \right\}.\] 
The following lemma helps us bound the probability of the events $\mathcal{E}(T)$.
\begin{lem}\label{Lemma:ProbBeKLPi}
 Let $\widetilde Z_{T}=\underset{\bTheta\in \mathcal{S}_{\bPi}(T)}{\sup} \left \vert \Delta \cK^{\bTheta^*}_{\bPi}(\bTheta, \widetilde\bTheta) - \Delta \cK^{\bTheta^*}_{\bX}(\bTheta, \widetilde\bTheta) \right \vert$. There exists two absolute constants $C$, $C' >0$ such that
 \begin{equation*}
 \mathbb{P}\left( \widetilde Z_{T} \geq \frac{T}{64^2\rho_n} +  C\frac{\rho_n^2}{\gamma_n^2}\left(n\log(k) + k^2 \right)\right) \leq \exp\left(-\frac{C'T\gamma_n^2}{\rho_n^2}\right).
 \end{equation*}
 \end{lem}
 \begin{proof}
 To prove Lemma \ref{Lemma:ProbBeKLPi}, we first show that $Z_T$ concentrates around its expectation and then bound this term. 
\begin{lem}\label{ConcentrationKlPi}
Let $\widetilde Z_T$ be defined as in Lemma \ref{Lemma:ProbBeKLPi}. Then there exists an absolute constant $C>0$ such that 
$$\mathbb{P} \left(\widetilde Z_T> 2\mathbb{E}[\widetilde Z_T] + \frac{T}{2\times 64^2\rho_n} \right) \leq \exp\left(-\frac{CT\gamma_n^2}{\rho_n^2}\right).$$
\end{lem}
 \begin{lem}\label{Lemma:esperanceKlPi}
Let $\widetilde Z_{T}$ be as in Lemma \ref{Lemma:ProbBeKLPi}, then there exists an absolute constant $C>0$ such that 
\begin{align}\label{eqEsperance}
\mathbb{E}\left[\widetilde Z_{T}\right] \leq \frac{T}{4\times 64^2\rho_n} + C\frac{\rho_n^2}{\gamma_n^2}\left(n\log(k) + k^2 \right).
\end{align}
 \end{lem}
Putting together Lemma \ref{ConcentrationKlPi} and Lemma \ref{Lemma:esperanceKlPi}, we get that there exists two absolute constants $C,C'>0$ such that
 $$\mathbb{P} \left( \widetilde Z_{T} \geq \frac{T}{64^2\rho_n}+ C\frac{\rho_n^2}{\gamma_n^2}(n\log(k) + k^2)\right) \leq \exp\left(-\frac{C'T\gamma_n^2}{\rho_n^2}\right).$$
This concludes the proof of Lemma \ref{Lemma:ProbBeKLPi}.
 
 \end{proof}
Lemma \ref{Lemma:ProbBeKLPi} and the arguments developped to prove Lemma \ref{EqNPiNX} help us conclude the proof of Lemma \ref{EqKl2KlPi}.
 
\subsubsection{Proof of Lemma \ref{ConcentrationKlPi}}

Recall that by definition of $\widetilde Z_{T}$,
\begin{eqnarray*}
\widetilde Z_{T} &=& \underset{\bTheta\in \mathcal{S}_{\bPi}(T)}{\sup} \left \vert \sum_{1\leq i<j \leq n} \left(\bPi_{ij} - \bX_{ij} \right) \left( \bTheta^*_{ij} \log \left(\frac{\widetilde \bTheta_{ij}}{\bTheta_{ij}} \right) + (1 - \bTheta^*_{ij}) \log \left(\frac{1 - \widetilde \bTheta_{ij}}{1 - \bTheta_{ij}} \right) \right)\right \vert\\
&=& \frac{1}{\gamma_n} \underset{\bTheta\in \mathcal{S}_{\bPi}(T)}{\sup} \left \vert \sum_{1 \leq i < j \leq n} f^{\bTheta}_{ij}(\bX_{ij}) \right \vert
\end{eqnarray*}
where we set $f^{\bTheta}_{ij}(\bX_{ij}) \triangleq \gamma_n\left(\bPi_{ij} - \bX_{ij} \right) \left( \bTheta_{ij}^* \log \left(\frac{\widetilde \bTheta_{ij}}{\bTheta_{ij}} \right) + (1 - \bTheta_{ij}^*) \log \left(\frac{1 - \widetilde \bTheta_{ij}}{1 - \bTheta_{ij}} \right) \right)$. Assuming that $\gamma_n \leq 1 - \rho_n$, $x \rightarrow \log(x)$ and $x \rightarrow \log(1-x)$ are $\frac{1}{\gamma_n}$ - Lipshitz on $[\gamma_n, \rho_n]$ so 
$$\left\vert \bTheta_{ij}^* \log \left(\frac{\widetilde \bTheta_{ij}}{\bTheta_{ij}} \right) + (1 - \bTheta_{ij}^*) \log \left(\frac{1 - \widetilde \bTheta_{ij}}{1 - \bTheta_{ij}} \right) \right\vert \leq\bTheta^* \frac{\left \vert \bTheta_{ij} - \widetilde \bTheta_{ij}\right \vert}{\gamma_n} + (1 - \bTheta^*) \frac{\left \vert \bTheta_{ij} - \widetilde \bTheta_{ij}\right \vert}{\gamma_n} \leq \frac{\left \vert \bTheta_{ij} - \widetilde \bTheta_{ij}\right \vert}{\gamma_n}$$
which implies that for any $1\leq i < j \leq n$, $\left\vert f^{\bTheta}_{ij}(\bX_{ij})\right\vert \leq 1$. Moreover for any $1\leq i < j \leq n$, $\mathbb{E}\left[f^{\bTheta}_{ij}(\bX_{ij})\right] = 0$ and $\mathbb{E} \left[\left(\bX_{ij} - \bPi_{ij}\right)^2 \right] \leq \bPi_{ij}$, hence
\begin{eqnarray*}
\underset{\bTheta \in \mathcal{S}_{\bPi}(T)}{\sup} \underset{1\leq i<j \leq n}{\sum} \mathbb{E} \left[f^{\bTheta}_{ij}(X_{ij})^2\right] &\leq& \gamma_n^2\underset{\bTheta \in \mathcal{S}_{\bPi}(T)}{\sup} \underset{1\leq i \leq n}{\sum}\bPi_{ij}\left( \bTheta_{ij}^* \log \left(\frac{\widetilde \bTheta_{ij}}{\bTheta_{ij}} \right) + (1 - \bTheta_{ij}^*) \log \left(\frac{1 - \widetilde \bTheta_{ij}}{1 - \bTheta_{ij}} \right) \right)^2\\
&\leq& \gamma_n^2 \underset{\bTheta \in \mathcal{S}_{\bPi}(T)}{\sup} \underset{1\leq i \leq n}{\sum}\bPi_{ij} \frac{1}{\gamma_n^2}\left( \bTheta_{ij} - \widetilde\bTheta_{ij}\right)^2\\
&\leq& T.
\end{eqnarray*}
Then, Theorem \ref{Bousquet} implies
\begin{eqnarray*}
\mathbb{P} \left(\gamma_n \widetilde Z_T > \gamma_n\mathbb{E}[ \widetilde Z_T] + \frac{x}{3} + \sqrt{2x(2\gamma_n\mathbb{E}[ \widetilde Z_T] + T)} \right) \leq \exp(-x) \\
\mathbb{P} \left( \widetilde Z_T> \mathbb{E}[ \widetilde Z_T] + \frac{ x}{3\gamma_n} + \frac{ 4x}{\gamma_n} + \mathbb{E}[ \widetilde Z_T] + \frac{2x\times 4 \times 64^2\rho_n}{\gamma_n^2} + \frac{T}{4 \times 64^2\rho_n} \right) \leq \exp(-x)\\
\mathbb{P} \left( \widetilde Z_T> 2 \mathbb{E}[ \widetilde Z_T] + \frac{ 9\times 64^2 x\rho_n}{\gamma_n^2} + \frac{T}{4 \times 64^2\rho_n} \right) \leq \exp(-x)
\end{eqnarray*}
where we have used $\sqrt{a+b} \leq \sqrt{a} + \sqrt{b}$, $\sqrt{ab} \leq a + b$ and $\frac{\rho_n}{\gamma_n}\geq 1$. Setting $x = \frac{T\gamma_n^2}{9\times 64^2 \times 4 \times 64^2 \rho_n^2}$ yields the desired result.

\subsubsection{Proof of Lemma \ref{Lemma:esperanceKlPi}}

In Lemma \ref{Lemma:esperanceKlPi}, we bound
\begin{eqnarray*}
\mathbb{E}\left[ \widetilde  Z_{T}\right] &=& \mathbb{E} \left[ \underset{\bTheta\in \mathcal{S}_{\bPi}(T)}{\sup} \left \vert \sum_{1\leq i<j \leq n} \left(\bPi_{ij} - \bX_{ij} \right) \left( \bTheta_{ij}^* \log \left(\frac{\widetilde \bTheta_{ij}}{\bTheta_{ij}} \right) + (1 - \bTheta^*) \log \left(\frac{1 - \widetilde \bTheta_{ij}}{1 - \bTheta_{ij}} \right) \right)\right \vert\right].
\end{eqnarray*}
Let $(\bepsilon_{ij})_{1 \leq i<j\leq n}$ be a Rademacher sequence. We apply Lemma  \ref{Lemma:symmetrization} and get
\begin{eqnarray*}
\mathbb{E}\left[  \widetilde Z_{T}\right] &\leq & 2 \mathbb{E} \left[ \underset{\bTheta\in \mathcal{S}_{\bPi}(T)}{\sup} \left \vert \sum_{1\leq i<j \leq n} \bepsilon_{ij} \bX_{ij} \left( \bTheta_{ij}^* \log \left(\frac{\widetilde \bTheta_{ij}}{\bTheta_{ij}} \right) + (1 - \bTheta_{ij}^*) \log \left(\frac{1 - \widetilde \bTheta_{ij}}{1 - \bTheta_{ij}} \right) \right)\right \vert\right].
\end{eqnarray*}
For any $1 \leq i < j \leq n$, let $\phi_{ij}: x \rightarrow \frac{\gamma_n}{2\rho_n}\bX_{ij}\left(\bTheta^*_{ij} \log\left(\frac{\widetilde\bTheta_{ij} - x}{\widetilde\bTheta_{ij}}\right) + (1 - \bTheta^*_{ij}) \log\left(\frac{1 + x - \widetilde\bTheta_{ij}}{1 - \widetilde\bTheta_{ij}}\right) \right)$. Note that on $[\widetilde\bTheta_{ij} - \rho_n, \widetilde\bTheta_{ij} - \gamma_n]$, $\phi_{ij}$ is $1$-Lipschitz and vanishes at $0$. Then we apply Lemma \ref{Lemma:contraction} and compute 
\begin{eqnarray*}
\mathbb{E}\left[  \widetilde  Z_{T}\right] &\leq & \frac{4\rho_n}{\gamma_n} \mathbb{E} \left[\underset{\bTheta \in \mathcal{S}_{\bPi}(T)}{\sup} \left \vert \sum_{1 \leq i < j \leq n} \bepsilon_{ij}\phi_{ij} \left( \bX_{ij}\left( \bTheta_{ij} - \widetilde \bTheta_{ij} \right) \right)\right \vert\right]\\
&\leq & \frac{8\rho_n}{\gamma_n} \mathbb{E} \left[\underset{\bTheta \in \mathcal{S}_{\bPi}(T)}{\sup} \left \vert \sum_{1 \leq i < j \leq n} \bepsilon_{ij}\bX_{ij}\left( \bTheta_{ij} - \widetilde \bTheta_{ij} \right) \right \vert\right].
\end{eqnarray*}
Now, applying Lemma \ref{BornePdtScalairePi} with $\alpha = 1$ and $\bB= \bPi$ allows us to conclude that there exists an absolute constant $C>0$ such that $$\mathbb{E}\left[ \widetilde  Z_{T}\right] \leq \frac{T}{8\times 64^2\rho_n} + C\frac{\rho_n^3}{\gamma_n^2}(n\log(n) + k^2) \leq \frac{T}{8\times 64^2\rho_n} + C\frac{\rho_n^2}{\gamma_n^2}(n\log(n) + k^2).$$

 \subsection{Proof of Lemma \ref{ControlKlX}}
 \label{subsection:proofbad_event}
 
 The proof of Lemma \ref{ControlKlX} closely follows that of Lemma \ref{EqKl2KlPi}, and we only sketch it. Recall that $\epsilon_n \triangleq C\frac{\rho_n^2}{\gamma_n^2}\left(n\log(k) + k^2 \right)$ where the absolute constant $C$ is larger than the constant appearing in Lemma \ref{Lemma:ProbBlKLX}, and that $\epsilon^0 \triangleq \rho_n\epsilon_n$. We show that conditionally on $\bX$, the probability of the following ``bad" event is small and does not depend on $\bX$:
\[\mathcal{E}_{\bX} \triangleq \left\{ \exists \bTheta \in \mathcal{S}_{\bX}: \left \vert \Delta \cK^{\bTheta^*}_{\bX}(\bTheta, \widetilde\bTheta) - \Delta \cK^{\bA}_{\bX}(\bTheta, \widetilde \bTheta) \right \vert > \frac{1}{2\times 64 \rho_n} \left \Vert \bTheta - \bTheta \right \Vert_{2,\bX}^2 + \epsilon_n\right\}.\]
We slice $\mathcal{S}_{\bX}$ in the following sets $\mathcal{S}_{l,\bX} \triangleq \left \{\bTheta \in \mathcal{S}_{\bX}:  64^{l-1}\epsilon^0 \leq \left \Vert \bTheta - \widetilde\bTheta \right \Vert_{2,\bX}^2 \leq 64^l\epsilon^0\right \}$ and control the probability of the events 
$\mathcal{E}_{l,\bX} \triangleq \left\{ \exists \bTheta \in \mathcal{S}_{l,\bX}: \left \vert \Delta \cK^{\bTheta^*}_{\bX}(\bTheta, \widetilde\bTheta) - \Delta \cK^{\bA}_{\bX}(\bTheta, \widetilde \bTheta) \right \vert > \frac{64^{l}\epsilon^0 }{2\times 64^2\rho_n} + \epsilon_n \right\}.$
To do this, we control the probability of the events
$\mathcal{E}_{\bX}(T) = \left\{ \exists \bTheta \in \mathcal{S}_{\bX}(T): \left \vert \Delta \cK^{\bTheta^*}_{\bX}(\bTheta, \widetilde\bTheta) - \Delta \cK^{\bA}_{\bX}(\bTheta, \widetilde \bTheta) \right \vert > \frac{T}{2\times64^2\rho_n} +\epsilon_n \right\}$
where $\mathcal{S}_{\bX}(T) = \left \{\bTheta \in \mathcal{S}_{\bX}:  \left \Vert \bTheta - \widetilde\bTheta \right \Vert_{2,\bX}^2 \leq T \right \}.$
\begin{lem}\label{Lemma:ProbBlKLX}
 Let $Z_{T,\bX}=\underset{\bTheta\in \mathcal{S}_{\bX}(T)}{\sup} \left \vert \Delta \cK^{\bTheta^*}_{\bX}(\bTheta, \widetilde\bTheta) - \Delta \cK^{\bA}_{\bX}(\bTheta, \widetilde \bTheta) \right \vert$. There exists two absolute constants $C$, $C' >0$ such that
$\mathbb{P}^{\bX} \left( Z_{T,\bX} \geq \frac{T}{2\times64^2\rho_n} + C\frac{\rho_n^2}{\gamma_n^2} \left(n\log(k) + k^2 \right)\right) \leq 4\exp \left(-\frac{C'\gamma_n^2 T}{\rho_n^2}\right).$
 \end{lem}
 \begin{proof}
 To prove Lemma \ref{Lemma:ProbBlKLX}, we first show that $Z_{T,\bX}$ concentrates around its expectation and then bound this term. 
\begin{lem}\label{Lemma:concentrationKlX}
Let $Z_{T,\bX}$ be as in Lemma \ref{Lemma:ProbBlKLX}, then there exists two absolute constants $C$, $C'>0$ such that 
\begin{align*}
\mathbb{P}^{\bX} \left( \Big \vert Z_{T,\bX} - \mathbb{E}^{\bX} (Z_{T,\bX}) \Big \vert > \frac{C\rho_n}{\gamma_n^2} + \frac{T}{4 \times 64^2\rho_n}\right)\leq 4 \exp\left(-\frac{C'T\gamma_n^2}{\rho_n^2} \right).
\end{align*}
 \end{lem}
 \begin{lem}\label{Lemma:esperanceKLX}
Let $Z_{T,\bX}$ be as in Lemma \ref{Lemma:ProbBlKLX}, then there exists an absolute constant $C>0$ such that 
\begin{align*}
\mathbb{E}^{\bX} \left[Z_{T,\bX}\right] \leq \frac{T}{4 \times 64^2 \rho_n} + \frac{C\rho_n^2}{\gamma_n^2}\left(n\log(k) + k^2 \right).
\end{align*}
 \end{lem}
 Putting together Lemma \ref{Lemma:concentrationKlX} and Lemma \ref{Lemma:esperanceKLX}, we get that 
 $$\mathbb{P}^{\bX} \left( Z_{T,\bX} \geq \frac{T}{2\times64^2\rho_n} + C\frac{\rho_n^2}{\gamma_n^2} \left(n\log(k) +\frac{1}{\rho_n} +k^2 \right)\right) \leq 4\exp \left(-\frac{C'\gamma_n^2 T}{\rho_n^2}\right).$$
 If $n\rho_n \rightarrow \infty$, for $n$ large enough $n > \frac{1}{\rho_n}$. This yields the desired result.
 \end{proof}
 We combine Lemma \ref{Lemma:ProbBlKLX} and the arguments developed in Lemma \ref{EqNPiNX}, and note that $\mathbb{P}^{\bX}\left(\cE_{\bX} \right)$ does not depend on $\bX$ to conclude the proof of Lemma \ref{ControlKlX}.
 
 \subsubsection{Proof of Lemma \ref{Lemma:concentrationKlX}}
  \label{subsection:proofConcentration}
 
 In this Section, we prove the Lemma \ref{Lemma:concentrationKlX} that helps us bound $\left \vert Z_{T,\bX} - \mathbb{E}^{\bX} (Z_{T,\bX}) \right \vert$ with hight probability. To prove that $Z_{T,\bX}$ concentrates around its mean, we use the following version of Talagrand's Theorem for Lipschitz convex functions (for a proof, see Theorem 3.3 of \cite{chatterjee2015}).
\begin{thm}\label{Talagrand}
Suppose that $f: [-1, 1]^{N} \rightarrow \mathbb{R}$ is a convex Lipschitz function with Lipschitz constant $L$. Let $R_1,..., R_N$ be independent random variables taking value in $[-1,1]$. Let $Z:= f(R_1,...,R_N)$. Then for any $t \geq 0$,
\begin{align*}
\mathbb{P} \left( \left \vert Z - \mathbb{E} (Z) \right \vert > 16L + t \right)\leq 4 e^{\left(\frac{-t^2}{2L^2} \right)} .
\end{align*}
\end{thm} 
Recall that 
\begin{equation*}
\begin{split}
Z_{T, \bX} &=\underset{\bTheta\in \mathcal{S}_{\bX}(T)}{\sup} \left \vert \Delta \cK^{\bTheta^*}_{\bX}(\bTheta, \widetilde\bTheta) - \Delta \cK^{\bA}_{\bX}(\bTheta, \widetilde \bTheta) \right \vert\\
& = \underset{\bTheta\in \mathcal{S}_{\bX}(T)}{\sup}\left \vert \sum_{1\leq i<j \leq n}(\bX_{ij}\bA_{ij}-\bX_{ij}\bTheta_{ij}^*) \left( \log\left (\frac{\bTheta_{ij}}{\widetilde\bTheta_{ij}}\right ) - \log\left (\frac{1-\bTheta_{ij}}{1-\widetilde\bTheta_{ij}}\right )\right) \right \vert.
\end{split}
\end{equation*}
Note that $Z_{T,\bX} = f(\bA)$ where $f(\bR)$ is defined for $\bR \in [-1,1]^{\frac{(n-1)(n-2)}{2}}$ by 
\begin{equation*}
\begin{split}
f: \bR & \rightarrow \underset{\bTheta\in \mathcal{S}_{\bX}(T)}{\sup}\left \vert \sum_{1\leq i<j \leq n}\bX_{ij}(\bR_{ij}-\bTheta_{ij}^*) \left( \log\left (\frac{\bTheta_{ij}}{\widetilde\bTheta_{ij}}\right ) - \log\left (\frac{1-\bTheta_{ij}}{1-\widetilde\bTheta_{ij}}\right )\right) \right \vert.
\end{split}
\end{equation*}
 It is easy to see that $f$ is indeed convex. Our next step is to show that $f$ is Lipschitz. Let $ \bR, \bS \in [-1,1]^{\frac{(n-1)(n-2)}{2}}$,
\begin{eqnarray*}
\left \vert f(\bR) - f(\bS) \right \vert
&=& \Bigg \vert  \underset{\bTheta\in \mathcal{S}_{\bX}(T)}{\sup}  \left \vert \sum_{1\leq i<j \leq n}\bX_{ij}(\bR_{ij}-\bTheta_{ij}^*) \left( \log\left (\frac{\bTheta_{ij}}{\widetilde\bTheta_{ij}}\right ) - \log\left (\frac{1-\bTheta_{ij}}{1-\widetilde\bTheta_{ij}}\right )\right) \right \vert\\
&& - \underset{\bTheta\in \mathcal{S}_{\bX}(T)}{\sup} \left \vert \sum_{1\leq i<j \leq n}\bX_{ij}(\bS_{ij}-\bTheta_{ij}^*) \left( \log\left (\frac{\bTheta_{ij}}{\widetilde\bTheta_{ij}}\right ) - \log\left (\frac{1-\bTheta_{ij}}{1-\widetilde\bTheta_{ij}}\right )\right) \right \vert \Bigg\vert\\
&\leq& \underset{\bTheta\in \mathcal{S}_{\bX}(T)}{\sup}  \Bigg\vert  \left \vert \sum_{1\leq i<j \leq n}\bX_{ij}(\bR_{ij}-\bTheta_{ij}^*) \left( \log\left (\frac{\bTheta_{ij}}{\widetilde\bTheta_{ij}}\right ) - \log\left (\frac{1-\bTheta_{ij}}{1-\widetilde\bTheta_{ij}}\right )\right) \right \vert\\
&& - \left \vert \sum_{1\leq i<j \leq n}\bX_{ij}(\bS_{ij}-\bTheta_{ij}^*) \left( \log\left (\frac{\bTheta_{ij}}{\widetilde\bTheta_{ij}}\right ) - \log\left (\frac{1-\bTheta_{ij}}{1-\widetilde\bTheta_{ij}}\right )\right) \right \vert \Bigg\vert \\
& \leq & \underset{\bTheta\in \mathcal{S}_{\bX}(T)}{\sup} \left \vert \sum_{1\leq i<j \leq n}\bX_{ij}(\bR_{ij}-\bS_{ij}) \left( \log\left (\frac{\bTheta_{ij}}{\widetilde\bTheta_{ij}}\right ) - \log\left (\frac{1-\bTheta_{ij}}{1-\widetilde\bTheta_{ij}}\right )\right) \right \vert \\
& \leq & \underset{\bTheta\in \mathcal{S}_{\bX}(T)}{\sup}  \left \{ \sum_{1\leq i<j \leq n}\bX_{ij}\left \vert (\bR_{ij}-\bS_{ij}) \log\left (\frac{\bTheta_{ij}}{\widetilde\bTheta_{ij}}\right) \right \vert + \bX_{ij}\left \vert (\bR_{ij}-\bS_{ij}) \log\left (\frac{1-\bTheta_{ij}}{1-\widetilde\bTheta_{ij}}\right)\right \vert \right \} \\
& \leq & \underset{\bTheta\in \mathcal{S}_{\bX}(T)}{\sup}  \left \Vert \bR - \bS \right \Vert_{2} \left( \sum_{1\leq i<j \leq n}\bX_{ij}\left(\log\left (\frac{\bTheta_{ij}}{\widetilde\bTheta_{ij}}\right) \right)^2 + \bX_{ij}\left(\log\left (\frac{1-\bTheta_{ij}}{1-\widetilde\bTheta_{ij}}\right)\right)^2\right)^{\frac{1}{2}}
\end{eqnarray*}
where we have used that $\bX \in \{0,1\}^{n\times n}$. Thus $f$ is Lipschitz with Lipschitz constant $$ \underset{\bTheta\in \mathcal{S}_{\bX}(T)}{\sup} \left( \sum_{1\leq i<j \leq n}\bX_{ij}\left(\log\left (\frac{\bTheta_{ij}}{\widetilde\bTheta_{ij}}\right) \right)^2 + \bX_{ij}\left(\log\left (\frac{1-\bTheta_{ij}}{1-\widetilde\bTheta_{ij}}\right)\right)^2\right)^{\frac{1}{2}}.$$
As stated before, assuming that $\gamma_n \leq 1-\rho_n$, $x \rightarrow \log(x)$ and $x \rightarrow \log(1-x)$ are Lipschitz functions on $[\gamma_n, \rho_n]$ with Lipschitz constant $\gamma_n^{-1}$. Thus $f$ is Lipschitz with Lipschitz constant
$$ \underset{\bTheta\in \mathcal{S}_{\bX}(T)}{\sup} \left( \sum_{1\leq i<j \leq n}\bX_{ij}\left(\frac{\left \vert \bTheta_{ij} - \widetilde\bTheta_{ij} \right \vert }{\gamma_n}\right)^2 + \bX_{ij}\left(\frac{\left \vert \bTheta_{ij} - \widetilde\bTheta_{ij} \right \vert }{\gamma_n}\right)^2\right)^{\frac{1}{2}}.$$
This implies
\begin{eqnarray*}
\left \vert f(\bR) - f(\bS) \right \vert
& \leq &\left \Vert \bR - \bS \right \Vert_2 \underset{\bTheta\in \mathcal{S}_{\bX}(T)}{\sup}  \frac{\sqrt{2} \left \Vert \widetilde\bTheta - \bTheta \right \Vert_{2, \bX}}{\gamma_n} \leq  \left \Vert \bR - \bS \right \Vert_2 \frac{\sqrt{2T}}{\gamma_n}.
\end{eqnarray*}
We have shown that $f$ has a Lipschitz constant $L = \frac{\sqrt{2T}}{\gamma_n}$. Applying Theorem \ref{Talagrand} for $t = \frac{T}{8 \times 64^2 \rho_n}$, we get
\begin{align*}
\mathbb{P}^{\bX} \left( \Big \vert Z_{T,\bX} - \mathbb{E} (Z_{T, \bX}) \Big \vert > \frac{16\sqrt{2 T}}{\gamma_n} + \frac{T}{8 \times 64^2\rho_n} \right)\leq 4 \exp\left(\frac{-T\gamma_n^2}{8^2 \times 2 \times 64^4 \rho_n^2} \right).
\end{align*}
Using for $\beta >0$, $2\sqrt{ab} \leq \beta a^2 + b^2/\beta$  yields
\begin{align*}
\mathbb{P}^{\bX} \left( \Big\vert Z_{T, \bX} - \mathbb{E} (Z_{T, \bX}) \Big\vert > \frac{8^2 \times 64^2 \times 16 \rho_n}{\gamma_n^2} + \frac{T}{8 \times 64^2\rho_n}+ \frac{T}{8 \times 64^2\rho_n}\right)\leq 4 \exp\left(\frac{-T\gamma_n^2}{4 \times 4^2 \times 32^4 \rho_n^2} \right).
 \end{align*}
 
 This concludes the proof of Lemma \ref{Lemma:concentrationKlX}.

\subsection{Proof of Lemma \ref{Lemma:esperanceKLX}}
 \label{subsection:proofEsperance}
 Once we have shown that $Z_{T,\bX}$ concentrates around its mean, we bound $\mathbb{E} \left[ Z_{T, \bX} \right]$. To do so, we follow the steps of Lemma \ref{Lemma:esperanceKlPi}.
  Let ${\bepsilon}_{1\leq i<j \leq n}$ a Rademacher sequence. Applying Lemma \ref{Lemma:symmetrization}, we get
  \begin{equation*}
\begin{split}
\mathbb{E}^{\bX} \left[Z_{T,\bX}\right]& = \mathbb{E}^{\bX}\left[ \underset{\bTheta\in \mathcal{S}_{\bX}(T)}{\sup}\left \vert \sum_{1\leq i<j \leq n}\bX_{ij}(\bA_{ij}-\mathbb{E}[\bA_{ij}] ) \left( \log\left (\frac{\bTheta_{ij}}{\widetilde\bTheta_{ij}}\right ) - \log\left (\frac{1-\bTheta_{ij}}{1-\widetilde\bTheta_{ij}}\right )\right) \right \vert \right] \\
& \leq 2 \mathbb{E}^{\bX} \left[ \underset{\bTheta\in \mathcal{S}_{\bX}(T)}{\sup}\left \vert \sum_{1\leq i<j \leq n}\bX_{ij}\bepsilon_{ij}\bA_{ij} \left( \log\left (\frac{\bTheta_{ij}}{\widetilde\bTheta_{ij}}\right ) - \log\left (\frac{1-\bTheta_{ij}}{1-\widetilde\bTheta_{ij}}\right )\right) \right \vert \right] .
\end{split}
\end{equation*}
For any $1 \leq i < j \leq n$, let $\phi_{ij}: x \rightarrow \frac{1}{2}\gamma_n\bX_{ij}\bA_{ij}\left( \log\left(\frac{\widetilde\bTheta_{ij} - x}{\widetilde\bTheta_{ij}}\right) - \log\left(\frac{1 + x - \widetilde\bTheta_{ij}}{1 - \widetilde\bTheta_{ij}}\right) \right)$. Note that $\phi_{ij}$ is $1$-Lipschitz and vanishes at $0$ on the interval $[\widetilde\bTheta_{ij} - \rho_n, \widetilde\bTheta_{ij} - \gamma_n]$. Indeed,
\begin{eqnarray*}
\phi_i(x)' &=&  \frac{1}{2}\gamma_n\bX_{ij}\bA_{ij}\left( \frac{-1}{\widetilde\bTheta_{ij} - x} - \frac{1}{1 + x - \widetilde\bTheta_{ij}}\right)\\
&\leq& \bX_{ij}\bA_{ij}.
\end{eqnarray*}
 By definition of the functions $\phi_{ij}$, $$\mathbb{E}^{\bX}\left[Z_{T,\bX} \right] \leq \frac{4}{\gamma_n} \mathbb{E}^{\bX} \left[ \underset{\bTheta \in S_{\bX}(T)}{\sup}\left \vert \underset{i<j}{\sum}\bepsilon_{ij}\phi_{ij}(\bX_{ij}\bA_{ij}(\widetilde\bTheta_{ij}- \bTheta_{ij}))  \right \vert\right].$$ We apply Lemma \ref{Lemma:contraction} to get
 \begin{equation}\label{eq:EZX}
\mathbb{E}^{\bX} \left[ Z_{T,\bX}\right] \leq \frac{8}{\gamma_n} \mathbb{E}^{\bX} \left[ \underset{\bTheta \in S_{\bX}(T)}{\sup}\left \vert \underset{i<j}{\sum}\bX_{ij}\bepsilon_{ij}\bA_{ij}(\widetilde \bTheta_{ij}- \bTheta_{ij}))  \right \vert\right].
\end{equation}
Next, we apply Lemma \ref{BornePdtScalairePi} with $\bB = \bX$, $\bSigma_{ij} = \bX_{ij}\bA_{ij}\bepsilon_{ij}$ and $\alpha = \rho_n$. Note that $\left\Vert \bSigma \right\Vert_{\infty}\leq 1 $ and that for any matrix $\bTheta$, $\underset{i<j}{\sum}\mathbb{E}^{\bX} \left[\bX_{ij}^2\bepsilon_{ij}^2\bTheta_{ij}^2\right] \leq \rho_n\left \Vert \bTheta \right \Vert_{2,\bX}^2.$ Combining Lemma \ref{BornePdtScalairePi} and \eqref{eq:EZX} yields 
\begin{equation*}
\begin{split}
\mathbb{E}^{\bX} \left[ Z_{T,\bX}\right]& \leq \frac{8}{\gamma_n} \times \left(\frac{T\gamma_n}{32\times 64^2\rho_n} + C\frac{\rho_n^2}{\gamma_n}\left(n\log(k) + k^2\right)\right).
\end{split}
\end{equation*}
This concludes the proof of Lemma \ref{Lemma:esperanceKLX}.
 
\subsection{Proof of Theorem \ref{thm:AdaptEst}}

Our proof relies on two steps: first, we show that with high probability, $\widehat{d}$ is close to its expected value, which belongs to $[\gamma_n, \rho_n]$. More precisely, let $\underline{\gamma_n} = \frac{C_{inf}}{2}\rho_n\log(n)^{\frac{-1}{5}}$ and $\underline{\rho_n} = \left( 1 + \frac{C_{inf}}{2}\right)\rho_n\log(n)^{\frac{1}{5}}$. We prove that with high probability, $\underline{\gamma_n} \leq \widehat{\gamma_n} \leq \gamma_n$ and $\rho_n \leq \widehat{\rho_n}  \leq \underline{\rho_n}$. Then this implies that the oracle matrix $\widetilde{\bTheta}$ belongs to the set of definition of our estimator and its likelihood is greater than that of $\widehat{\bTheta}$. Then both $\widetilde{\bTheta}$ and $\widehat{\bTheta}$ belong to the set $\left[\underline{\gamma_n}, \underline{\rho_n}\right]^{n \times n}$ and we adapt the proof of Theorem \ref{thm_matrix_oracle} to get the desired result.
\begin{lem}
Let $\cE = \left \{\widehat{\gamma_n} \in \left[ \underline{\gamma_n}, \gamma_n\right], \widehat{\rho_n} \in \left[\rho_n, \underline{\rho_n}\right] \right \}$. There exists a positive constant $C$ and an integer $N$, both depending only on $C_{inf}$, such that $\forall n \geq N$, $\mathbb{P}(\cE) \geq 1 - \exp(-Cn\rho_n)$.
\end{lem}
\begin{proof}
Note that $\left \Vert \bA - \bTheta^*\right \Vert_{\infty} \leq 1$ almost surely, and that for any $1 \leq i < j \leq n$, $(\bA_{ij} - \bTheta^*_{ij})$ is centered and has a variance smaller than $\rho_n$. Applying Bernstein's inequality \ref{thm:Bernstein}  yields
\begin{eqnarray*}
\mathbb{P} \left( \left \vert \sum_{(i,j) \in \Omega} \left(\bA_{ij} - \bTheta^*_{ij} \right)\right \vert \geq \sqrt{2tn\rho_n} + \frac{3t}{2}\right) &\leq & 2e^{-t},\ \forall t>0.
\end{eqnarray*}
Choosing $t = \rho_nn C$ with $C>0$ sucht that $\sqrt{2C} + \frac{3C}{2} \leq \frac{C_{inf}}{2}$ yields 
\begin{eqnarray*}
\mathbb{P} \left( \left \vert \widehat{d} - \frac{\sum_{(i,j) \in \Omega}\bTheta^*_{ij}}{n} \right \vert \geq \frac{C_{inf}}{2}\rho_n\right) &\leq & 2e^{-Cn\rho_n}.
\end{eqnarray*}
Note that  in the sparse graphon model \eqref{sparse_graphon}, when $0 < C_{inf} \triangleq \underset{(x,y) \in [0,1]^2}{\inf}W^*(x,y)$, we see that $\gamma_n = C_{inf}\rho_n$ and $\frac{\sum_{(i,j) \in \Omega}\bTheta^*_{ij}}{n} \in \left[\gamma_n, \rho_n \right] = \left[C_{inf}\rho_n, \rho_n \right]$. So, with probability greater than $1 - 2e^{-Cn\rho_n}$, $\widehat{A} \in \left[\frac{C_{inf}\rho_n}{2}, (1 +\frac{C_{inf}}{2}) \rho_n\right] $. Let $N$ be such that $\log(N)^{-\frac{1}{5}} \leq \frac{C_{inf}}{1 + \frac{C_{inf}}{2}}$ and $ \log(N)^{\frac{1}{5}} \geq 2C_{inf}^{-1}$. For any $n \geq N$, with probability greater than $1 - 2e^{-Cn\rho_n}$, $\widehat{\gamma_n} \in \left[ \underline{\gamma_n}, \gamma_n\right]$ and $\widehat{\rho_n} \in \left[\rho_n, \underline{\rho_n}\right]$.
\end{proof}

To prove Theorem \ref{thm:AdaptEst}, we work conditionally on the event $\cE$. Note that in the model \eqref{graphSeq}, the law of the remaining entries $\left(A_{i,j} \right)_{(i,j) \not \in \Omega}$ is independent of $\cE$. Since on $\cE$ both $\widehat{\bTheta}$ and $\bTheta^*$ belong to the set $\left[ \underline{\gamma_n}, \underline{\rho_n}\right]$, we have

\begin{eqnarray}\label{omega}
\left \Vert \widehat{\bTheta} - \bTheta^* \right \Vert_2^2 &=&  \sum_{(i,j) \in \Omega} \left(\widehat{\bTheta}_{ij} - \bTheta^*_{ij} \right)^2 + \sum_{(i,j) \not \in \Omega} \left(\widehat{\bTheta}_{ij} - \bTheta^*_{ij} \right)^2 \nonumber \\
&\leq&  n\underline{\rho_n^2}+ \sum_{(i,j) \not \in \Omega} \left(\widehat{\bTheta}_{ij} - \bTheta^*_{ij} \right)^2.
\end{eqnarray}

  We adapt the proof of Theorem \ref{thm_matrix_oracle} to bound the second term. Let $\underline{\epsilon_n} \triangleq C \left( \underline{\rho_n}\Big/\underline{\gamma_n}\right)^2 \left(n\log(k) + k^2 \right)$ where $C$ is the same absolute constant as in Theorem \ref{thm_matrix_oracle}, and let $\underline{\epsilon}^0 \triangleq \underline{\rho_n \epsilon_n}$. We start by considering the following two cases:

\textbf{Case 1}: $\sum_{(i,j) \not \in \Omega} \left(\widehat{\bTheta}_{ij} - \widetilde{\bTheta}_{ij} \right)^2 \leq \underline{\epsilon}^0$. Then, the statement of Theorem \ref{thm:AdaptEst} follows from \eqref{omega} and Lemma \ref{kl_frobenius}:
\begin{eqnarray*}
 \sum_{(i,j) \not \in \Omega} \left(\widehat{\bTheta}_{ij} - \bTheta^*_{ij} \right)^2 &\leq& 2\sum_{(i,j) \not \in \Omega} \left(\widehat{\bTheta}_{ij} - \widetilde{\bTheta}_{ij} \right)^2  + 2\sum_{(i,j) \not \in \Omega} \left(\widetilde{\bTheta}_{ij} - \bTheta^*_{ij} \right)^2\\
&\leq& 2\underline{\rho_n\epsilon_n} + 16 \underline{\rho_n} \mathcal{K}(\bTheta^{*}, \widetilde\bTheta)\\
&\leq& C\log(n)\rho_n\left( \mathcal{K}(\bTheta^{*}, \widetilde\bTheta)+n\log(k) + k^2\right).
\end{eqnarray*}

\textbf{Case 2}: $\sum_{(i,j) \not \in \Omega} \left(\widehat{\bTheta}_{ij} - \widetilde{\bTheta}_{ij} \right)^2 > \underline{\epsilon}^0$. Then $\widehat{\bTheta}$ belongs to the set \[\underline{\mathcal{S}} = \left \{\bTheta \in \underset{z \in \mathcal{Z}_{n,k}}{\cup}\mathcal{T}_z: \sum_{(i,j) \not \in \Omega} \left(\widehat{\bTheta}_{ij} - \widetilde{\bTheta}_{ij} \right)^2 > \underline{\epsilon}^0,\ \Vert \bTheta\Vert_{\infty}\leq \underline{\rho_n},\ \underset{i<j}{\min} \{\bTheta_{ij}\} \geq \underline{\gamma_n} \right \}.\]

As before, Lemma \ref{kl_frobenius} implies
\begin{eqnarray*}
\sum_{(i,j) \not \in \Omega} \left(\widehat{\bTheta}_{ij} - \bTheta^*_{ij} \right)^2 &\leq& 8\underline{\rho_n}\sum_{(i,j) \not \in \Omega} \mathcal{K}(\bTheta^{*}_{ij}, \widehat \bTheta_{ij})\\
&\leq& 8\underline{\rho_n}\mathcal{K}(\bTheta^{*}, \widetilde\bTheta)+8\underline{\rho_n}\sum_{(i,j) \not \in \Omega}\left (\mathcal{K}(\bTheta^{*}_{ij}, \widehat \bTheta_{ij})-\mathcal{K}(\bTheta^{*}_{ij}, \widetilde\bTheta_{ij})\right ).
\end{eqnarray*}
On the event $\cE$, $\widetilde{\bTheta}$ belongs to the set of matrices on which the maximum likelihood estimator is defined, thus the definition of $\widehat\bTheta$ implies $\sum_{(i,j) \not \in \Omega}\left (\mathcal{K}(\bA_{ij}, \widehat \bTheta_{ij})-\mathcal{K}(\bA_{ij}, \widetilde\bTheta_{ij})\right) \leq 0$ and
\begin{equation*}
\sum_{(i,j) \not \in \Omega} \left(\widehat{\bTheta}_{ij} - \bTheta^*_{ij} \right)^2 \leq 8\underline{\rho_n}\mathcal{K}(\bTheta^{*}, \widetilde\bTheta)+8\underline{\rho_n}\sum_{(i,j) \not \in \Omega}\left (\mathcal{K}(\bTheta^{*}_{ij}, \widehat \bTheta_{ij})- \mathcal{K}(\bA_{ij}, \widehat\bTheta_{ij}) - \left( \mathcal{K}(\bTheta^{*}_{ij}, \widetilde\bTheta_{ij}) - \mathcal{K}(\bA_{ij}, \widetilde\bTheta_{ij})\right) \right ).
\end{equation*}

The proof of the following lemma follows the lines of the proof of Lemma \ref{ControlKlX}, and we do not present it.

\begin{lem}\label{Lemma:bad_event_adapt}
 There exists a constant $C>0$ depending only on $C_{inf}$ such that for any $\bTheta\in \underline{\mathcal{S}}$ simultaneously we have 
$$ \left \vert \sum_{(i,j) \not \in \Omega}\left (\mathcal{K}(\bTheta^{*}_{ij}, \bTheta_{ij})- \mathcal{K}(\bA_{ij}, \bTheta_{ij}) - \left( \mathcal{K}(\bTheta^{*}_{ij}, \widetilde\bTheta_{ij}) - \mathcal{K}(\bA_{ij}, \widetilde\bTheta_{ij})\right) \right )\right \vert \leq \frac{1}{32\underline{\rho_n}}\sum_{(i,j) \not \in \Omega} \left(\bTheta_{ij} - \widetilde{\bTheta}_{ij} \right)^2+\underline{\epsilon_n} $$
 with probability at least $1-5\exp\left(-C\underline{\rho_n} \left(n\log(k) + k^2 \right)\right)$.
 \end{lem}
 
 This implies that on the event $\cE$, with large probability, 
\begin{eqnarray}\label{omegabar}
\sum_{(i,j) \not \in \Omega} \left(\widehat{\bTheta}_{ij} - \bTheta^*_{ij} \right)^2 & \leq & 8\underline{\rho_n}\mathcal{K}(\bTheta^{*}, \widetilde\bTheta)+ \frac{1}{4} \sum_{(i,j) \not \in \Omega} \left(\widehat \bTheta_{ij} - \widetilde{\bTheta}_{ij} \right)^2+ 8\underline{\epsilon_n\rho_n}\nonumber\\
&\leq&  8\underline{\rho_n}\mathcal{K}(\bTheta^{*}, \widetilde\bTheta)+ \frac{1}{2} \sum_{(i,j) \not \in \Omega} \left(\widehat \bTheta_{ij} - \bTheta^*_{ij} \right)^2 + \frac{1}{2} \sum_{(i,j) \not \in \Omega} \left(\bTheta^*_{ij} - \widetilde{\bTheta}_{ij} \right)^2 +8\underline{\epsilon_n\rho_n} \nonumber\\
\frac{1}{2}\sum_{(i,j) \not \in \Omega} \left(\widehat{\bTheta}_{ij} - \bTheta^*_{ij} \right)^2 &\leq & \left( 8 \underline{\rho_n} + 4 \underline{\rho_n} \right)\mathcal{K}(\bTheta^{*}, \widetilde\bTheta) + 8 \underline{\epsilon_n\rho_n}.
\end{eqnarray}
Using \eqref{omega} and \eqref{omegabar}, we have shown that for $n\geq N$ and some constants $C, C'>0$ depending only on $C_{inf}$, with probability at least $1-5\exp\left(-C\underline{\rho_n} \left(n\log(k) + k^2 \right)\right) - 2\exp\left(-Cn\rho_n\right)$,

\[\left \Vert \widehat{\bTheta} - \bTheta^* \right \Vert_2^2 \leq  C\left(\underline{\rho_n^2}n + \underline{\rho_n}\mathcal{K}(\bTheta^{*}, \widetilde\bTheta)+ \underline{\rho_n\epsilon_n}\right).\]

We conclude the proof of Theorem \ref{thm:AdaptEst} by noticing that $ n\underline{\rho_n}^2 \leq \underline{\rho_n\epsilon_n}$ and using that $\underline{\rho_n} = C\log(n)^{\frac{1}{5}}\rho_n$.

\subsection{Proof of Proposition \ref{prop:MLGraphon}}

By definition, 
\begin{eqnarray*}
\cK\left(\bTheta^*, \bTheta^{bc} \right) &=& \sum_{i<j} \left(\bTheta^*_{ij}\log \left(\frac{\bTheta^*_{ij}}{\bTheta^{bc}_{ij}} \right) + \left(1 - \bTheta^*_{ij} \right)\log \left(\frac{1 - \bTheta^*_{ij}}{1 - \bTheta^{bc}_{ij}} \right)\right)\\
&\leq & \sum_{i<j}\left(\bTheta^*_{ij}\frac{\bTheta^*_{ij} - \bTheta^{bc}_{ij}}{\bTheta^{bc}_{ij}} + \left(1 - \bTheta^*_{ij} \right)\frac{\bTheta^{bc}_{ij} - \bTheta^{*}_{ij}}{1 - \bTheta^{bc}_{ij}}\right)\\
&=& \sum_{i<j}\frac{\left(\bTheta^{bc}_{ij} - \bTheta^{*}_{ij}\right)^2}{\left(1 - \bTheta^{bc}_{ij}\right)\bTheta^{bc}_{ij}}
\end{eqnarray*}
where the second line follows from the fact that for any $x > 0$, $\log(x) \leq x-1$. Since for any $1 \leq i < j \leq n$,  $\bTheta^{bc}_{ij}$ and $\bTheta^*_{ij}$ belong to $[C_{inf} \rho_n, \rho_n]$, this yields
\begin{eqnarray*}
\cK\left(\bTheta^*, \bTheta^{bc} \right) &\leq & \sum_{i<j}\frac{\left(\bTheta^{bc}_{ij} - \bTheta^{*}_{ij}\right)^2}{\left(1 - \rho_n\right)C_{inf}\rho_n}.
\end{eqnarray*}
Now, recall that $\bTheta^{*}_{ij} = \rho_n W\left (\zeta_i, \zeta_j\right)$ and define $z^*: [n] \rightarrow [k]$ by $z^*(i) = \underset{1\leq a \leq k}{\sum} a \mathds{1} \left \{\zeta_i \in \left[\frac{a-1}{k}, \frac{a}{k}\right) \right\}$ for any $i \in [n]$. Moreover, define $\bTheta^{bc}_{ij} = \rho_n W\left (\frac{z^*(i)}{k}, \frac{z^*(j)}{k}\right)$. Note that by definition of $z^*$, for any $i$, $\left \vert\zeta_i - \frac{z^*(i)}{k}\right \vert \leq \frac{1}{k}$. Thus
\begin{eqnarray*}
\cK\left(\bTheta^*, \bTheta^{bc} \right) &\leq &  \frac{\rho_n}{C_{inf}(1-\rho_n)}\sum_{i<j} \left( W\left(\zeta_i, \zeta_j \right) - W\left (\frac{z^*(i)}{k}, \frac{z^*(j)}{k}\right)\right )^2\\
&\leq &   \frac{4\rho_nM^2}{C_{inf}(1-\rho_n)}\sum_{i<j} \left(\frac{1}{k}\right)^{2(\alpha \land 1)}
\end{eqnarray*}
where the last equation follows from \eqref{Hoelder}.

\subsection{Technical lemmas}
\label{subsec:technicalLemmas}

 \begin{lem}\label{kl_frobenius}
 For any $\bTheta$, $\bTheta' \in \mathbb{R}^{n\times n}$ and $\bPi \in [0,1]^{n\times n}_{sym}$, 
 \[\Vert \bTheta-\bTheta'\Vert^{2}_{2,\bPi}\leq 8\left (\Vert \bTheta\Vert_{\infty}\vee \Vert  \bTheta'\Vert_{\infty}\right ) \mathcal{K}_{\bPi}(\bTheta, \bTheta').\]
 \end{lem}
 \begin{proof}
 By definition of Bernoulli Kullback-Leibler divergence for any $0<q,q'<1$ we have that
 \begin{equation*}
 \begin{split}
 \cK(q,q')&=q\log\left (\frac{q}{q'}\right )+(1-q)\log\left (\frac{1-q}{1-q'}\right )\geq \left (\sqrt{q}-\sqrt{q'}\right )^{2}+\left (\sqrt{1-q}-\sqrt{1-q'}\right )^{2}\\
 &\geq \frac{1}{2}\left [\left (\sqrt{q}-\sqrt{ q'}\right )-\left (\sqrt{1-q}-\sqrt{1-q'}\right )\right ]^{2}.
 \end{split}
 \end{equation*}
 Using Taylor's Theorem for some $\eta$ between $q$ and $q'$ we get
  \begin{equation}\label{one-dim_kl}
   \begin{split}
   \cK(q,q')&\geq \frac{1}{2}\left [\frac{1}{2\sqrt{\eta}}\left (q-q'\right )+\frac{1}{2\sqrt{1-\eta}}\left (q-q'\right )\right ]^{2}= \frac{\left (q-q'\right )^{2}}{8}\left [\frac{1}{\sqrt{\eta}}+\frac{1}{\sqrt{1-\eta}}\right ]^{2}\\&= \frac{\left (q-q'\right )^{2}}{8}\left [\frac{1}{\eta}+\frac{1}{1-\eta}\right ]= \frac{\left (q-q'\right )^{2}}{8} \frac{1}{\eta(1-\eta)}\geq \frac{(q-q')^{2}}{8(q\vee  q')}.
   \end{split}
   \end{equation}
   Now Lemma \ref{kl_frobenius} follows from  \eqref{one-dim_kl} and 
   \begin{align*}
   \mathcal{K}_{\bPi}(\bTheta, \bTheta')=\sum_{i<j}\bPi_{ij}\cK (\bTheta_{ij}, \bTheta_{ij}').
     \end{align*}
 \end{proof}
 
 \begin{lem}\label{KLSeuil}
 Let $\widetilde{\bTheta}^s$ and $n_s$ be defined as in \eqref{eq:barThetaS}, and assume that $\gamma_n \leq \frac{1}{2}$, then
 $$\mathcal{K}_{\bPi}(\bTheta^{*}, \widetilde \bTheta^s) - \mathcal{K}_{\bPi}(\bTheta^{*}, \widetilde \bTheta) \leq 2\gamma_n n_s.$$
 \end{lem}
\begin{proof}
\begin{eqnarray*}
\mathcal{K}_{\bPi}(\bTheta^{*}, \widetilde \bTheta^s) - \mathcal{K}_{\bPi}(\bTheta^{*}_{ij}, \widetilde \bTheta_{ij}) &=& \underset{i<j}{\sum}\bPi_{ij} \left(  \bTheta^*_{ij} \log \left(\frac{\widetilde \bTheta_{ij}}{\widetilde \bTheta^s_{ij}} \right) + (1-\bTheta^*_{ij}) \log \left(\frac{1-\widetilde \bTheta_{ij}}{1-\widetilde \bTheta^s_{ij}} \right)\right) \\
&=& \underset{i<j}{\sum} \bPi_{ij}\mathds{1}\left\{\widetilde \bTheta_{ij} < \gamma_n \right\} \left( \bTheta^*_{ij} \log \left(\frac{\widetilde \bTheta_{ij}}{\gamma_n} \right) + (1-\bTheta^*_{ij}) \log \left(\frac{1-\widetilde \bTheta_{ij}}{1-\gamma_n} \right)\right)\\
&\leq &\underset{i<j}{\sum} \bPi_{ij}\mathds{1}\left\{\widetilde \bTheta_{ij} < \gamma_n \right\} (1-\bTheta^*_{ij}) \log \left(1 + \frac{\gamma_n-\widetilde \bTheta_{ij}}{1-\gamma_n} \right)\\
&\leq& \underset{i<j}{\sum} \bPi_{ij}\mathds{1}\left\{\widetilde\bTheta_{ij} < \gamma_n \right\}  \frac{\gamma_n-\widetilde\bTheta_{ij}}{1-\gamma_n} \leq \underset{i<j}{\sum} \bPi_{ij}\mathds{1}\left\{\widetilde \bTheta_{ij} < \gamma_n \right\} 2\gamma_n \leq 2n_s\gamma_n.
\end{eqnarray*}
\end{proof}


\subsection{Proof of Proposition \ref{prp:variational}} \label{proof:variational}

In the case of fully observed network, we alleviate notations and write 
\begin{eqnarray*}
\cL(\bA;z,\bQ) &=& \underset{i \neq j}{\sum}\bA_{ij}\log\left(\bQ_{z(i), z(j)}\right) + \left(1-\bA_{ij}\right)\log\left(1 -\bQ_{z(i), z(j)}\right) ,\\
\mathfrak{l}\left(\bA; \alpha, \bQ\right) &=&  \underset{z \in \mathcal{Z}_{n,k}}{\sum}\left(\underset{i}{\prod} \alpha_{z(i)}\right)\exp\left(\cL(\bA ;z,\bQ)\right),\\
\text{and } \mathcal{J}\left(\bA; \tau, \alpha, \bQ\right) &=& \log\left(\mathfrak{l}\left(\bA; \alpha, \bQ\right)\right) - KL \left(\mathbb{P}_{\tau} \left(\cdot \right) \vert \vert \mathbb{P} \left(\cdot \vert \bA, \alpha, \bQ\right) \right).
\end{eqnarray*}

For any $z \in \mathcal{Z}_{n,k}$ and $(\alpha, \bQ) \in \mathcal{Q}$, we denote
$$\mathfrak{l}'\left(\bA, z ; \alpha, \bQ\right) = \left(\underset{i \leq n}{\prod}\alpha_{z(i)} \right)\exp \left(\cL(\bA;z,\bQ) \right)$$ the likelihood of the parameters $(\alpha, \bQ)$ and the label function $z$. Then, the likelihood of the stochastic block model with parameters $(\alpha, \bQ)$ is given by $\mathfrak{l}\left(\bA; \alpha, \bQ\right) = \underset{z \in \mathcal{Z}_{n,k}}{\sum} \mathfrak{l}'\left(\bA, z ; \alpha, \bQ\right)$.  Note that the likelihood functions $\mathfrak{l}\left(\bA; \alpha, \bQ\right)$ and $\mathfrak{l}'\left(\bA, z ; \alpha, \bQ\right)$ provide lower and upper bounds on the variational objective function $\mathcal{J}\left(\bA; \tau, \alpha, \bQ\right)$ : for any parameter $(\alpha, \bQ)$ and any label function $z \in \cZ_{n,k}$,
\begin{equation}\label{eq:orderLikelihood}
\mathfrak{l}'\left(\bA, z ; \alpha, \bQ\right) \leq \sup_{\tau \in \cT}\exp\left(\mathcal{J}\left(\bA; \tau, \alpha, \bQ\right)\right) \leq \mathfrak{l}\left(\bA; \alpha, \bQ\right).
\end{equation}

To prove Proposition \ref{prp:variational}, we first show that $\mathbb{P}\left(\cdot \vert \bA,  \widehat{\alpha}^{Var}, \widehat{\bQ}^{Var}\right)$, i.e. the posterior distribution of $z$ at the variational estimator $(\widehat{\alpha}^{Var}, \widehat{\bQ}^{Var})$, concentrates around $\delta_{z'}$, the dirac distribution at the label function 
$z' = \argmax_{z : z\sim z^*} \mathfrak{l}'\left(\bA, z; \widehat{\alpha}^{VAR},  \widehat{\bQ}^{VAR}\right)$ :

\begin{equation}\label{eq:concentrationThetaVar}
    \mathbb{P}\left(z' \vert \bA, \widehat{\alpha}^{Var}, \widehat{\bQ}^{Var}\right) = 1 - o_p(1).
\end{equation}
Then, we show that it implies the concentration of the estimator $\widehat{z}^{Var}$ :
\begin{equation}\label{eq:concentrationzVar}
    \mathbb{P}\left(\widehat{z}^{Var} = z' \vert \bA\right) = 1 - o_p(1).
\end{equation}
Together \eqref{eq:concentrationThetaVar} and \eqref{eq:concentrationzVar} imply $\mathbb{P}\left(\widehat{z}^{Var} \sim z^* \vert \bA\right) = 1 - o_p(1)$. Since the random variable $\mathbb{P}\left(\widehat{z}^{Var} \sim z^* \vert \bA\right)$ is bounded, Equation \eqref{eq:concentrationzVar} also implies that it converges to $1$ in expectation. Finally, we show that with probability going to one, the maximum likelihood estimator of the label function is equal to the true label function (up to permutation):
\begin{equation}\label{eq:concentrationzML}
    \mathbb{P}\left(\widehat{z} \sim z^*\right) = 1 - o_p(1)
\end{equation}
which concludes the proof of Proposition  \ref{prp:variational}.

\bigskip

\noindent \textbf{\underline{Proof of Equation \eqref{eq:concentrationThetaVar}}}

\noindent The proof of Equation \eqref{eq:concentrationThetaVar} relies on results proven in \cite{LikelihoodBickel}, which we recall for the sake of completeness. For any two parameters $(\alpha, \bQ)$ and  $(\alpha', \bQ')$ in $\mathcal{Q}$, we say that $(\alpha', \bQ') \in \mathcal{S}_{\alpha, \bQ}$ if there exists a permutation $\sigma$ of $\{1, ..., k\}$ such that for any $(a,b) \in \{1, ..., k\}^2$, $\bQ'_{\sigma(a), \sigma(b)} = \bQ_{a,b}$ and $\alpha'_{\sigma(a)} = \alpha_{a}$.

\begin{thm}[Theorem 1 in \cite{LikelihoodBickel}]\label{thm:1Bickel}
Let $(z^*, A)$ be generated from a stochastic block model with parameters $(\alpha^*,\bQ^*) \in \mathcal{Q}$ such that $\bQ^0$ has no identical columns and $\rho_n \gg \log(n)/n$. Then, for any $(\alpha, \bQ) \in \mathcal{Q}$,
$$\frac{\mathfrak{l}\left(\bA; \alpha, \bQ\right)}{\mathfrak{l}\left(\bA; \alpha^*, \bQ^*\right)} = \underset{(\alpha', \bQ') \in \mathcal{S}_{\alpha, \bQ}}{\max} \frac{\mathfrak{l}'\left(\bA, z^* ; \alpha', \bQ'\right)}{\mathfrak{l}'\left(\bA, z^* ; \alpha^*, \bQ^*\right)} \left(1 + \epsilon_n\left(\left(\alpha', \bQ'\right), k\right)\right) +\epsilon_n\left(\left(\alpha', \bQ'\right), k\right)$$
where $\sup_{(\alpha, \bQ) \in \cQ} \epsilon_n\left(\left(\alpha, \bQ\right), k\right) = o_p(1)$.
\end{thm}
\begin{lem}[Lemma 3 in \cite{LikelihoodBickel}]\label{Lem:3Bickel}
Let $(z^*, A)$ be generated from a stochastic block model with parameters $(\alpha^*,\bQ^*) \in \mathcal{Q}$ such that $\bQ^0$ has no identical columns and $\rho_n \gg \log(n)/n$. Then,
$$\frac{\mathfrak{l}'\left(\bA, z^*; \alpha^*, \bQ^*\right)}{\mathfrak{l}\left(\bA; \alpha^*, \bQ^*\right)} = 1 + o_p(1).$$ 
\end{lem}

\noindent Recall that $z' = \argmax_{z : z\sim z^*} \mathfrak{l}'\left(\bA, z^*; \widehat{\alpha}^{VAR},  \widehat{\bQ}^{VAR}\right)$. By definition of $\mathfrak{l}$ and $\mathfrak{l}'$, 
\begin{eqnarray*}
\underset{z \neq z'}{\sum}\mathfrak{l}'\left(\bA, z; \widehat{\alpha}^{VAR}, \widehat{\bQ}^{VAR}\right)&=& \mathfrak{l}\left(\bA; \widehat{\alpha}^{VAR}, \widehat{\bQ}^{VAR}\right) -  \mathfrak{l}'\left(\bA, z'; \widehat{\alpha}^{VAR}, \widehat{\bQ}^{VAR}\right).
\end{eqnarray*}
Thus
\begin{eqnarray}\label{eq:decompo}
 \frac{\underset{z \neq z'}{\sum}\mathfrak{l}'\left(\bA, z; \widehat{\alpha}^{VAR}, \widehat{\bQ}^{VAR}\right)}{\mathfrak{l}'\left(\bA, z^*; \alpha^*, \bQ^*\right)}&=& \frac{\mathfrak{l}\left(\bA; \alpha^*, \bQ^*\right)}{\mathfrak{l}'\left(\bA, z^*; \alpha^*, \bQ^*\right)}\times \frac{\mathfrak{l}\left(\bA; \widehat{\alpha}^{VAR}, \widehat{\bQ}^{VAR}\right)}{\mathfrak{l}\left(\bA; \alpha^*, \bQ^*\right)} -  \frac{\mathfrak{l}'\left(\bA, z'; \widehat{\alpha}^{VAR}, \widehat{\bQ}^{VAR}\right)}{\mathfrak{l}'\left(\bA, z^*; \alpha^*, \bQ^*\right)}.
\end{eqnarray}
Using Lemma \ref{Lem:3Bickel}, we have that 
\begin{equation}\label{eq:Lem3}
\frac{\mathfrak{l}\left(\bA; \alpha^*, \bQ^*\right)}{\mathfrak{l}'\left(\bA, z^*; \alpha^*, \bQ^*\right)} = 1 + o_p(1).
\end{equation}
Moreover, we note that 
\begin{eqnarray*}
\max_{(\alpha', \bQ') \in \mathcal{S}_{\widehat{\alpha}^{VAR},\widehat{\bQ}^{VAR}}} \mathfrak{l}'\left(\bA, z^* ; \alpha', \bQ'\right) & = & \max_{z \sim z^*} \mathfrak{l}'\left(\bA, z ; \widehat{\alpha}^{VAR},\widehat{\bQ}^{VAR}\right) \\
& = & \mathfrak{l}'\left(\bA, z' ; \widehat{\alpha}^{VAR},\widehat{\bQ}^{VAR}\right)
\end{eqnarray*}
by the definition of $z'$. Then, applying Theorem \ref{thm:1Bickel}, we get that 
\begin{equation}\label{eq:Th1}
\frac{\mathfrak{l}\left(\bA; \widehat{\alpha}^{VAR}, \widehat{\bQ}^{VAR}\right)}{\mathfrak{l}\left(\bA; \alpha^*, \bQ^*\right)} = \frac{\mathfrak{l}'\left(\bA, z' ; \widehat{\alpha}^{VAR}, \widehat{\bQ}^{VAR}\right)}{\mathfrak{l}'\left(\bA, z^* ; \alpha^*, \bQ^*\right)} \left(1 + o_p(1)\right) + o_p(1).
\end{equation}
Combining Equations \eqref{eq:decompo}, \eqref{eq:Lem3} and \eqref{eq:Th1}, we obtain that 
\begin{eqnarray*}
 \frac{\underset{z \neq z'}{\sum}\mathfrak{l}'\left(\bA, z; \widehat{\alpha}^{VAR}, \widehat{\bQ}^{VAR}\right)}{\mathfrak{l}'\left(\bA, z^*; \alpha^*, \bQ^*\right)}&=& \frac{\mathfrak{l}'\left(\bA, z' ; \widehat{\alpha}^{VAR}, \widehat{\bQ}^{VAR}\right)}{\mathfrak{l}'\left(\bA, z^* ; \alpha^*, \bQ^*\right)} o_p(1) + o_p(1).
\end{eqnarray*}
Thus, 
\begin{eqnarray}\label{eq:petit_o}
\underset{z \neq z'}{\sum}\mathfrak{l}'\left(\bA, z; \widehat{\alpha}^{VAR}, \widehat{\bQ}^{VAR}\right) = \max \left\{\mathfrak{l}'\left(\bA, z^*; \alpha^*, \bQ^*\right), \mathfrak{l}'\left(\bA, z' ; \widehat{\alpha}^{VAR}, \widehat{\bQ}^{VAR}\right)\right\} o_p(1).
\end{eqnarray}
On the one hand, using Equation \eqref{eq:orderLikelihood} and the definition of $(\widehat{\tau}^{VAR}, \widehat{\alpha}^{VAR}, \widehat{\bQ}^{VAR})$, we find that 
\begin{eqnarray*}
\mathfrak{l}'\left(\bA, z^* ; \alpha^*, \bQ^*\right)& \leq& \sup_{\tau \in \cT} \exp\left(\mathcal{J}\left(\bA;\tau,  \alpha^*, \bQ^*\right)\right)\\
&\leq &\exp\left(\mathcal{J}\left(\bA; \widehat{\tau}^{VAR}, \widehat{\alpha}^{VAR}, \widehat{\bQ}^{VAR}\right)\right) \\
&\leq &\mathfrak{l}\left(\bA;  \widehat{\alpha}^{VAR}, \widehat{\bQ}^{VAR}\right).
\end{eqnarray*}
Also, by the definition of $\mathfrak{l}$ and $\mathfrak{l}'$, we have that $\mathfrak{l}'\left(\bA, z' ; \widehat{\alpha}^{VAR}, \widehat{\bQ}^{VAR}\right)\leq\mathfrak{l}\left(\bA;  \widehat{\alpha}^{VAR}, \widehat{\bQ}^{VAR}\right)$. Thus, Equation \eqref{eq:petit_o} implies
\begin{eqnarray}\label{eq:FinalPosterior}
\underset{z \neq z'}{\sum}\mathfrak{l}'\left(\bA, z; \widehat{\alpha}^{VAR}, \widehat{\bQ}^{VAR}\right) =  \mathfrak{l}\left(\bA;  \widehat{\alpha}^{VAR}, \widehat{\bQ}^{VAR}\right)o_p(1).
\end{eqnarray}
Now, we can conclude the proof of Equation \eqref{eq:concentrationThetaVar} by noticing that
\begin{eqnarray*}
\mathbb{P}\left(z' \vert \bA, \widehat{\alpha}^{Var}, \widehat{\bQ}^{Var}\right) &=& \frac{\mathfrak{l}'\left(\bA, z'; \widehat{\alpha}^{VAR}, \widehat{\bQ}^{VAR}\right)}{\mathfrak{l}\left(\bA; \widehat{\alpha}^{VAR}, \widehat{\bQ}^{VAR}\right)}\\
&=&  1 - \frac{\underset{z \neq z'}{\sum}\mathfrak{l}'\left(\bA, z; \widehat{\alpha}^{VAR}, \widehat{\bQ}^{VAR}\right)}{\mathfrak{l}\left(\bA; \widehat{\alpha}^{VAR}, \widehat{\bQ}^{VAR}\right)}
\end{eqnarray*}
and using Equation \eqref{eq:FinalPosterior}.

\bigskip

\noindent \textbf{\underline{Proof of Equation \eqref{eq:concentrationzVar}}}
By the definition of $\mathcal{J}\left(\bA; \tau, \alpha, \bQ\right)$, we have that
\begin{equation*}
KL\left(\mathbb{P}_{\widehat{\tau}^{VAR}}(\cdot)\vert\vert\mathbb{P}\left(\cdot \vert \bA, \widehat{\alpha}^{VAR}, \widehat{\bQ}^{VAR}\right) \right) = \log\left(\mathfrak{l} \left(\bA; \widehat{\alpha}^{VAR}, \widehat{\bQ}^{VAR}\right) \right)- \mathcal{J}\left(\bA; \widehat{\tau}^{VAR}, \widehat{\alpha}^{VAR}, \widehat{\bQ}^{VAR}\right).
\end{equation*}
Equation \eqref{eq:orderLikelihood} implies that $\mathcal{J}\left(\bA; \widehat{\tau}^{VAR}, \widehat{\alpha}^{VAR}, \widehat{\bQ}^{VAR}\right) \geq \log\left(\mathfrak{l}\left(\bA, z' ; \widehat{\alpha}^{VAR}, \widehat{\bQ}^{VAR}\right)\right)$, so
\begin{equation*}
KL\left(\mathbb{P}_{\widehat{\tau}^{VAR}}(\cdot)\vert\vert\mathbb{P}\left(\cdot \vert \bA, \widehat{\alpha}^{VAR}, \widehat{\bQ}^{VAR}\right) \right) \leq \log\left(\mathfrak{l} \left(\bA; \widehat{\alpha}^{VAR}, \widehat{\bQ}^{VAR}\right) \right) - \log\left(\mathfrak{l}\left(\bA, z' ; \widehat{\alpha}^{VAR}, \widehat{\bQ}^{VAR}\right)\right).
\end{equation*}
Note that Equation \eqref{eq:concentrationThetaVar} implies
\begin{equation*}
\log\left(\mathfrak{l} \left(\bA; \widehat{\alpha}^{VAR}, \widehat{\bQ}^{VAR}\right) \right) - \log\left(\mathfrak{l}\left(\bA, z' ; \widehat{\alpha}^{VAR}, \widehat{\bQ}^{VAR}\right)\right) = o_p(1).
\end{equation*}
Now, using Pinsker's inequality, we see that 
\begin{equation*}
\left \vert \mathbb{P}_{\widehat{\tau}^{VAR}}(z') - \mathbb{P}\left(z' \vert \bA, \widehat{\alpha}^{VAR}, \widehat{\bQ}^{VAR}\right) \right\vert = o_p(1).
\end{equation*}
We use Equation \eqref{eq:concentrationThetaVar} and the definition of $\widehat{z}^{(VAR)}$ to conclude the proof of Equation \eqref{eq:concentrationzVar}.

\bigskip

\textbf{\underline{Proof of Equation \eqref{eq:concentrationzML}}}

Equation \eqref{eq:concentrationzML} is proven in \cite{Bickel21068}. In this work, the authors define the profile likelihood modularity $\mathcal{Q}_{LM}(A,z)$ of a label function $z \in \cZ_{n,k}$ as
\begin{equation*}
\mathcal{Q}_{LM}(A,z) = \frac{1}{2}\underset{a,b}{\sum}n_{ab}\left(\frac{\bO_{ab}}{n_{ab}}\log\left(\frac{\bO_{ab}}{n_{ab}}\right) + \left(1 - \frac{\bO_{ab}}{n_{ab}}\right)\log\left(1-\frac{\bO_{ab}}{n_{ab}}\right)\right).
\end{equation*}
for $\bO_{ab} = \underset{i \in z^{-1}(a), j \in z^{-1}(b)}{\sum}\bA_{ij}$ and $$n_{ab}  = \left\{
    \begin{array}{ll}
        \vert z^{-1}(a)\vert \times \vert z^{-1}(b)\vert & \mbox{ if } a \neq b \\
        \vert z^{-1}(a)\vert \times \left(\vert z^{-1}(a)\vert - 1 \right)  & \mbox{otherwise}
    \end{array}
\right.
$$ For $\hat{z}^{LM} = \argmax_{z\in\cZ_{n,k}}\mathcal{Q}_{LM}(A,z)$, the authors of \cite{Bickel21068} prove that under the assumptions of Proposition \ref{prp:variational}, with probability going to $1$, $\hat{z}^{LM} \sim z^*$. Since maximizing $\mathcal{Q}(A,z)$ is equivalent to maximizing $\max_{\bQ} \mathcal{L}\left(\bA; \bQ, z\right)$, this implies that $\widehat{z} \sim z^*$ with probability going to $1$.


\subsection{Proof of Proposition \ref{prp:boundvariational}}

To prove Proposition \ref{prp:boundvariational}, we show that the maximum likelihood estimator studied in Proposition \ref{prp:variational} is equal to the restricted maximum estimator studied in Corollary \ref{corPi1}. Proposition \ref{prp:boundvariational} then follows from Corollary \ref{corPi1}. 

Define $c_{min} = \min_{a,b}\bQ^0_{a,b}$ and $c_{max} = \max_{a,b}\bQ^0_{a,b}$. Corollary \ref{corPi1} implies that for some absolute constant $C>0$, 
\begin{equation*}
\begin{split}
\mathbb{P}\left(\left\Vert  \bTheta^* - \widehat{\bTheta}^{r} \right \Vert_2^2 \leq  C(c_{max}/c_{min})^2\rho_n\left(k^2 + n\log(k)\right) \right) \rightarrow 1,
\end{split}
\end{equation*}
where the restricted maximum likelihood estimator $\widehat{\bTheta}^{r}$ is defined as
\begin{equation*}
\begin{split}
&\widehat{\bTheta}^{r}_{i<j} = \widehat{\bQ}^{r}_{\widehat{z}^{r}(i) \widehat{z}^{r}(j)},\  \widehat{\bTheta}^{r}_{ii}=0\\
&(\widehat{\bQ}^{r},\widehat{z}^{r}) \in \underset{\bQ \in [c_{min} \rho_n /2, 2c_{max} \rho_n]^{k\times k}_{\rm sym}, z\in \cZ_{n,k}}{\argmin} \sum_{i\neq j}\cK(\bA_{ij},\bQ_{z(i)z(j)}).
\end{split}
\end{equation*}
One the other hand, Proposition \ref{prp:variational} implies that with probability going to one, the variational estimator of the probabilities of connections $\widehat{\bTheta}^{VAR}$ is equal to the maximum likelihood estimator $\widehat{\bTheta} $ given by 
\begin{equation*}
\begin{split}
&\widehat{\bTheta}_{i<j} = \widehat{\bQ}_{\widehat{z}(i) \widehat{z}(j)},\  \widehat{\bTheta}_{ii}=0\\
\text{ for }&(\widehat{\bQ}, \widehat{z}) \in \underset{\bQ \in \cQ, z\in \cZ_{n,k}}{\argmin} \sum_{i\neq j}\cK(\bA_{ij},\bQ_{z(i)z(j)}).
\end{split}
\end{equation*}
Thus, it is enough to show that $\widehat{\bTheta} = \widehat{\bTheta}^r$ with large probability to prove Proposition \ref{prp:boundvariational}. To do so, we show that with large probability, $\bQ(\widehat{z}) \in [c_{min}\rho_n/2, 2c_{max}\rho_n]^{k\times k}]$. We define
$$n_{ab}(z)  = \left\{
    \begin{array}{ll}
        \vert z^{-1}(a)\vert \times \vert z^{-1}(b)\vert & \mbox{ if } a \neq b \\
        \vert z^{-1}(a)\vert \times \left(\vert z^{-1}(a)\vert - 1 \right)  & \mbox{otherwise}
    \end{array}
\right.
$$
for $z \in \cZ_{n,k}$, and $\bQ(z) = \left(\bQ(z)_{ab}\right)$ such that $\bQ(z)_{ab} = \left(\underset{i \in z^{-1}(a), j \in z^{-1}(b)}{\sum}\bA_{ij}\right)/n_{ab}(z)$. With these notations, we note that $\widehat{\bQ} = \bQ(\widehat{z})$.

Recall that $\vert(z^*) ^{-1}(a)\vert$ is a sum of $n$ independent Bernoulli random variables with mean $\alpha^0_a$. Using Bernstein's inequality \ref{thm:Bernstein}, we find that for any $a$, 
$$\mathbb{P}\left(n\alpha^0_a - \vert(z^*) ^{-1}(a)\vert \geq 0.5n \alpha^0_{a} \right) \leq 2e^{-n\alpha^0_{a}/16}.$$
Thus, 
$$\mathbb{P}\left(\min_{a} \vert(z^*) ^{-1}(a)\vert  \leq 0.5n \min_{a} \alpha^0_{a} \right) \leq 2ke^{-n\min_{a} \alpha^0_a/16}.$$
Therefore, the event $\Omega = \left\{ \min_{a,b} n_{a,b}(z^*) \geq n^2 \min_a (\alpha^0_a)^2/5 \right\}$ holds with probability going to $1$. 

Now, we show that on the event $\Omega$, with large probability, $\bQ(z^*) \in [c_{min} \rho_n /2, 2c_{max} \rho_n]^{k\times k}$. Recall that for any $a,b$, conditionally on $z^*$, $n_{ab}(z^*)\bQ(z^*)_{ab}$ is a sum of $n_{ab}(z^*)$ independent Bernoulli random variables with mean $\rho_n \bQ^0_{ab}$. Then, Bernstein's inequality \ref{thm:Bernstein} implies that for any $t>0$
$$ \mathbb{P}\left(\left \vert n_{ab}(z^*)\bQ(z^*)_{ab}  - n_{ab}(z^*)\rho_n\bQ^0_{ab}\right \vert \geq \sqrt{2tn_{ab}(z^*)\rho_n\bQ^0_{ab}} + \frac{2t}{3}\right) \leq 2e^{-t}.$$
Choosing $t = n_{ab}(z^*)\rho_n\bQ^0_{ab}/16$ yields
$$\mathbb{P}\left(\left \vert n_{ab}(z^*)\bQ(z^*)_{ab}  - n_{ab}(z^*)\rho_n\bQ^0_{ab}\right \vert \geq 0.5n_{ab}(z^*)\rho_n\bQ^0_{ab}\right) \leq 2e^{-n_{ab}(z^*)\rho_n\bQ^0_{ab}/16}.$$
On the event $\Omega$, this implies that 
$$\mathbb{P}\left(\left \vert n_{ab}(z^*)\bQ(z^*)_{ab}  - n_{ab}(z^*)\rho_n\bQ^0_{ab}\right \vert \geq 0.5n_{ab}(z^*)\rho_n\bQ^0_{ab}\right) \leq 2e^{-n^2\rho_n\bQ^0_{ab}(\min_a\alpha_a^0)^2/80}.$$
A union bound yields
$$\mathbb{P}\left(\bQ(z^*) \notin [c_{min} \rho_n /2, 2c_{max} \rho_n]^{k\times k} \right) \leq 2k^2e^{-n^2\rho_n\min_{a,b}\bQ^0_{ab}(\min_a\alpha_a^0)^2/80}$$
on the event $\Omega$. Since $\mathbb{P}\left(\Omega\right) \rightarrow 1$ and $n^2\rho_n \rightarrow +\infty$, this shows that $$\mathbb{P}\left(\bQ(z^*) \in [c_{min} \rho_n /2, 2c_{max} \rho_n]^{k\times k} \right) \rightarrow 1.$$
Now, Equation \eqref{eq:concentrationzML} shows that with probability going to $1$, $\widehat{z} \sim z^*$. Thus, $Q(\widehat{z}) \in [c_{min} \rho_n /2, 2c_{max} \rho_n]^{k\times k}$ with probability going to one, and the maximum likelihood estimator of the probabilities of connections between nodes coincides with the restricted maximum likelihood estimator. This concludes the proof of Proposition \ref{prp:boundvariational}.


\bibliographystyle{apalike}
\bibliography{ref_MLE}

\begin{thebibliography}{}

\bibitem[{Abbe} and {Sandon}, 2015]{Abbe2015CommunityDI}
{Abbe}, E. and {Sandon}, C. (2015).
\newblock Community detection in general stochastic block models: Fundamental
  limits and efficient algorithms for recovery.
\newblock In {\em 2015 IEEE 56th Annual Symposium on Foundations of Computer
  Science}, pages 670--688.

\bibitem[Agarwal et~al., 2017]{SPDSBM}
Agarwal, N., Bandeira, A.~S., Koiliaris, K., and Kolla, A. (2017).
\newblock {\em Multisection in the Stochastic Block Model Using Semidefinite
  Programming}, pages 125--162.
\newblock Springer International Publishing, Cham.

\bibitem[Albert and Barab\'asi, 2002]{RevModPhys.74.47}
Albert, R. and Barab\'asi, A.-L. (2002).
\newblock Statistical mechanics of complex networks.
\newblock {\em Reviews of Modern Physics}, 74:47--97.

\bibitem[Amini et~al., 2013]{Bickl}
Amini, A.~A., Chen, A., Bickel, P.~J., and Levina, E. (2013).
\newblock {Pseudo-likelihood methods for community detection in large sparse
  networks}.
\newblock {\em The Annals of Statistics}, 41(4):2097 -- 2122.

\bibitem[Amini and Levina, 2018]{amini2018}
Amini, A.~A. and Levina, E. (2018).
\newblock {On semidefinite relaxations for the block model}.
\newblock {\em The Annals of Statistics}, 46(1):149 -- 179.

\bibitem[Bandeira, 2015]{Bandeira2018}
Bandeira, A. (2015).
\newblock Random laplacian matrices and convex relaxations.
\newblock {\em Foundations of Computational Mathematics}, 18.

\bibitem[{Benyahia} et~al., 2017]{Benyahia2017CommunityDI}
{Benyahia}, O., {Largeron}, C., and {Jeudy}, B. (2017).
\newblock Community detection in dynamic graphs with missing edges.
\newblock In {\em 2017 11th International Conference on Research Challenges in
  Information Science (RCIS)}, pages 372--381.

\bibitem[Bickel et~al., 2013]{LikelihoodBickel}
Bickel, P., Choi, D., Chang, X., and Zhang, H. (2013).
\newblock {Asymptotic normality of maximum likelihood and its variational
  approximation for stochastic blockmodels}.
\newblock {\em The Annals of Statistics}, 41(4):1922 -- 1943.

\bibitem[Bickel and Chen, 2009]{Bickel21068}
Bickel, P.~J. and Chen, A. (2009).
\newblock A nonparametric view of network models and newman{\textendash}girvan
  and other modularities.
\newblock {\em Proceedings of the National Academy of Sciences},
  106(50):21068--21073.

\bibitem[Bleakley et~al., 2007]{BleakleyMissingProteins}
Bleakley, K., Biau, G., and Vert, J.-P. (2007).
\newblock {Supervised reconstruction of biological networks with local models}.
\newblock {\em Bioinformatics}, 23(13):i57--i65.

\bibitem[Bordenave et~al., 2018]{bordenave2018}
Bordenave, C., Lelarge, M., and Massoulié, L. (2018).
\newblock {Nonbacktracking spectrum of random graphs: Community detection and
  nonregular Ramanujan graphs}.
\newblock {\em The Annals of Probability}, 46(1):1 -- 71.

\bibitem[Cand\`{e}s and Recht, 2012]{Candes2009}
Cand\`{e}s, E. and Recht, B. (2012).
\newblock Exact matrix completion via convex optimization.
\newblock {\em Foundations of Computational Mathematics}, 55(6):111–119.

\bibitem[Celisse et~al., 2012]{celisse2012}
Celisse, A., Daudin, J.-J., and Pierre, L. (2012).
\newblock {Consistency of maximum-likelihood and variational estimators in the
  stochastic block model}.
\newblock {\em Electronic Journal of Statistics}, 6(none):1847 -- 1899.

\bibitem[Chatterjee, 2015]{chatterjee2015}
Chatterjee, S. (2015).
\newblock {Matrix estimation by Universal Singular Value Thresholding}.
\newblock {\em The Annals of Statistics}, 43(1):177 -- 214.

\bibitem[Chen and Lei, 2014]{CrossValK}
Chen, K. and Lei, J. (2014).
\newblock Network cross-validation for determining the number of communities in
  network data.
\newblock {\em Journal of the American Statistical Association}, 113:241 --
  251.

\bibitem[Clauset et~al., 2008]{Clauset2008HierarchicalSA}
Clauset, A., Moore, C., and Newman, M. (2008).
\newblock Hierarchical structure and the prediction of missing links in
  networks.
\newblock {\em Nature}, 453:98--101.

\bibitem[Daudin et~al., 2008]{VarEst}
Daudin, J.-J., Picard, F., and Robin, S. (2008).
\newblock A mixture model for random graph.
\newblock {\em Statistics and Computing}, 18:173--183.

\bibitem[Davenport et~al., 2014]{Davenport1Bit}
Davenport, M.~A., Plan, Y., van~den Berg, E., and Wootters, M. (2014).
\newblock {1-Bit matrix completion}.
\newblock {\em Information and Inference: A Journal of the IMA}, 3(3):189--223.

\bibitem[Decelle et~al., 2011]{Lenka}
Decelle, A., Krzakala, F., Moore, C., and Zdeborov\'a, L. (2011).
\newblock Asymptotic analysis of the stochastic block model for modular
  networks and its algorithmic applications.
\newblock {\em Physical review. E}, 84:066106.

\bibitem[Gao et~al., 2016]{2015gaoBiclustering}
Gao, C., Lu, Y., Ma, Z., and Zhou, H.~H. (2016).
\newblock Optimal estimation and completion of matrices with biclustering
  structures.
\newblock {\em Journal of Machine Learning Research}, 17(1):5602–5630.

\bibitem[Gao et~al., 2015]{gao2015optimal}
Gao, C., Lu, Y., and Zhou, H.~H. (2015).
\newblock {Rate-optimal graphon estimation}.
\newblock {\em The Annals of Statistics}, 43(6):2624 -- 2652.

\bibitem[Giné and Nickl, 2015]{gine_nickl_2015}
Giné, E. and Nickl, R. (2015).
\newblock {\em Mathematical Foundations of Infinite-Dimensional Statistical
  Models}.
\newblock Cambridge Series in Statistical and Probabilistic Mathematics.
  Cambridge University Press.

\bibitem[Giraud and Verzelen, 2018]{Giraud2018PartialRB}
Giraud, C. and Verzelen, N. (2018).
\newblock {Partial recovery bounds for clustering with the relaxed $K$-means}.
\newblock {\em {Mathematical Statistics and Learning}}, 1(3):317--374.

\bibitem[Guimer{\`a} and Sales-Pardo, 2009]{GuimerMissing}
Guimer{\`a}, R. and Sales-Pardo, M. (2009).
\newblock Missing and spurious interactions and the reconstruction of complex
  networks.
\newblock {\em Proceedings of the National Academy of Sciences},
  106(52):22073--22078.

\bibitem[{Hagen} and {Kahng}, 1992]{HagenSpecClus}
{Hagen}, L. and {Kahng}, A.~B. (1992).
\newblock New spectral methods for ratio cut partitioning and clustering.
\newblock {\em IEEE Transactions on Computer-Aided Design of Integrated
  Circuits and Systems}, 11(9):1074--1085.

\bibitem[{Hajek} et~al., 2016]{HajekSPD}
{Hajek}, B., {Wu}, Y., and {Xu}, J. (2016).
\newblock Achieving exact cluster recovery threshold via semidefinite
  programming: Extensions.
\newblock {\em IEEE Transactions on Information Theory}, 62(10):5918--5937.

\bibitem[Handcock and Gile, 2010]{handcock2010}
Handcock, M.~S. and Gile, K.~J. (2010).
\newblock Modeling social networks from sampled data.
\newblock {\em The Annals of Applied Statistics}, 4(1).

\bibitem[Klopp, 2014]{KloppLowRank}
Klopp, O. (2014).
\newblock {Noisy low-rank matrix completion with general sampling
  distribution}.
\newblock {\em Bernoulli}, 20(1):282 -- 303.

\bibitem[Klopp et~al., 2015]{KloppMultinomial}
Klopp, O., Lafond, J., Moulines, E., and Salmon, J. (2015).
\newblock {Adaptive multinomial matrix completion}.
\newblock {\em Electronic Journal of Statistics}, 9(2):2950 -- 2975.

\bibitem[Klopp et~al., 2017]{KloppGraphon}
Klopp, O., Tsybakov, A.~B., and Verzelen, N. (2017).
\newblock {Oracle inequalities for network models and sparse graphon
  estimation}.
\newblock {\em The Annals of Statistics}, 45(1):316 -- 354.

\bibitem[Klopp and Verzelen, 2017]{KloppCut}
Klopp, O. and Verzelen, N. (2017).
\newblock {Optimal graphon estimation in cut distance}.
\newblock Working Papers 2017-42, Center for Research in Economics and
  Statistics.

\bibitem[Koltchinskii et~al., 2011]{koltchinskii2011}
Koltchinskii, V., Lounici, K., and Tsybakov, A.~B. (2011).
\newblock {Nuclear-norm penalization and optimal rates for noisy low-rank
  matrix completion}.
\newblock {\em The Annals of Statistics}, 39(5):2302 -- 2329.

\bibitem[Kossinets, 2006]{KOSSINETS2006247}
Kossinets, G. (2006).
\newblock Effects of missing data in social networks.
\newblock {\em Social Networks}, 28(3):247--268.

\bibitem[Kshirsagar et~al., 2012]{Kshirsagar2012TechniquesTC}
Kshirsagar, M., Carbonell, J., and Klein-Seetharaman, J. (2012).
\newblock Techniques to cope with missing data in host–pathogen protein
  interaction prediction.
\newblock {\em Bioinformatics}, 28(18):i466–i472.

\bibitem[Latouche et~al., 2011]{LatoucheBlogsphere}
Latouche, P., Birmelé, E., and Ambroise, C. (2011).
\newblock Overlapping stochastic block models with application to the french
  political blogosphere.
\newblock {\em The Annals of Applied Statistics}, 5(1):309--336.

\bibitem[Leger et~al., 2014]{Leger2014}
Leger, J.-B., Vacher, C., and Daudin, J.-J. (2014).
\newblock Detection of structurally homogeneous subsets in graphs.
\newblock {\em Statistics and Computing}, 24(5):675–692.

\bibitem[Lei, 2016]{lei2016}
Lei, J. (2016).
\newblock {A goodness-of-fit test for stochastic block models}.
\newblock {\em The Annals of Statistics}, 44(1):401 -- 424.

\bibitem[Lov{\'a}sz, 2012]{LovaszBook}
Lov{\'a}sz, L. (2012).
\newblock {\em Large Networks and Graph Limits}.
\newblock American Mathematical Society colloquium publications. American
  Mathematical Society.

\bibitem[Lü and Zhou, 2011]{Lu2010LinkPI}
Lü, L. and Zhou, T. (2011).
\newblock Link prediction in complex networks: A survey.
\newblock {\em Physica A: Statistical Mechanics and its Applications},
  390(6):1150--1170.

\bibitem[Mariadassou and Tabouy, 2020]{TabThe}
Mariadassou, M. and Tabouy, T. (2020).
\newblock {Consistency and asymptotic normality of stochastic block models
  estimators from sampled data}.
\newblock {\em Electronic Journal of Statistics}, 14(2):3672 -- 3704.

\bibitem[Massouli\'{e}, 2014]{mass2013}
Massouli\'{e}, L. (2014).
\newblock Community detection thresholds and the weak ramanujan property.
\newblock In {\em Proceedings of the Forty-Sixth Annual ACM Symposium on Theory
  of Computing}, STOC '14, page 694–703, New York, NY, USA. Association for
  Computing Machinery.

\bibitem[{Matias, Catherine} and {Robin, St\'ephane}, 2014]{refId0}
{Matias, Catherine} and {Robin, St\'ephane} (2014).
\newblock Modeling heterogeneity in random graphs through latent space models:
  a selective review*.
\newblock {\em ESAIM: Proc.}, 47:55--74.

\bibitem[{McSherry}, 2001]{SpecClusMcSherry}
{McSherry}, F. (2001).
\newblock Spectral partitioning of random graphs.
\newblock In {\em Proceedings 42nd IEEE Symposium on Foundations of Computer
  Science}, pages 529--537.

\bibitem[Mossel et~al., 2016]{mossel2014consistency}
Mossel, E., Neeman, J., and Sly, A. (2016).
\newblock {Consistency thresholds for the planted bisection model}.
\newblock {\em Electronic Journal of Probability}, 21(none):1 -- 24.

\bibitem[Negahban and Wainwright, 2011]{negahban2011}
Negahban, S. and Wainwright, M.~J. (2011).
\newblock {Estimation of (near) low-rank matrices with noise and
  high-dimensional scaling}.
\newblock {\em The Annals of Statistics}, 39(2):1069 -- 1097.

\bibitem[Newman, 2006]{Newman8577}
Newman, M. E.~J. (2006).
\newblock Modularity and community structure in networks.
\newblock {\em Proceedings of the National Academy of Sciences},
  103(23):8577--8582.

\bibitem[Olhede and Wolfe, 2014]{Olhede14722}
Olhede, S.~C. and Wolfe, P.~J. (2014).
\newblock Network histograms and universality of blockmodel approximation.
\newblock {\em Proceedings of the National Academy of Sciences},
  111(41):14722--14727.

\bibitem[Picard et~al., 2009]{picard:hal-00391483}
Picard, F., Miele, V., Daudin, J., Cottret, L., and Robin, S. (2009).
\newblock {Deciphering the connectivity structure of biological networks using
  MixNet}.
\newblock {\em {BMC Bioinformatics}}, 10:1--11.

\bibitem[Rohe et~al., 2011]{rohe2011}
Rohe, K., Chatterjee, S., and Yu, B. (2011).
\newblock {Spectral clustering and the high-dimensional stochastic blockmodel}.
\newblock {\em The Annals of Statistics}, 39(4):1878 -- 1915.

\bibitem[Tabouy et~al., 2020]{TabPra}
Tabouy, T., Barbillon, P., and Chiquet, J. (2020).
\newblock Variational inference for stochastic block models from sampled data.
\newblock {\em Journal of the American Statistical Association},
  115(529):455--466.

\bibitem[Vershynin, 2012]{vershynin}
Vershynin, R. (2012).
\newblock {\em Introduction to the non-asymptotic analysis of random matrices},
  page 210–268.
\newblock Cambridge University Press.

\bibitem[Wang and Bickel, 2017]{BickelMS}
Wang, Y. X.~R. and Bickel, P.~J. (2017).
\newblock {Likelihood-based model selection for stochastic block models}.
\newblock {\em The Annals of Statistics}, 45(2):500 -- 528.

\bibitem[Wasserman and Faust, 1994]{WassermanSocio}
Wasserman, S. and Faust, K. (1994).
\newblock {\em Social Network Analysis: Methods and Applications}.
\newblock Structural Analysis in the Social Sciences. Cambridge University
  Press.

\bibitem[Xu, 2018]{USVTXu}
Xu, J. (2018).
\newblock Rates of convergence of spectral methods for graphon estimation.
\newblock In Dy, J. and Krause, A., editors, {\em Proceedings of the 35th
  International Conference on Machine Learning}, volume~80 of {\em Proceedings
  of Machine Learning Research}, pages 5433--5442. PMLR.

\bibitem[Yamanishi et~al., 2004]{yamanishi:hal-00433586}
Yamanishi, Y., Vert, J.-P., and Kanehisa, M. (2004).
\newblock {Protein network inference from multiple genomic data: a supervised
  approach.}
\newblock {\em {Bioinformatics}}, 20 Suppl 1:i363--70.

\bibitem[Yan and Gregory, 2012]{Yan2012FindingME}
Yan, B. and Gregory, S. (2012).
\newblock Finding missing edges in networks based on their community structure.
\newblock {\em Physical review. E}, 85:056112.

\bibitem[Zhang et~al., 2017]{Zhang2017EstimatingNE}
Zhang, Y., Levina, E., and Zhu, J. (2017).
\newblock {Estimating network edge probabilities by neighbourhood smoothing}.
\newblock {\em Biometrika}, 104(4):771--783.

\bibitem[Zhao et~al., 2017]{LevinaLinkPred}
Zhao, Y., Wu, Y.-J., Levina, E., and Zhu, J. (2017).
\newblock Link prediction for partially observed networks.
\newblock {\em Journal of Computational and Graphical Statistics},
  26(3):725--733.

\end{thebibliography}

\end{document}